\documentclass[11pt,letterpaper]{amsart}

\usepackage[T1]{fontenc}
\usepackage[utf8]{inputenc}
\usepackage{lmodern}
\usepackage[english]{babel}
\usepackage[autostyle]{csquotes}

\usepackage[hidelinks]{hyperref}
\usepackage{amsthm}
\usepackage{amsmath}
\usepackage{amsfonts}
\usepackage{mathtools}
\usepackage{graphicx}
\usepackage{mathrsfs}
\usepackage{amssymb}
\usepackage{centernot}
\usepackage{nicematrix}
\usepackage{enumerate}
\usepackage{colonequals}
\usepackage{mathbbol}
\usepackage[all]{xy}
\usepackage{tikz}
\usepackage{tikz-cd}
\usepackage[mathcal]{euscript}
\usepackage{comment}
\usepackage{mathdots}
\usepackage{aliascnt}
\usepackage{cleveref}

\usepackage[backend=biber,doi=false,isbn=false,url=false,eprint=false, style=trad-alpha]{biblatex} 

\usepackage{dynkin-diagrams}

\numberwithin{equation}{subsection}

\newtheorem{theorem}{Theorem}[section]
\newtheorem{theoremintro}{Theorem}

\newaliascnt{corollary}{theorem}
\newtheorem{corollary}[corollary]{Corollary}
\aliascntresetthe{corollary}
\newaliascnt{lemma}{theorem}
\newtheorem{lemma}[lemma]{Lemma}
\aliascntresetthe{lemma}
\newaliascnt{claim}{theorem}

\aliascntresetthe{claim}
\newaliascnt{proposition}{theorem}
\newtheorem{proposition}[proposition]{Proposition}
\aliascntresetthe{proposition}
\newaliascnt{conjecture}{theorem}

\aliascntresetthe{conjecture}
\newaliascnt{definition}{theorem}
\newtheorem{definition}[definition]{Definition}
\aliascntresetthe{definition}

\newaliascnt{assumption}{theorem}

\aliascntresetthe{assumption}

\theoremstyle{remark}

\theoremstyle{definition}

\newaliascnt{remark}{theorem}
\newtheorem{remark}[remark]{Remark}
\aliascntresetthe{remark}
\newaliascnt{example}{theorem}
\newtheorem{example}[example]{Example}
\aliascntresetthe{example}
\newaliascnt{exercise}{theorem}

\aliascntresetthe{exercise}
\newaliascnt{notation}{theorem}

\aliascntresetthe{notation}
\newaliascnt{question}{theorem}

\aliascntresetthe{question}
\newaliascnt{slogan}{theorem}

\aliascntresetthe{slogan}
\newaliascnt{warning}{theorem}

\aliascntresetthe{warning}

\def\C{\mathbb C}
\def\Q{\mathbb Q}
\def\Z{\mathbb Z}

\def\Fp{\mathbb{F}_p}
\def\R{\mathbb R}

\def\ul#1{\underline{#1}}

\def\mc#1{\mathcal{#1}}
\def\mb#1{\mathbb{#1}}

\def\mr#1{\mathrm{#1}}

\def\sch{\ul{\mathrm{Sch}}}

\renewcommand{\hom}{\mr{Hom}}
\DeclareMathOperator{\ord}{ord}

\DeclareMathOperator{\Span}{Span}

\DeclareMathOperator{\tor}{tor}
\DeclareMathOperator{\spec}{Spec}

\DeclareMathOperator{\rk}{rk}
\DeclareMathOperator{\sub}{sub}

\newcommand{\Asf}{\mathsf{A}}
\newcommand{\Bsf}{\mathsf{B}}
\newcommand{\Csf}{\mathsf{C}}
\newcommand{\Dsf}{\mathsf{D}}
\newcommand{\Esf}{\mathsf{E}}

\newcommand{\Bcal}{{\mathcal B}}

\newcommand{\Dcal}{{\mathcal D}}
\newcommand{\Ecal}{{\mathcal E}}
\newcommand{\Fcal}{{\mathcal F}}

\newcommand{\Hcal}{{\mathcal H}}
\newcommand{\Ical}{{\mathcal I}}

\newcommand{\Lcal}{{\mathcal L}}

\newcommand{\Ocal}{{\mathcal O}}
\newcommand{\Pcal}{{\mathcal P}}

\newcommand{\Scal}{{\mathcal S}}

\newcommand{\Ucal}{{\mathcal U}}
\newcommand{\Vcal}{{\mathcal V}}

\newcommand{\Xcal}{{\mathcal X}}
\newcommand{\Ycal}{{\mathcal Y}}
\newcommand{\Zcal}{{\mathcal Z}}

\newcommand{\Lscr}{{\mathscr L}}

\renewcommand{\AA}{\mathbb{A}}

\newcommand{\FF}{\mathbb{F}}
\newcommand{\GG}{\mathbb{G}}

\newcommand{\QQ}{\mathbb{Q}}
\newcommand{\RR}{\mathbb{R}}

\newcommand{\ZZ}{\mathbb{Z}}

\newcommand{\Sfr}{{\mathfrak S}}

\newcommand{\FZip}{\mathop{\text{$F$-{\tt Zip}}}\nolimits}
\newcommand{\GZip}{\mathop{\text{$G$-{\tt Zip}}}\nolimits}
\newcommand{\GprimeZip}{\mathop{\text{$G'$-{\tt Zip}}}\nolimits}

\newcommand{\MZip}{\mathop{\text{$M$-{\tt Zip}}}\nolimits}
\newcommand{\GF}{\mathop{\text{$G$-{\tt ZipFlag}}}\nolimits}
\newcommand{\FFZip}{\mathop{\text{$F$-{\tt ZipFlag}}}\nolimits}
\newcommand{\IW}{{}^IW}

\DeclareMathOperator{\Res}{Res}
\DeclareMathOperator{\divi}{div}
\DeclareMathOperator{\Sbt}{Sbt}
\DeclareMathOperator{\Ha}{Ha}
\DeclareMathOperator{\U}{U}
\DeclareMathOperator{\GL}{GL}
\DeclareMathOperator{\SL}{SL}

\DeclareMathOperator{\Spec}{Spec}
\DeclareMathOperator{\pr}{pr}
\DeclareMathOperator{\Gal}{Gal}

\DeclareMathOperator{\Stab}{Stab}

\DeclareMathOperator{\ad}{ad}
\DeclareMathOperator{\GSp}{GSp}
\DeclareMathOperator{\Sp}{Sp}
\DeclareMathOperator{\PGL}{PGL}
\DeclareMathOperator{\diag}{diag}

\DeclareMathOperator{\id}{id}

\DeclareMathOperator{\supp}{supp}

\DeclareMathOperator{\fil}{Fil}
\DeclareMathOperator{\Hdg}{Hdg}
\DeclareMathOperator{\conj}{conj}
\DeclareMathOperator{\dr}{dR}
\DeclareMathOperator{\Br}{Br}

\newcommand{\relmiddle}[1]{\mathrel{}\middle#1\mathrel{}}
\newcommand{\loccit}{{\em loc.\@ cit.\@ }}

\title{Singularities in the Ekedahl--Oort stratification}

\author{Jean-Stefan Koskivirta}
\address{Department of Mathematics, Université de Caen-Normandie}
\email{jeanstefan.koskivirta@gmail.com}

\author{Lorenzo La Porta}
\address{Department of Mathematics, Universit\`{a} degli Studi di Padova}
\email{lorenzo.laporta@protonmail.com}

\author{Stefan Reppen}
\address{Department of Mathematics, University of California, Berkeley}
\email{stefan.reppen@gmail.com}

\begin{document}

\begin{abstract}
We give conceptual and combinatorial criteria for the normality and Cohen--Macaulayness of unions of Ekedahl--Oort strata in the special fiber of abelian type Shimura varieties. For unions of two strata, one of the two having codimension one in the closure of the other, we determine exactly when their union is smooth. We provide explicit numerical criteria for the smoothness of any one-dimensional EO-stratum closure. For groups of type $\mathsf{B}_n$, we describe the smooth and normal loci of all EO-strata closures. We construct reduced strata Hasse invariants of explicit weights on EO-strata of codimension at most $n-1$, showing that their Zariski closures are local complete intersections. We also provide a closed form formula for their cycle classes.
\end{abstract}

\maketitle

\setcounter{tocdepth}{1} 

\tableofcontents 

\renewcommand{\thesection}{I}
\renewcommand{\thesubsection}{\thesection.\arabic{subsection}}
\renewcommand{\thesubsubsection}{\thesubsection.\arabic{subsubsection}}

\section*{Introduction}
\label{sec-intro}

\setcounter{subsection}{0}

\subsection{Motivation}\label{section: intro motivation}
This paper studies singularities in the Ekedahl--Oort (EO) stratification of the special fiber of abelian type Shimura varieties at primes of good reduction. 
This stratification was first defined by Oort in \cite{oortastrat} for the moduli space of principally polarized abelian varieties (i.e.\@ Siegel Shimura varieties) by declaring that two points lie in the same stratum if the $p$-torsions of the corresponding abelian varieties are isomorphic. This construction was extended to abelian type Shimura varieties in a series of papers \cite{vwPELstrat, zhang,shen.zhang}, with a gradual shift from the language of $p$-torsion to that of $G$-zips (\S\ref{sec-revgzip}).

The topology and geometry of EO-strata have been studied in increasing levels of generality. For example, each EO-stratum is smooth, affine in the minimal compactification, and its dimension is given by the length of the Weyl group element that parametrizes it. The closure relations between different strata are also known, by \cite{Pink-Wedhorn-Ziegler-F-Zips-additional-structure}, and certain connectedness results were shown in \cite{wedhorn.ziegler.tautological.rings}.

Although EO-strata are well studied, little is known about the geometry of their Zariski closures. This paper studies geometric properties, such as Cohen--Macaulayness, normality, and smoothness, of these subschemes. Questions about singularities are among the most natural and fundamental in algebraic geometry. However, in this setting, such questions are also intimately related to important arithmetic information, in many different ways.
\begin{enumerate}
    \item\label{item: intro motivation -  hasse invariants} The existence and order of vanishing of Hasse invariants are reflected in the geometric properties of EO-strata closures. We give a few examples of this. 
    \begin{enumerate}
        \item When appropriate families of Hasse invariants (e.g.\@ partial Hasse invariants) with reduced zeroscheme exist, we can use them to show that certain EO-strata closures are Cohen--Macaulay, or even local complete intersections, as in \S\ref{sec-orthogonal}. 
        \item By Serre's criterion for normality, the previous point shows that existence of reduced strata Hasse invariants is linked to the normality of EO-strata closures, see \Cref{thm-normchargen}. A converse relation also holds, as normality of an EO-stratum closure allows one to construct Hasse invariants by doing so on open subsets whose complements have large codimension (see \cite{wedhorn.koskivirta, koskgold}, \S\ref{sec-orthogonal}). 
        \item The vanishing order of a reduced strata Hasse invariant can describe the singular locus of the stratum closure. For example, the vanishing order of the classical Hasse invariant was computed in \cite{srhasseinv,goldring2024ogusprinciplezipperiod} for most Hodge-type Shimura varieties whose Hodge character is defined over $\mathbb{F}_p$, in which case the vanishing order describes the singular locus of the codimension one EO-stratum closure. The order of vanishing of a non-ordinary Hasse invariant for unitary Shimura varieties is computed in work in progress by the first and third named authors, joint with Goldring \cite{Koskivirta-Reppen-unitary}. The vanishing order of the classical Hasse invariant has also been proved to relate to $F$-pure thresholds for Calabi--Yau varieties \cite{bhatt.singh}.
    \end{enumerate}
    \item\label{item: intro motivation - canonical torsor} The geometry of the closure of an EO-stratum is also related to the existence of extensions of the canonical filtration from the stratum to other strata in its closure, see for instance \cite{boxer}. We extend this relation and prove in full generality that the existence of extensions implies normality and Cohen--Macaulayness of the EO-stratum closure, see \Cref{thm-extnorm}.
    \item\label{item: intro motivation - deformation} When the Shimura variety is a moduli space of abelian varieties with extra structure, the geometry of unions of EO-strata is also connected to the deformation theory of the underlying abelian varieties. This is an essential ingredient in the construction and study of theta operators, see, for instance \cite{katzresult, EFGMM, eiscmant212, laporta2023generalised}.
    \item\label{item: intro motivation - cycle classes} As generating objects for the tautological ring inside the Chow ring, the geometry of EO-strata closures is related to the intersection theory of Shimura varieties. In the simple $\Bsf_n$-case we prove that EO-strata closures of codimension $\leq n-1$ are local complete intersections, and deduce a closed form formula for their cycle classes.
\end{enumerate}

\subsection{A heuristic: relation to the Bruhat stratification via flag spaces}\label{section: main heuristic}
Let $S$ be the geometric special fiber of a Hodge-type Shimura variety at a prime of good reduction and let $G$ be the $\FF_p$-fiber of the $\mathbb{Q}$-group defining $S$. Let $S=\cup_{w \in {}^{I}W} S_w$ denote the Ekedahl--Oort stratification; it is parametrized by a certain subset ${}^{I}W$ of the Weyl group of $G$ (\S\ref{sec-param}). It arises from a morphism of stacks $\zeta\colon S\to \GZip^\mu$, called the \emph{zip period map}, where $\mu$ is the Hodge cocharacter of $G$ deduced from the Shimura datum and $\GZip^\mu$ is the stack of $G$-zips of Pink--Wedhorn--Ziegler (see \ref{section: intro g-zips to shimura varieties} for details). Forgetting part of this structure, we obtain a smooth surjective morphism
\[
\beta\colon S\to \mathcal{B}\coloneqq [(P\times Q)\backslash G]
\]
where $P\subseteq G$ is the Hodge parabolic subgroup and $Q$ is the Frobenius twist of the opposite parabolic of $P$ (\S\ref{sec-zip-data}). The classical Bruhat stratification yields a stratification on $\mathcal{B}$ (\S\ref{sec-Bruhat}). The closures of Bruhat strata are normal and Cohen--Macaulay. Hence, if $S_w$ is open in the preimage of a Bruhat stratum, then $\overline{S}_w$ is also normal and Cohen--Macaulay. Unfortunately, very few EO-strata are open in a Bruhat stratum. For example, if $S$ is of simple $\Bsf_n$-type, then there are only three such EO-strata while there are $2n$ many EO-strata in total.

Nonetheless, for a fixed Borel $B$ and each intermediate parabolic $B\subseteq P_0\subseteq P$, there is a $P/P_0$-fibration 
\[
\pi_{P_0}\colon S^{(P_0)}\to S
\]
over $S$ \cite{ekedahl.geer.cycles.classes, Goldring-Koskivirta-zip-flags} (\S\ref{ssecflagspace}). The scheme $S^{(P_0)}$ also admits an ``EO-like'' stratification, 
\[S^{(P_0)}=\cup_{w\in {}^{I_0}W}S_w^{(P_0)},\] where ${}^{I_0}W\subseteq W$ is a certain subset containing ${}^{I}W$. Analogously to $\beta$, there is a smooth surjective map
\[
\beta^{(P_0)}\colon S^{(P_0)}\to \mathcal{B}^{(P_0)}\coloneqq [(P_0\times Q_0)\backslash G],
\]
where $Q_0\subseteq Q$ is a certain parabolic subgroup constructed from $P_0$. Again, if $S_w^{(P_0)}$ is open in the preimage of a Bruhat stratum under $\beta^{(P_0)}$, then $\overline{S}_w^{(P_0)}$ is normal and Cohen--Macaulay. Hence, if the fibration $\pi_{P_0}$ takes $S_w^{(P_0)}$ to $S_w$ and is ``nice enough'', then we can deduce that $\overline{S}_w$ is normal and Cohen--Macaulay. 

So far we have let $P_0$ be arbitrary. However, associated to $w\in {}^{I}W$ there is a unique standard parabolic $P_w$, called \emph{canonical}, (see \ref{section intro: bounded separating cover} below for details) such that
\begin{enumerate}
    \item the stratum $S_w^{(P_w)}$ coincides with a Bruhat stratum of $S^{(P_w)}$, and
    \item \label{Pw-item2} the morphism $\pi_{P_w}$ restricts to an isomorphism $\pi_{P_w}\colon S_w^{(P_w)}\to S_w$.
\end{enumerate}

Let $s_w\colon S_w\to S_{w}^{(P_w)}$ be the inverse of the above isomorphism of \eqref{Pw-item2}. This is called the \textit{canonical section of $w$} (\Cref{def-cantorcansec}). As $\pi_{P_w}$ is a $P/P_w$-fibration, $s_w$ corresponds to a canonical $P_w$-torsor $\mathcal{I}_w$ over $S_w$. These canonical objects relate to normality and Cohen--Macaulayness of $\overline{S}_w$. Indeed, if $s_w$ extends to $\overline{S}_w$, then it induces an isomorphism $\overline{S}_w \simeq \overline{S}_w^{(P_w)}$, which is normal and Cohen--Macaulay.

Since the parametrization of the EO-stratification is combinatorially encoded by the Weyl group, it would also be desirable to have a combinatorial criterion for $\overline{S}_w$ to be normal and Cohen--Macaulay. The two notions that achieve this are the following. Let $U\subseteq \overline{S}_w$ be an open subscheme of $\overline{S}_w$ that is a union of EO-strata, say $U=\cup_{w' \in \Gamma_U}S_{w'}$ for some $\Gamma_U\subseteq {}^IW$.
\begin{itemize}
    \item Firstly, we say that $U$ is \emph{$w$-bounded} if $P_{w'}\subseteq P_w$ for all $w'\in \Gamma_U$. Roughly speaking, this condition is saying that there is no obvious obstruction preventing us from extending the canonical filtration, or equivalently \(s_w\), from \(S_w\) to \(S_{w'}\) in a way that is compatible with the canonical filtration on \(S_{w'}\). See \Cref{lem-rigidityofsubtorsors}.
    \item Secondly, we say that $U$ \emph{admits a separating canonical cover} if the preimage of $U$ via the projection $\pi_{P_w}\colon \overline{S}_w^{(P_w)}\to \overline{S}_w$ is $U^{(P_w)}\coloneqq \cup_{w'\in \Gamma_U}S_{w'}^{(P_w)}$. This tells us that the restriction \(\pi_{P_w} \colon U^{(P_w)} \to U\) has nice geometric properties. See \Cref{prop-proper}.
\end{itemize} 

\subsection{Main results for Shimura varieties} \label{section: main results for Shimura}
For the reader familiar with the theory of $G$-zips, we point out that all the results are first obtained on the level of $G$-zips, and then deduced for Shimura varieties via the zip period map \eqref{eq-zeta-shimura} (see also \S\ref{section: applications via zip period maps}). To keep the introduction more accessible, we only state the results for Shimura varieties and add internal references to the corresponding statements for stacks of $G$-zips. However, we emphasize that the $G$-zip version of Theorem \ref{theorem: thmA intro} applies to arbitrary cocharacter data $(G,\mu)$. In particular, $\mu$ need not be minuscule. Also, since the map $\zeta\colon S\to \GZip^{\mu}$ extends smoothly to a smooth toroidal compactification $S^{\tor}$ of $S$ (\cite{Goldring-Koskivirta-zip-flags,Andreatta-modp-period-maps}), all of our results also hold for the EO-strata of $S^{\tor}$. All objects introduced above (e.g.\ the section $s_w$, the $P_w$-torsor $\Ical_w$) already exist on the stack of $G$-zips and arise by pullback via the map $\zeta \colon S \to \GZip^\mu$. More generally, we say that an object is \emph{zip-theoretic} if it arises by pullback in this fashion.
\begin{theoremintro}[{\Cref{thm-extnorm}}]\label{theorem: thmA intro}
Let $S$ be the special fiber of an abelian type Shimura variety with good reduction. Let $S_w$ be an arbitrary Ekedahl--Oort stratum of $S$ and 
let $U$ be an open subscheme of $\overline{S}_w$ that is a union of Ekedahl--Oort strata (e.g.\@ $U=\overline{S}_w$). Then, the following are equivalent.
\begin{enumerate}
    \item\label{item: thmA intro sw and Iw extends} The section $s_w$ (equivalently the torsor $\Ical_w$) admits a zip-theoretic extension to $U$. 
    \item\label{item: thmA intro wbounded and sep can cov} The scheme $U$ is $w$-bounded and admits a separating canonical cover.
    \item\label{item: thmA intro normal and sep can cov} The scheme $U$ is normal and admits a separating canonical cover.
\end{enumerate}
In this case $U$ is Cohen--Macaulay and the extensions of $s_w$ and $\Ical_w$ are unique. 
\end{theoremintro}

We say that $U\subseteq \overline{S}_w$ is \emph{elementary} if $U=S_w\cup S_{w'}$, where $S_{w'}$ is an EO-stratum contained in $\overline{S}_w$ with codimension one. For elementary $U$, we show in \Cref{thm-elemsmoothchar} that the conditions of Theorem \ref{theorem: thmA intro} are equivalent to the smoothness of $U$ (equivalently, its normality), because the existence of a separating canonical cover is automatic in this case.

Point \eqref{item: thmA intro sw and Iw extends} of Theorem \ref{theorem: thmA intro} is conceptual and relates to canonical filtrations. On the other hand, \eqref{item: thmA intro wbounded and sep can cov} is combinatorial and, in some cases, computable (see \Cref{rmk-computable}). Thus, Theorem \ref{theorem: thmA intro} gives both conceptual and combinatorial criteria for an EO-stratum closure to be normal and Cohen--Macaulay. We highlight this by computing the smooth locus of all one-dimensional EO-strata closures in any abelian type Shimura variety.

\begin{theoremintro}[{\Cref{theorem: one-dim smooth locus bn cn dn}, \S\ref{sec-antype}}]\label{theorem: thmB intro}
Let $S$ be the special fiber of an abelian type Shimura variety with good reduction. For any one-dimensional Ekedahl--Oort stratum $S_w$, there is an explicit description of the smooth locus of $\overline{S}_w$. When the Shimura datum is of type $\Bsf$, $\Csf$ or $\Dsf$, the smooth locus of $\overline{S}_w$ is always $S_w$.
\end{theoremintro}
We point out that the \(G\)-zip version of Theorem \ref{theorem: thmB intro} also covers the exceptional cases \(\Esf_6\) and \(\Esf_7\). See \S\ref{section: exceptional dim one} and \Cref{remark: exceptional periods}. While Theorem \ref{theorem: thmA intro} applies in complete generality, Theorem \ref{theorem: thmB intro} restricts to EO-strata of dimension one, for arbitrary abelian type Shimura varieties. Our final main theorem, on the other hand, is a result about EO-strata of arbitrary dimension, for Shimura varieties of simple $\Bsf_n$-type. An interesting feature of $\Bsf_n$-type Shimura varieties is their relation to moduli spaces of K3-surfaces \S\ref{section: applications via zip period maps}.
\begin{theoremintro}\label{theorem: thmC intro}
Let $S$ be the special fiber of an abelian type Shimura variety of simple $\Bsf_n$-type with good reduction. Let $S= \overline{S}_{0}\supseteq \overline{S}_{1}\supseteq \ldots \supseteq \overline{S}_{2n-1}$ denote the EO-stratification\footnote{The linearity of the EO-stratification was proved by Wedhorn in \cite{wedhorn.bruhat}.} of $S$.
\begin{enumerate}
    \item For $0\leq j\leq n-1$, we have that:
    \begin{enumerate}
        \item The smooth locus of $\overline{S}_j$ is $\bigcup_{i=j}^{2n-1-j}S_i$.
        \item The closure $\overline{S}_j$ is normal and a local complete intersection.
        \item If $j\neq n-1$, then the closure $\overline{S}_j$ admits a reduced strata Hasse invariant (\S\ref{section: intro reduced Hasse inv}) of weight $(p^{j+1}-1)\eta_\omega$, where $\eta_\omega$ is the Hodge character, whose zero scheme is $\overline{S}_{j+1}$ (with reduced structure).
    \end{enumerate}
    \item For $j\geq n$, the smooth (resp.\ normal) locus of $\overline{S}_j$ is exactly $S_j$.
\end{enumerate}
\end{theoremintro}

\subsection{Method}
We introduce minimal background from the theory of $G$-zips (\S\ref{section: the stack of gzips}) to describe our strategy. The description of our approach starts at~\S\ref{section: flag spaces of gzips}.

Fix a rational prime $p$ and let $k$ be an algebraic closure of $\FF_p$. 
Let $\Zcal$ be a zip datum, i.e.\@ a tuple $\Zcal=(G,P,Q,L,M)$, consisting of the following data. The object $G$ is a connected reductive $\FF_p$-group and $P,Q\subseteq G_k$ are parabolic subgroups with Levi subgroups $L, M$, respectively, satisfying $L^{(p)}=M$ (here $(-)^{(p)}$ denotes base change along the absolute Frobenius). For example, let $\mu\in X_{*}(G)$ be a cocharacter of $G$ and let $P,P_{+}$ be the pair of opposite parabolic subgroups defined via $\mu$. Then, letting $Q\coloneqq P_+^{(p)}$ and $L\coloneqq \text{Cent}_G(\mu)$ we obtain a zip datum denoted $\Zcal_\mu$.

To the datum $\Zcal$, Pink--Wedhorn--Ziegler \cite{Pink-Wedhorn-Ziegler-zip-data, Pink-Wedhorn-Ziegler-F-Zips-additional-structure} attach an algebraic stack 
\[
\Xcal^\Zcal\coloneqq \GZip^\Zcal
\]
parametrizing \textit{$G$-zips of type $\Zcal$}. For example, suppose that $A\to S$ is an abelian scheme over a $k$-scheme $S$. Let $\fil_{\Hdg}^{\bullet}$ and $\fil_{\conj,\bullet}$ denote the Hodge and conjugate filtrations on  $H_{\dr}^1(A/S)$, and let $\varphi_{\bullet}\colon (\text{gr}_{\Hdg}^i)^{(p)}\to \text{gr}_{i,\conj}$ denote the isomorphisms on graded pieces induced by the Cartier isomorphisms \cite{katznilpotent, moonen.wedhorn.discrete.invariants}. The tuple \[(H_{\dr}^1(A/S), \fil_{\Hdg}^{\bullet}, \fil_{\conj,\bullet}, \varphi_{\bullet})\] corresponds to a $\GL_n$-zip for $n=\rk(H_{\dr}^1(A/S))=2\dim(A/S)$. 

\subsubsection{}\label{intro-subsec-Gzip} More generally, a \emph{$G$-zip of type $\Zcal$} over a \(k\)-scheme \(S\) is a tuple $\underline{\Ical}=(\Ical, \Ical_P, \Ical_Q, \iota)$ consisting of 
\begin{enumerate}
    \item a $G$-torsor $\Ical$ over \(S\), 
    \item a sub-$P$-torsor, respectively sub-$Q$-torsor, $\Ical_P, \Ical_Q\subseteq \Ical$, and 
    \item an isomorphism of $M$-torsors $\iota\colon \Big(\Ical_P/R_u(P)\Big)^{(p)}\cong \Ical_Q/R_u(Q)$ (where $R_u(-)$ denotes the unipotent radical).
\end{enumerate}
For more details, see \cite{Pink-Wedhorn-Ziegler-F-Zips-additional-structure}.

\subsubsection{\texorpdfstring{$G$}{G}-zips corresponding to abelian type Shimura varieties}\label{section: intro g-zips to shimura varieties}
Let $(\mathbf{G},\mathbf{X})$ be an abelian type Shimura datum with $\mathbf{G}$ unramified at $p$.
Let $K_p \subset \mathbf{G}(\mathbb{Q}_p)$ be a hyperspecial maximal compact subgroup. As $K^p$ ranges over small enough open compact subgroups of $\mathbf{G}(\mathbf{A}_f^p)$, the associated system of Shimura varieties admits an integral canonical model \cite{kisin-hodge-type-shimura, vasiu}. Fix a level $K=K_pK^p$ and let $S_K$ be the geometric special fiber of the integral canonical model at level $K$. 

Assume now that $(\mathbf{G},\mathbf{X})$ is a Hodge-type Shimura datum. For $g \geq 1$, let $(\GSp_{2g}, \mathbf{X}_g)$ be the Siegel-type Shimura datum. Given a symplectic embedding
\begin{equation}
\label{eq-symplectic-embedding}
\varphi\colon (\mathbf{G},\mathbf{X}) \hookrightarrow (\GSp_{2g}, \mathbf{X}_g), 
\end{equation}
for all sufficiently small $K^p$ there exists a level $K' \subset \GSp_{2g}(\mathbf{A}_f)$ and an induced finite map from $S_K$ to the geometric special fiber of the Siegel-type Shimura variety $S_{g,K'}$  \cite[(2.3.3)]{kisin-hodge-type-shimura}. 

Let $G$ be the reductive $\mathbb{F}_p$-group deduced from $\mathbf{G}$, let $\mu \in X_*(G)$ be a Hodge cocharacter and let $\Zcal_\mu$ be the zip datum obtained from $(G,\mu)$. 
Let $A_K\to S_K$ denote the pullback of the universal principally polarized abelian variety $A_{g,K'}\to S_{g,K'}$. Then, by \cite{zhang, Kisin-Madapusi-Pera-Shin-Honda-Tate}, $H_{\dr}^1(A_K/S_K)$ determines a smooth and surjective morphism
\begin{equation}
\label{eq-zeta-shimura}
\zeta_K:S_K \to \Xcal^{\Zcal_\mu}.
\end{equation}    
Furthermore, $\zeta_K$ extends to a smooth toroidal compactification $S_K^{\tor}$ of $S_K$ by \cite{Goldring-Koskivirta-zip-flags}, and the extension $\zeta^{\tor}_K\colon S_K^{\tor}\to  \Xcal^{\Zcal_\mu}$ is again smooth \cite{Andreatta-modp-period-maps}.

\begin{remark}\label{remark: zip period map abelian type}
In the abelian, non-Hodge-type case, one still has a morphism \eqref{eq-zeta-shimura} upon replacing $\GZip^\mu$ with $G^{\text{ad}}\texttt{-Zip}^{\Zcal_{\mu^{\text{ad}}}}$ \cite{shen.zhang} (see also \cite{imai2024prismaticrealizationfunctorshimura} for a slight generalization of this).
\end{remark}
\begin{remark}
If $S_{K'}$ is another object in the projective system (for $K\subseteq K'$) and $\pi_{K,K'}\colon S_K\to S_{K'}$ the corresponding transition map, then $\zeta_{K'}\circ \pi_{K,K'}=\zeta_K$. Hence, we often omit the subscript $K$ and simply write $S$ and $\zeta$.
\end{remark}
\subsubsection{The Ekedahl--Oort stratification}
Let $\Zcal$ be an arbitrary zip datum, and let $E\coloneqq P\times_M Q$ act on $G$ by $(a,b)\cdot g=agb^{-1}$ (see \S\ref{section: the stack of gzips}). By \cite{Pink-Wedhorn-Ziegler-zip-data}, we have an isomorphism of $k$-stacks
\begin{equation}\label{equation: gzip is quotient, intro}
    \Xcal^\Zcal \cong [E\backslash G_k].
\end{equation} 
In particular, $\Xcal^{\Zcal}$ is a smooth algebraic $k$-stack of dimension 0. 

The $E$-orbits of $G_k$ are parametrized by a certain subset ${}^{I}W$ of the Weyl group $W$ of $G$, described as follows. Let $T\subseteq B\subseteq G$ be a Borel pair, i.e.\@ a maximal torus inside a Borel. Let $\Delta$ denote the set of simple $T$-roots with respect to $B$, and let $I\subseteq \Delta$ denote the type of $P$. Then, ${}^{I}W$ is the set of minimal length representatives in the right cosets $W_I\backslash W$, where $W_I\coloneqq \langle s_\alpha \colon \alpha\in I\rangle$. Furthermore, there is a partial order $\preccurlyeq_{I}$ on ${}^{I}W$ inducing a homeomorphism $|\Xcal^\Zcal|\cong {}^{I}W$. The corresponding pointwise stratification of $\Xcal^\Zcal$ is called the \textit{zip stratification} and, by definition, the \textit{Ekedahl--Oort} stratification of an abelian type Shimura variety is the stratification on $S$ induced by the zip stratification via \eqref{eq-zeta-shimura}. For $w\in \IW$, let $\Xcal_w^\Zcal$ denote the corresponding zip stratum.

From now on, set $\Xcal\coloneqq \Xcal^{\Zcal_\mu}$ for a zip datum $\Zcal_\mu$ coming from a cocharacter $\mu\in X_{*}(G)$. The inclusion $E\subseteq P\times Q$ yields a smooth surjection $\Xcal\cong [E\backslash G_k]\to [(P\times Q)\backslash G_k]$. The classical Bruhat stratification on $G$ provides a stratification on $[(P\times Q)\backslash G_k]$, all of whose strata closures are normal and Cohen--Macaulay, by \cite{ramarama, rama, andersi}. Thus, if $\Xcal_w$ is open in the preimage of a Bruhat stratum, then its closure is normal and Cohen--Macaulay. The problem is that very few zip strata are open in such a Bruhat stratum. For this reason, we introduce flag spaces over $\Xcal$. 

\subsubsection{Flag spaces over the stack of \texorpdfstring{$G$}{G}-zips}\label{section: flag spaces of gzips}
Adding to the zip datum $\Zcal$ the datum of an intermediate parabolic $B\subseteq P_0 \subseteq P$, Goldring--Koskivirta \cite{Goldring-Koskivirta-zip-flags} construct the stack of \textit{$G$-zip flags of type $\Zcal$}, denoted $\Fcal^{(P_0)}\coloneqq \GF^{\Zcal}_{P_0}$. It parametrizes $G$-zips $(\Ical,\Ical_P,\Ical_Q,\iota)$ of type $\Zcal$ equipped with the extra datum of a sub-$P_0$-torsor $\Ical_{P_0}\subseteq \Ical_{P}$. For example, for the $G$-zip $H_{\dr}^1(A/S)$ attached to a Hodge-type Shimura variety \eqref{eq-zeta-shimura}, sub-$P_0$-torsors correspond to refinements of the Hodge filtration. Forgetting the $P_0$-torsor, one obtains a $P/P_0$-fibration $\pi_{P_0}\colon \Fcal^{(P_0)}\to \Xcal^\Zcal$. 

Let $L_0$ be the Levi subgroup of $P_0$ containing $T$ and let $M_0\coloneqq L_0^{(p)}$. Letting $Q_0\coloneqq M_0 {}^{z}B$, for $z = \sigma(w_{0,I})w_0$ the Frobenius twist of the longest element in $\IW$ (cf.\@ \S\ref{ssecflagspace}), one obtains another zip datum $\Zcal_{P_0}$ and a smooth and surjective morphism $\Psi_{P_0}\colon \Fcal^{(P_0)}\to \Xcal^{\Zcal_{P_0}}$. Composing with the natural projection $\Xcal^{\Zcal_{P_0}}\to [(P_0\times Q_0)\backslash G_k]$ we also obtain a smooth surjection $\Fcal^{(P_0)}\to [(P_0\times Q_0)\backslash G_k]$. 

The zip stratification of $\Xcal^{\Zcal_{P_0}}$ is indexed by ${}^{I_0}W$, where $I_0\subseteq \Delta$ is the type of $P_0$. Pulling back via $\Psi_{P_0}$, this gives a zip stratification on $\Fcal^{(P_0)}$. For $w\in {}^{I_0}W$, let $\Fcal_w^{(P_0)}$ denote the corresponding zip stratum. Note that ${}^{I_0}W\supseteq \IW$ since $I_0\subseteq I$. If $w\in \IW$, then $\pi_{P_0}$ restricts to a finite \'etale morphism $\Fcal_w^{(P_0)} \to \Xcal_w^\Zcal$. 

Suppose that $\Ucal\subseteq \overline{\Xcal}_w^{\Zcal}$ is an open substack of $\overline{\Xcal}_w^{\Zcal}$. Let $\Gamma_\Ucal\subseteq \IW$ correspond to $|\Ucal|$ and set $\Ucal^{(P_0)}\coloneqq \bigcup_{w'\in \Gamma_\Ucal}\Fcal_{w'}^{(P_0)}$. The previous discussion gives us the following commutative diagram, which we describe in further detail in the next two subsections.

\begin{equation*}
\begin{tikzcd}[remember picture]
 \Ucal^{(P_0)} \arrow[drr, bend left = 15, dotted] & & \\
 \pi_{P_0}^{-1}(\Ucal)\cap \overline{\Fcal}_w^{(P_0)} \arrow[r, hook] \arrow[d, hook] & \pi_{P_0}^{-1}(\Ucal) \arrow[r, "\pi_{P_0}"] \arrow[d, hook] & \Ucal \arrow[d, hook]\\
 \overline{\Fcal}_w^{(P_0)} \arrow[r, hook] & \pi_{P_0}^{-1}(\overline{\Xcal}_w^\Zcal) \arrow[r, "\pi_{P_0}"] \arrow[d, hook] & \overline{\Xcal}_w^\Zcal \arrow[d, hook]\\
 & \Fcal^{(P_0)} \arrow[r, "\pi_{P_0}"] \arrow[d] \arrow[dd, bend right = 40, "\beta^{(P_0)}", swap] & \Xcal^\Zcal \\
 & \Xcal^{\Zcal_{P_0}} \arrow[d] \\
 & \left[(P_0\times Q_0)\backslash G\right]
\end{tikzcd}
\begin{tikzpicture}[overlay,remember picture]
    \path (\tikzcdmatrixname-1-1) to node[midway,sloped]{$\subseteq$}
(\tikzcdmatrixname-2-1);
\end{tikzpicture}
\end{equation*}

\subsubsection{The role of \texorpdfstring{$w$}{w}-boundedness and separating covers}\label{section intro: bounded separating cover}

To obtain meaningful information about the geometry of $\Ucal$, we choose $P_0\coloneqq P_w$, where $P_w$ is the unique parabolic with $B\subseteq P_w\subseteq P$ such that
\begin{enumerate}
    \item the flag zip stratum $\Fcal^{(P_w)}_w$ coincides with the preimage of a Bruhat stratum of $[(P_w\times Q_w)\backslash G]$, and
    \item the morphism $\Fcal^{(P_w)}_w \to \Xcal_w$ is an isomorphism.
\end{enumerate}
The unique parabolic $P_w$ such that (1) and (2) above hold is called the \textit{canonical parabolic of $w$} (cf.\@ \S\ref{sec-canonical}). Let $I_w\subseteq \Delta$ denote the type of $P_w$. 

In the above diagram, the map $\pi_{P_w}^{-1}(\Ucal)\cap \overline{\Fcal}_w^{(P_w)}\to \Ucal$ is surjective, birational, proper and its source is normal and Cohen--Macaulay. However, it is in general too coarse to detect the singularities of $\Ucal$. For this reason, we instead study the subset $\Ucal^{(P_w)}$. In order to make sense of $\Ucal^{(P_w)}$ as a substack we need it to be open in $\overline{\Fcal}_w^{(P_w)}$, in which case the dotted arrow in the above diagram exists as a map of stacks. In this case, we say that $\Ucal$ \textit{admits a canonical cover} (\Cref{def-cover}). Since $\Ucal^{(P_w)}$ is open in $\overline{\Fcal}_w^{(P_w)}$, it is normal and Cohen--Macaulay. If furthermore $\Ucal^{(P_w)}=\Ucal\times_{\overline{\Xcal}_w}\overline{\Fcal}_w^{(P_w)}$, in which case we say that $\Ucal$ \textit{admits a separating canonical cover} (\Cref{def-sepcov}), then the map $\Ucal^{(P_w)}\to \Ucal$ is finite and birational so we can hope to detect some geometry of $\Ucal$ from $\Ucal^{(P_w)}$. However, a finite birational map whose target is not normal need not be an isomorphism. The final condition we must add in this situation to ensure that $\Ucal^{(P_w)}\to \Ucal$ is an isomorphism is $w$-boundedness. By definition, this means that for all $w'\in \Gamma_\Ucal$, the canonical parabolic of $w'$ is contained in the canonical parabolic of $w$. This condition implies that $\Fcal_{w'}^{(P_w)}\to \Xcal_{w'}$ is an isomorphism for all $w'\in \Gamma_\Ucal$ and allows us to show that $\Ucal^{(P_w)}\to \Ucal$ is unramified, hence an isomorphism (cf.\@ \Cref{bounded-implies-normal}).

This describes roughly the role of \eqref{item: thmA intro wbounded and sep can cov} in Theorem \ref{theorem: thmA intro}.

\subsubsection{The role of the canonical torsor}\label{section: the role of the can torsor}
Next we describe the role of \eqref{item: thmA intro sw and Iw extends} in Theorem \ref{theorem: thmA intro}. The homeomorphism $\IW\cong |\Xcal|$ is induced by the map $w\mapsto E\cdot (wz^{-1})$, see \S\ref{sec-param}. Hence, $\Xcal_w\cong [\Stab_E(wz^{-1})\backslash 1]$. This gives a different characterization of the canonical parabolic $P_w$. Namely, it is the smallest standard parabolic containing $\pr_1(\Stab_E(wz^{-1}))$, where $\pr_1\colon E\to P$ is the first projection (cf.\@ \Cref{lem-stabinpw}). Hence we have a map $\Xcal_w\to [P_w\backslash 1]$ defining the \textit{canonical torsor of $w$}, denoted $\Ical_w$ (\Cref{def-cantorcansec}). In the case of a PEL-type Shimura variety, this torsor corresponds to the canonical filtration (\S\ref{ssec-fzipcanfilt}). From the modular description of $\Fcal^{(P_w)}$, we see that $\Ical_w$ corresponds to the canonical section $s_w\colon \Xcal_w\to \Fcal_w^{(P_w)}$. Suppose that $s_w$ extends to a section $\widetilde{s}_w\colon \Ucal \to \pi_{P_w}^{-1}(\Ucal)\cap \overline{\Fcal}_w^{(P_w)}$. Then, $\widetilde{s}_w$ being a section and $\pi_{P_w}^{-1}(\Ucal)\cap \overline{\Fcal}_w^{(P_w)}$ being irreducible implies that $\Ucal\cong \widetilde{s}_w(\Ucal)=\pi_{P_w}^{-1}(\Ucal)\cap \overline{\Fcal}_w^{(P_w)}$. Thus, $\Ucal$ is normal and Cohen--Macaulay.

\subsection{Auxiliary results and applications}

\subsubsection{Applications via zip period maps: moduli of \texorpdfstring{$K3$}{K3}-surfaces} 
\label{section: applications via zip period maps}
Given an algebraic stack $\Scal$ and a smooth surjective morphism $\Scal\to \Xcal^\Zcal$, we obtain a zip stratification on $\Scal$ from that on $\Xcal^\Zcal$. By smoothness, all the results obtained on $\Xcal^\Zcal$ carry over to $\Scal$. In particular, for $\Scal=S$ a Shimura variety as in \S\ref{section: intro g-zips to shimura varieties} we have such a map and Theorem \ref{theorem: thmA intro}, \ref{theorem: thmB intro}, and \ref{theorem: thmC intro} are all obtained from corresponding results on the stack of $G$-zips. Here we single out another important example.

Let \(M^\circ_K\) be the moduli space over \(k\) of polarized K3 surfaces with level structure \(K\). Artin, in \cite{artink3ssing}, used the formal Brauer group to define a linear stratification given by the height and the Artin invariant $M_K^\circ = M_{K,0}^\circ \supseteq M_{K,1}^\circ \supseteq \cdots \supseteq M_{K,19}^\circ$. This stratification was later redefined and studied by van der Geer and Katsura \cite{geer.katsura.stratification.of.k3}, respectively Ogus \cite{ogus}; the former continuing with Artin's approach with formal groups, the latter using crystalline techniques. One of the main results of \cite{ogus, geer.katsura.stratification.of.k3} states that for all $0<i\leq 9$, $M_{K,i}^\circ$ is normal, with singular locus $M_{K,20-i}^\circ$.

Let $(\mathbf{G},\mathbf{X})$ be a Shimura datum where $\mathbf{G}$ is a $\QQ$-group such that $\mathbf{G}_{\RR}\cong \mathbf{SO}(19,2)$. Let $S_K$ denote the corresponding positive characteristic Shimura variety as in \S\ref{section: intro g-zips to shimura varieties}. The Kuga--Satake construction \cite{madapusiK3} gives an open embedding $M_K^\circ \hookrightarrow S_K$ providing a commutative diagram
\begin{equation}
    \begin{tikzcd}
        M_K^\circ \arrow[rr, hook] \arrow[rd, "\zeta_{M_K^\circ}"'] & & S_K \arrow[dl, "\zeta_K"] \\
         & \GZip^\mu &
    \end{tikzcd}
\end{equation}
where $\zeta_{M_K^\circ}$ is smooth and surjective. 
The aforementioned stratification on $M_K^\circ$ is obtained from the zip stratification on $\GZip^\mu$ via $\zeta_{M_K^\circ}$ \cite[\S7.16]{moonen.wedhorn.discrete.invariants}. Hence, Theorem \ref{theorem: thmC intro} provides a group-theoretical proof of the main results of \cite{ogus, geer.katsura.stratification.of.k3}, as well as a generalization of them to simple $\Bsf_n$-type Shimura varieties of arbitrary dimension.

\subsubsection{Reduced Hasse invariants}\label{section: intro reduced Hasse inv}
The theory of Hasse invariants has been very useful both to produce congruences in the Langlands program \cite{deligne.serre, taylor, emerton.et.al, koskgold} and in the study of the geometry of Shimura varieties mod $p$; in particular to the study of the Ekedahl--Oort stratification \cite{wedhorn.ziegler.tautological.rings,koskgold}. For applications to both of these areas, it is useful to know if the divisor of a Hasse invariant is reduced. We say that a Hasse invariant is \textit{reduced} if its scheme-theoretic vanishing locus is reduced. 
By \cite{koskgold}, Hasse invariants exist on all Ekedahl--Oort strata for any Hodge-type cocharacter data $(G,\mu)$. However, reduced Hasse invariants may not always exist. When sufficiently many reduced Hasse invariants exist, one can use them to show that certain Ekedahl--Oort strata closures are Cohen--Macaulay. Existence of reduced Hasse invariants is also related to normality and smoothness of EO-strata closures \S\ref{section: intro motivation}\eqref{item: intro motivation -  hasse invariants}. In this text, we prove the existence of reduced strata Hasse invariants on all EO-strata of codimension $\leq n-1$ for any cocharacter datum $(G,\mu)$ of simple type $\Bsf_n$, which we use to determine the singular locus therein. These reduced Hasse invariants also allow us to compute cycle classes of EO-strata closures, as well as deduce cohomological vanishing results for them \S\ref{section: intro cycle classes and coh van}.

By applying Theorem \ref{theorem: thmA intro} (\Cref{thm-extnorm}) to elementary substacks (\Cref{def-elemop}) of codimension zero, we can compute the multiplicities of the divisor of any $\mu$-ordinary Hasse invariant of arbitrary abelian type Shimura varieties with good reduction by computing the multiplicities on the flag space via Chevalley's formula. More generally, for any EO stratum closure we can use the same technique to compute the divisor restricted to any normal elementary substack. This is work in progress by the authors \cite{Koskivirta-LaPorta-Reppen-multiplicity-one}. 
\subsubsection{Applications to cycle classes and cohomological vanishing of EO-strata closures}\label{section: intro cycle classes and coh van}
Let $S^{\tor}$ denote a smooth toroidal compactification of an abelian type Shimura variety of type $\Bsf_n$ as in \S\ref{section: intro g-zips to shimura varieties}. Let $S^{\tor}=\overline{S}^{\tor}_0\supseteq \cdots \supseteq \overline{S}^{\tor}_{2n-1}$ denote the EO-stratification. Let $\omega$ denote the Hodge line bundle. We obtain an explicit closed form formula for the cycle classes of all EO-strata $\overline{S}^{\tor}_j$ for $j\leq n-1$ in terms of the cycle class of $\omega$ (cf.\@ \Cref{corollary: cycle class formula}). For such strata, we also obtain vanishing for all cohomology groups $H^i(\overline{S}^{\tor}_j,\Vcal(\eta)^{\rm sub}\otimes\omega^{m})$, $i>0$, for $m\gg0$, where $\Vcal^{\rm sub}(\eta)$ is the sub-canonical extension of a ``sufficiently nice'' automorphic vector bundle $\Vcal(\eta)$ (cf.\@ \Cref{corollary: coh vanishing}). 

\subsubsection{Changing the center}
Using the isomorphism $\Xcal\cong [E\backslash G_k]$, it is often possible to study intricate geometric questions by group theory. Embedding $G$ into a general linear group one can even do explicit matrix computations (e.g.\@ as in \S\ref{sec-orthogonal}). However, it is not clear if statements obtained from such computations apply to other groups of the same type (e.g.\@ compare $\GL_n$ and $\PGL_n$). To allow for the freedom of choosing an explicit group to work with, we study the map $\Xcal'\coloneqq \GprimeZip^{\Zcal'}\to \Xcal$ induced by a central extension $G'\to G$. We prove that the morphism $\Xcal'\to \Xcal$ is a smooth homeomorphism. Moreover, denoting by $K$ the kernel of the extension $G'\to G$, when \(K\) is connected, we show that the morphism $\Xcal'\to \Xcal$ is a $K(\FF_p)$-banded gerbe.

\subsection{Outline}
\par In \S\ref{sec-simpexa}, we work out a simple example, using the language of Shimura varieties and Dieudonn\'{e} modules, in the case of \(G = \GL_n\). We explain when and why the closure of the only EO-stratum of dimension one is smooth. We include this section to guide readers unfamiliar with the language of $G$-zips and to motivate what we do in the rest of the paper.
\par In \S\ref{sec-revgzip}, we review the theory of \(G\)-zips and \(G\)-zip flags. We define the zip and Bruhat stratifications on those objects. We also recall the definition and basic properties of the canonical parabolic \(P_w\).
\par In \S\ref{section: changing the center}, we show that central extensions of \(\Fp\)-reductive groups give rise to smooth homeomorphisms between the corresponding stacks of \(G\)-zips.
\par In \S\ref{sec-singostrat}, we discuss general results about singularities in unions of zip strata. We prove the zip-theoretic version of Theorem \ref{theorem: thmA intro} and some important corollaries.
\par In \S\ref{ssec-lengthone}, we study length one strata in the case of groups of type $\Asf, \Bsf, \Csf, \Dsf, \Esf_6$ and \(\Esf_7\) with a minuscule cocharacter and prove Theorem \ref{theorem: thmB intro}.
\par In \S\ref{sec-orthogonal}, we study EO-strata in Shimura varieties of simple $\Bsf_n$-type and prove Theorem C. 

\subsection*{Acknowledgments}
We thank Wushi Goldring for several fruitful discussions on topics related to this paper. We are also very grateful to Pol van Hoften for reading through an earlier draft of the paper and giving valuable feedback that greatly improved the article. Last but not least, we wish to thank the referee for their useful and detailed comments on the manuscript.

\par J-SK was supported by the University of Caen-Normandie. LLP wishes to thank the Department of Mathematics ``\emph{Tullio Levi-Civita}'' of the University of Padova and its members for their continued support and friendship. SR was supported by JSPS KAKENHI Grant Number 24KF0146 and Knut \& Alice Wallenberg Foundation Grant Number KAW 2023.0417. 

\subsection*{AI disclosure}
The authors did not use AI when writing the first complete version of this manuscript. The authors did use AI when responding to the referee's comments. Publicly available AI models were helpful in discussions concerning \Cref{lemma: canonical parabolic is everywhere}, particularly in clarifying the construction of Frobenius-fixed points, suggesting the consideration of regular unipotent elements, and to search for counterexamples. OpenAI's \emph{GPT-6}, \emph{GPT-6.2} Plus, Anthropic's \emph{Claude Sonnet 5.5} and Google's \emph{Gemini 4 Hydrogen} were used.

\renewcommand{\thesection}{N}
\renewcommand{\thesubsection}{{\thesection.\arabic{subsection}}}
\renewcommand{\thesubsubsection}{\thesubsection.\arabic{subsubsection}}
\setcounter{subsection}{0}
\section*{Notation}
\label{sec-notation}

\subsection{Ring theory}
Throughout, \(k\) denotes an algebraic closure of the field \(\FF_p\), for \(p\) a fixed prime. 
\begin{itemize}
    \item We denote the absolute Frobenius \(x \mapsto x^p\) on \(k\) by \(\sigma\). It is a topological generator of \(\Gal(k/\FF_p)\), the absolute Galois group of \(\FF_p\).
\end{itemize}

\subsection{Scheme theory}
\begin{itemize}
    \item For a fixed base scheme \(S \in \sch\), we denote by \(\sch_S\) the category of \(S\)-schemes. If \(S= \spec(A)\), for \(A\) a ring, we simply write \(\sch_A\).
    \item If \(X \in \sch_S\) and \(T \to S\) is a morphism, we sometimes write \(X_T\) instead of \(X \times_S T\). If \(S=\spec(A),T = \spec(B)\), we may write \(X_B\).
    \item For a field extension \(E/F\) and \(X \in \sch_E\) we write \(\Res_{E\mid F}(X)\) for the Weil restriction of \(X\). Recall that \(\Res_{E\mid F}(X)(S) = X(S_E)\) for all \(S \in \sch_F\).
\end{itemize}
\subsubsection{Frobenius morphism} Let \(S \in \sch_{\FF_p}\) and \(X \in \sch_S\).
\begin{itemize}
    \item Let \(F_S \colon S\to S\) denote the absolute Frobenius of \(S\), \(F_A\) if \(S = \spec(A)\). In particular, \(F_k = \spec(\sigma)\).
    \item We denote by \(X^{(p)}\) the fiber product \(X \times_{F_S} S\), called the \emph{Frobenius twist} (or \emph{\(p\)-twist}) of \(X\) (over \(S\)).
    \item For \(\mc F\) any sheaf over \(S\), we write \(\mc F^{(p)} \coloneqq F_S^\ast(\mc F)\), the \emph{Frobenius twist} (or \emph{\(p\)-twist}) of \(\mc F\).
    \item We denote by  \(\varphi \colon X \to X^{(p)}\) the relative Frobenius of \(X\) over \(S\).
    \item If \(S = \spec(\FF_p)\), the group \(\Gal(k/\FF_p)\) acts naturally on \(X(k)\), by \(\gamma\cdot x \coloneqq x \circ \spec(\gamma)\), for \(x \in X(k), \gamma \in \Gal(k/\FF_p)\). In particular, \(\Gal(k/\FF_p)\) acts naturally on \(\hom(\FF_{p^d}, k)\). This action factors through \(\Gal(\FF_{p^d}/\FF_{p})\) and \(\hom(\FF_{p^d}, k)\) is naturally a \(\Gal(\FF_{p^d}/\FF_{p})\)-torsor.
\end{itemize}

\subsection{Group theory}
\begin{itemize}
    \item We write \(\Sfr_n\) for the symmetric group on \(n\) elements. We endow \(\Sfr_n\) with the set of simple reflections \(\{(i, i+1), 1\leq i < n\}\). 
    \item For \(w \in \Sfr_n\), we sometimes write \(w = [w(1), w(2), \ldots, w(n)]\).
\end{itemize}

\subsubsection{Algebraic groups} Let \(G\) be a linear algebraic group over a field \(F\).
\begin{itemize}
    \item For a subgroup $H\subseteq G$ and an element $g \in G(F)$, we denote by ${}^gH$ the subgroup $gHg^{-1}$.
    \item Write $R_{\textrm{u}}(G)$ for the unipotent radical of $G$. 
    \item Write \(X^\ast(G)\), respectively \(X_\ast(G)\), for the characters, respectively cocharacters, of \(G_{\overline{F}}\).
\end{itemize}

\subsubsection{Reductive groups}\label{not-redWeyl} Let \(G\) be a connected, reductive group over a field \(F\) with a Borel pair \((B, T)\). Let \(\overline{F}\) be an algebraic closure of \(F\). 
\begin{itemize}
    \item Write \(W \coloneqq W(G_{\overline F}, T_{\overline F})\) for the Weyl group \(N(T_{\overline F})/T_{\overline F}\) of \(G\).
    \item Write \(B^+\) for the unique Borel of \(G\) such that $B\cap B^+=T$.
    \item Write \(\Phi = \Phi(G_{\overline F}, T_{\overline F})\) for the \(T_{\overline F}\)-roots of \(G_{\overline F}\).
    \item Write \(\Phi^+ \subseteq \Phi\) for the positive roots relative to \(B^+\), that is, roots \(\alpha \in \Phi\) such that the \(\alpha\)-root group \(U_\alpha\) is contained in \(B^+_{\overline F}\).
    \item Write \(\Delta \subseteq \Phi^+\) for the simple roots.
    \item Write \(\ell \colon W \to \Z_{\geq 0}\) for the length function of \(W\) relative to \(\Delta\).
    \item Write \(\leq\) for the Bruhat order on \(W\) given by \(\Delta\).
    \item For a standard Levi subgroup $L\subseteq G_{\overline{F}}$, write \(\Phi_L, \Phi^+_L, \Delta_L\) for the root data of \(L\) with respect to $(B\cap L,T)$. When we write \(I = \Delta_L\), then we set \(W_I \colonequals W(L_{\overline F}, T_{\overline F})\).
    \item Write $\IW$ (resp.\ $W^I$) for the set of elements $w\in W$ of minimal length in the coset $W_I w$ (resp.\@ $wW_I$).
    \item We have a natural action of \(\Gal(\overline F/ F)\) on \(X^\ast(T), X_\ast(T)\) and \(W\). For all \(\gamma \in \Gal(\overline F/ F), w \in W\) and \(\lambda \in X^\ast(T)\), we have \((\gamma\cdot w)(\gamma\cdot \lambda) = \gamma(w(\lambda))\).
    \item We write \(w_0\) for the longest element of \(W\) and \(w_{0, I}\), for the longest element of \(W_I\). The longest element of $\IW$ is $w_{0,I}w_0$. The shortest element of $\IW$ is the identity element $\id\in \IW$.
\end{itemize}
In \S\ref{ssec-lengthone}, we consider groups of type $\Asf$, $\Bsf$, $\Csf$, $\Dsf$, \(\Esf_6\) and \(\Esf_7\). In each case, we say that $G$ is of a certain type if $G^{\ad}_k$ is isomorphic to a direct product of simple adjoint groups of this type. Throughout the paper, $\Xcal$, $\Xcal^{\Zcal}$ denote the stack of $G$-zips of type $\Zcal$ (\S\ref{section: the stack of gzips}).

\renewcommand{\thesection}{\arabic{section}}
\renewcommand{\thesubsection}{{\thesection.\arabic{subsection}}}
\renewcommand{\thesubsubsection}{\thesubsection.\arabic{subsubsection}}

\section{A simple example}
\label{sec-simpexa}
For this section, let us denote by $S$ the fiber over $k$ of a unitary Shimura variety relative to a quadratic imaginary extension $E/\Q$, at $p$-hyperspecial level, for $p$ split in $E$ and of signature $(r, s)$, with $r+s=n$. In the notation of \ref{section: intro g-zips to shimura varieties}, this Shimura variety is obtained from a reductive unitary group \(\mathbf{G}\) over \(\Q\). Write $\nu$ for the unique embedding of $E$ into \(\C\) giving the isomorphism \(\mathbf{G}(\R) \cong U(r, s)\) and let $\overline{\nu}$ be its complex conjugate.

We will see (\Cref{thm-len1gln}) that the smoothness of the Zariski closure of the one-dimensional Ekedahl--Oort stratum is equivalent to the condition that $r,s$ are coprime. We study below the behavior of the canonical filtration on the zero- and one-dimensional Ekedahl--Oort strata of $S$, in the classical language of Dieudonn\'{e} modules. For more precise results, using the theory of $G$-zips, see \S\ref{subsub-simple-split-A}. 

\par For most notation, and the special case $s = 1$, see \cite{laporta2023generalised} and \ref{section: intro g-zips to shimura varieties}. See \cite[Section~3.5]{wooding} and references therein, such as \cite{gsas}, for the standard Dieudonn\'{e} modules on general EO-strata.

\subsection{Standard Dieudonn\'{e} modules}
We denote the unique one-dimensional stratum by $S_\alpha$ (for consistency with the notation used in \S\ref{sec-antype}). Here $\alpha$ refers to the unique non-compact simple root. It corresponds to the reflection $\alpha = (r, r+1) \in \IW$, where $W \cong \Sfr_n$ is the Weyl group of the underlying reductive group $G \cong \GL_{n, \Fp}$. In particular, $\alpha$ is the unique element of length $1$ of $\IW$. The unique zero-dimensional stratum $S_{\id}$ corresponds to the identity element $\id \in \IW$.
\par Let $A \to S$ be the universal abelian scheme and take $x$ a closed point of either $S_\alpha$ or $S_{\id}$. By Kraft's classification of Barsotti--Tate groups truncated at level \(1\), see \cite[Section~3.5]{wooding} or \cite{gsas}, we can give a standard basis of the \(k\)-vector space $D = D_x \coloneqq \mb{D}(A_x[p])_{\nu} \cong H^1_{\mr{dR}, \nu}(A_x/k)$, the CM-component corresponding to $\nu$ of the (contravariant) Dieudonn\'{e} module, and describe the action of Frobenius $F$ and Verschiebung $V$ with respect to such a basis. Note that it is enough to describe the canonical filtration on $D_x$ to deduce it for the $\overline{\nu}$-component and, by the Dieudonn\'{e} correspondence, also deduce it for $A_x[p]$. We summarize this description in the following proposition.

\begin{proposition}[Standard bases]
\label{propstdb}
We have the following.
\begin{enumerate}
    \item If $x \in S_{\id}$, there is a $k$-basis $(e_1, e_2, \ldots, e_n)$ of $D$ such that:
    \begin{itemize}
        \item $F(e_i) = 0$ for $1 \leq i \leq r$ and $F(e_i) = e_{i-r}$ for $r < i \leq n$,
        \item $V(e_i) = 0$ for $1 \leq i \leq s$ and $V(e_i) = e_{i-s}$ for $s < i \leq n$.
    \end{itemize}
    \item If $x \in S_\alpha$, there is a $k$-basis $(e_1, e_2, \ldots, e_n)$ of $D$ such that:
    \begin{itemize}
        \item $F(e_i) = 0$ for $1 \leq i \leq r-1$, $F(e_r) = e_1, F(e_{r+1}) = 0$ and $F(e_i) = e_{i-r}$ for $r+1 < i \leq n$,
        \item $V(e_i) = 0$ for $1 \leq i \leq s$, $V(e_i) = e_{i-s}$ for $s < i < n$ and $V(e_n) = e_{r+1}$.
    \end{itemize}
\end{enumerate}
This makes it possible to describe the Hodge filtration $\ker F \subseteq D$:
    \begin{enumerate}
    \setcounter{enumi}{2}
        \item On $S_{\id}$ the Hodge filtration is given by $\Span_k (e_1, \ldots, e_r) \subseteq D$.
        \item On $S_\alpha$ the Hodge filtration is given by $\Span_k(e_1, \ldots, e_{r-1}, e_{r+1}) \subseteq D$.
    \end{enumerate}
\end{proposition}

\subsection{Action on intervals}
Set $D_i = \Span_k (e_1, \ldots, e_i), 1 \leq i \leq n$. We refer to these as \emph{interval subspaces} of $D$. 
\begin{lemma}
\label{lem-int} We have the following.
\begin{enumerate}
    \item If $x \in S_{\id}$, then:
    \begin{itemize}
        \item $F(D_i) = (0)$ for $1 \leq i \leq r$ and $F(D_i) = D_{i-r}$ for $r < i \leq n$,
        \item $V^{-1}(D_i) = D_{i+s}$ for $1 \leq i < r$ and $V^{-1}(D_i) = D$ for $i \geq r$.
    \end{itemize}
    \item If $x \in S_\alpha$, then:
    \begin{itemize}
        \item $F(D_i) = (0)$ for $1 \leq i < r$, $F(D_i) = D_1$ for $i = r, r+1$, and $F(D_i) = D_{i-r}$ for $r+1 < i \leq n$,
        \item $V^{-1}(D_i) = D_{i+s}$ for $1 \leq i < r$, $V^{-1}(D_r) = D_{n-1}$ and $V^{-1}(D_i) = D$ for $i \geq r+1$.
    \end{itemize}
\end{enumerate}
Moreover, we observe that:
\begin{enumerate}
\setcounter{enumi}{2}
    \item The action of $F, V^{-1}$, in the bases of \Cref{propstdb}, sends interval subspaces to interval subspaces, so that the same holds for any word in the letters $F, V^{-1}$. In particular, the canonical filtrations on $S_{\id}$ and $S_\alpha$ are given by a flag of interval subspaces.
    \item \label{lem-int4} The action of $F, V^{-1}$, for $1 \leq i \leq n$, differs from $S_\alpha$ to $S_{\id}$ only on the interval subspace $D_i$ given by $i = r$.
\end{enumerate}
\end{lemma}
\begin{proof}
From \Cref{propstdb} and by direct inspection.
\end{proof}

\subsection{A minimal word for \texorpdfstring{$D_r$}{the exceptional interval}}
Assume from now on that $r>s>0$ and that $r,s$ are coprime. We describe a word in $F, V^{-1}$, of minimal length among those that give $D_r$ starting from $(0) \subseteq D$, using \Cref{lem-int}. This will allow us to compute the \emph{(conjugate) canonical filtration} or \emph{flag} of \(D\), which we define as the unique coarsest refinement of the conjugate filtration \((0)\subseteq \ker(V) \subseteq D\) stable under taking images via \(F\) and pre-images via \(V\). 
\par First, we observe that since $r > s$ and $F(D_i) \subseteq D_s, 1\leq i \leq n$, we must have $D_r = V^{-1}(D_{r-s})$. By Euclidean division, there is a unique $i_1>0$ such that
\[
    r - i_1 s = r_1, \quad 0 \leq r_1 < s.
\]
Now, $D_{r} = V^{-i_1}(D_{r_1})$ and, since $r_1 < s$, if \(r_1 \neq 0\), we can take  $D_{r_1} = F(D_{r+r_1})$. We remark that, over $S_\alpha$, if $r_1 = 1$, then we could also take, by \Cref{lem-int}, $D_{r_1} = F(D_r)$, but since we are looking for a \emph{minimal} word in $F, V^{-1}$, for $D_r$, this choice is not allowed. Repeating this procedure we obtain a sequence of integers $i_j, r_j \geq 0,$ such that
\begin{align*}
    0 \leq r &- i_1 s = r_1 < s,\\
    0 \leq 2 r & - (i_1 + i_2)s = r + r_1 - i_2 s = r_2 < s, \\
    &\ldots\,,\\
    0 \leq a r & - \Big(\sum_{j=1}^a i_j\Big)s = r_a < s.
\end{align*}
Write $b_a = \sum_{j=1}^a i_j$. Notice that, eventually, we must hit $r_a = 0$, which corresponds to having reduced to the trivial subspace $(0)$ of $D$. In fact, we can take $a = s$ and $b_a = r$: we can choose $i_a$ so that $b_a = r$ because at every step $a' \leq s$ we must have $b_{a'}\leq r$ to have $r_{a'} \geq 0$. This shows that the process terminates and the interval subspace $D_r$ is part of the canonical filtration of $D$, both on $S_{\id}$ and $S_\alpha$. Moreover, since $r$ and $s$ are coprime, $a = s$ is the smallest possible integer such that $r_a = 0$, in which case $b_a = r$ is forced, because otherwise we would have $ar = b_a s < rs$, for some $a < s$, contradicting $\mr{lcm}(r, s) = rs$. This means that the word we found for $D_r$ in $F, V^{-1}$ has $a + b_a - 1 = r + s - 1 = n - 1$ letters (we need not take the image via $F$ once we hit $r_a = 0$, hence the $-1$), both on $S_{\id}$ and $S_\alpha$. Also, each intermediate word $\mc W$ gives an interval $D_i = \mc W(0), 1 \leq i < n$, and these cannot be repeated, since otherwise the process described above would enter a loop and not terminate. In particular, each interval subspace $D_i, 1 \leq i < n, i \neq r,$ is given by a proper subword of the minimal word for $D_r$ we found and by \Cref{lem-int}.(\ref{lem-int4}), the same subword gives the same interval subspace on both $S_{\alpha}$ and $S_{\id}$. Recall that $\overline{S}_\alpha = S_\alpha \sqcup S_{\id}$. We have proved \Cref{prop-extcanfilt-exa}.
\begin{proposition}\label{prop-extcanfilt-exa}
We have the following.
\begin{enumerate}
    \item Both on $S_{\id}$ and $S_\alpha$, in the standard bases of \Cref{propstdb}, the canonical flag of $D$ is given by the full standard flag $D_1 \subseteq D_2 \subseteq \ldots \subseteq D_n = D$.
    \item \label{prop-extcanfilt-exa2} The canonical flags can be given on both strata by the same words in $F, V^{-1}$.
\end{enumerate}
In particular, from (\ref{prop-extcanfilt-exa2}), we deduce that the canonical filtration of $A[p]_{\nu}/S_\alpha$ extends to a filtration by finite locally free group subschemes over $\overline{S}_\alpha$ and corresponds to a full filtration by subbundles of $H^1_{\mr{dR}, {\nu}}(A/\overline{S}_\alpha)$ given by the words in
\[
    F((\cdot)^{(p)}), V^{-1}((\cdot)^{(p)})
\]
described above applied to $(0) \subseteq H^1_{\mr{dR}, {\nu}}(A/\overline{S}_\alpha)$.
\end{proposition}
The last statement follows from standard results and the fact that the ranks of the coherent sheaves involved are constant over $\overline{S}_\alpha$. This extension of the canonical filtration, in turn, gives a section over \(\overline{S}_\alpha = S_\alpha\sqcup S_{\id}\) of the natural projection from \(\Fcal_{\overline{S}_\alpha}^{(P_w)} \subseteq \Fcal_S^{(P_w)}\), the flag space over \(\overline{S}_\alpha\) parametrizing full flags in \(H^1_{\mr{dR}, {\nu}}(A/\overline{S}_\alpha)\) refining the Hodge filtration (meaning that, in the notation of \ref{sec-canonical}, \(P_w = B\)).

We saw in Theorem \ref{theorem: thmA intro} that the existence of a zip-theoretic extension of the canonical torsor implies that \(\overline{S}_\alpha\) is normal, hence smooth (because it is one-dimensional). This extension is exactly the one constructed explicitly in \Cref{prop-extcanfilt-exa}. We will see how \Cref{thm-extnorm} generalizes this picture to arbitrary reductive groups and EO-strata of any dimension.

\section{Review of \texorpdfstring{$G$}{G}-zips}
\label{sec-revgzip}

\subsection{Setup}\
We follow the notation set up in \S\nameref{sec-notation}.

\subsubsection{Zip data} \label{sec-zip-data} 
A \emph{zip datum} is a tuple $\Zcal=(G,P,Q,L,M)$ where $G$ is a connected, reductive group over $\FF_p$, $P,Q$ are parabolic subgroups of $G_k$, with Levi subgroups $L,M$, respectively, satisfying $M=L^{(p)}$. A natural way to define a zip datum is to start with a connected, reductive group $G$ over $\FF_p$ and a cocharacter $\mu\colon \GG_{m,k}\to G_k$. Denote by $P_{\pm}$ the pair of opposite parabolic subgroups naturally determined by $\mu$ (see for example \cite[\S 1.2]{Goldring-Koskivirta-zip-flags}). The centralizer of $\mu$ is the common Levi subgroup $L\colonequals P_-\cap P_+$. Set $P\colonequals P_-$, $Q\colonequals P_+^{(p)}$, and $M\colonequals L^{(p)}$. Then the tuple $\Zcal_\mu\colonequals (G,P,Q,L,M)$ is a zip datum.

\subsubsection{The stack of \texorpdfstring{$G$}{G}-zips}\label{section: the stack of gzips}
Let $\Zcal=(G,P,Q,L,M)$ be a zip datum. 
The \emph{zip group} $E_{\Zcal}\subseteq P\times Q$ is defined by
\begin{equation}
    E_{\Zcal}\colonequals \{(x,y)\in P\times Q, \ \varphi(\overline{x})=\overline{y}\},
\end{equation}
where $\overline{x}\in L$, $\overline{y}\in M$, denote the Levi projections of $x\in P$ and $y\in Q$, respectively.

We let $E_{\Zcal}$ act on $G_k$ by the rule $(x,y)\cdot g \colonequals xgy^{-1}$ for all $(x,y)\in E_{\Zcal}$ and $g\in G$.
The \emph{stack of $G$-zips of type $\Zcal$} of Pink--Wedhorn--Ziegler can be defined as the quotient stack
\begin{equation}
    \GZip^\Zcal \colonequals \left[ E_{\Zcal}\backslash G_k \right].
\end{equation}
An equivalent definition in terms of torsors was given in \S\ref{intro-subsec-Gzip}. As we will repeatedly work with the stack of $G$-zips, we use for simplicity the notation
\[
    \Xcal^{\Zcal}\coloneqq \GZip^\Zcal.
\]
When the zip datum $\Zcal$ is clear from context, we omit it from notation and write simply $\Xcal$. 
The stack of $G$-zips of type $\Zcal$ is a smooth stack over $k$ whose underlying topological space is finite. When $\Zcal=\Zcal_\mu$ for a cocharacter $\mu\colon \GG_{m,k}\to G_k$, we write $\Xcal^\mu$ for $\Xcal^{\Zcal_\mu}$. 

\subsubsection{Frames}
Let $\Zcal$ be a zip datum. Let $(B,T)$ be a Borel pair of $G$ and $z\in W$. We call $(B,T,z)$ a \emph{frame} for $\Zcal$ if the following conditions are satisfied.
\begin{enumerate}[(i)]
\item $B,T$ are defined over $\FF_p$.
\item $B\subseteq P$ and $T\subseteq L$.
\item ${}^z \!B\subseteq Q$.
\item \label{item-Wframe4} $B\cap M =  {}^z \!B \cap M$.
\end{enumerate}
After possibly conjugating $\Zcal$ by an element of $G(k)$, a frame always exists (\cite[Remark 1.4.1]{wedhorn.koskivirta}). We will therefore assume throughout that a frame $(B,T,z)$ for $\Zcal$ exists. We follow the notation of \ref{not-redWeyl} for \(G, P, B, T\), as defined here. We will also write $I\colonequals \Delta_L$. We call $I$ the set of \emph{compact simple roots}, and $\Delta \setminus I$ the set of \emph{non-compact simple roots}. Let $J\subseteq \Delta$ denote the type of the parabolic subgroup $Q$. Since $(B,T)$ is defined over $\FF_p$, the Galois group $\Gal(k/\FF_p)$ acts on the sets $X^*(T)$, $\Phi$, $\Phi_+$, $\Delta$. Our conventions imply that
\begin{equation}\label{J-eq}
    J={}^{z^{-1}}\sigma(I).
\end{equation}

Let $\mu\colon {\GG}_{m,\textrm{k}}\to G_k$ be a cocharacter and suppose $(B,T)$ is a Borel pair defined over $\FF_p$ such that $\mu$ factors through $T$. In the case when $\Zcal=\Zcal_\mu$, we always set
\begin{equation}
    z\colonequals \sigma(w_{0,I}) w_0.
\end{equation}
The triple $(B,T,z)$ is then a frame for $\Zcal$.

\subsubsection{Parametrization of strata}\label{sec-param}
Let $\Zcal$ be a zip datum with frame $(B,T,z)$. The points of the underlying topological space of $\Xcal^{\Zcal}$ correspond to $E_\Zcal$-orbits in $G_k$. Any point of $\Xcal^\Zcal$ is locally closed. The pointwise decomposition of $\Xcal^\Zcal$ is the \emph{zip stratification} of $\Xcal^\Zcal$ and each point (viewed as a locally closed substack endowed with the reduced structure) is called a \emph{zip stratum}. We now explain the parametrization of these strata.

For $w\in W$, fix a representative $\dot{w}\in N_G(T)$, such that $(w_1w_2)^\cdot = \dot{w}_1\dot{w}_2$ whenever $\ell(w_1 w_2)=\ell(w_1)+\ell(w_2)$ (this is possible by choosing a Chevalley system \cite[XXIII, \S6]{sga3III}). When no confusion occurs, we write $w$ instead of $\dot{w}$. For $w\in {}^I W$, or $w\in W^J$, write $G_w\colonequals E_{\Zcal}\cdot (wz^{-1})$ for the $E_{\Zcal}$-orbit of $wz^{-1}$. The map $w\mapsto G_w$ induces two bijections ${}^IW \to \{E_{\Zcal}\textrm{-orbits in }G_k\}$ and $W^J \to \{E_\Zcal\textrm{-orbits in }G_k\}$. Furthermore, $G_w$ is a locally closed subset of $G_k$ of dimension $\dim(G_w)=\ell(w)+\dim(P)$. For $w\in {}^I W$ or $w\in W^J$, we denote by $\Xcal^{\Zcal}_w$ (or simply $\Xcal_w$) the locally closed substack $[E_\Zcal \backslash G_w]$, it is the \emph{zip stratum attached to $w$}. Since $G_w$ identifies with the quotient $E/\Stab_E(wz^{-1})$, we have an isomorphism of $k$-stacks
\begin{equation}\label{stab-isom}
    \Xcal_w \simeq \left[ \Stab_E(wz^{-1}) \backslash 1 \right].
\end{equation}

\subsubsection{Zip stratification}\label{section: definition zip stratification}
If $\Scal$ is a $k$-stack endowed with a morphism $\gamma \colon \Scal\to \Xcal^\Zcal$, the stratification obtained by pullback via $\gamma$ will be called the \emph{$\Zcal$-zip stratification}. For $w$ in $\IW$ (or $W^{J}$) we denote by $\Scal_w$ the $\Zcal$-zip stratum associated to $w$. For instance, the Ekedahl--Oort stratification is the $\Zcal_\mu$-zip stratification given by the morphism \eqref{eq-zeta-shimura}.

\subsubsection{Bruhat stratification}\label{sec-Bruhat}
Let $P_1,P_2\subseteq G_k$ be standard parabolic subgroups. The Bruhat stratification on $G$ induces a stratification on $[(P_1\times P_2)\backslash G_k]$, also called the Bruhat stratification (see \cite{wedhorn.bruhat} for details). If $I_1, I_2\subseteq \Delta$ denote the types of $P_1$ respectively $P_2$, then the Bruhat stratification on $[(P_1\times P_2)\backslash G_k]$ is parametrized by ${}^{I_1}W^{I_2}\coloneqq {}^{I_1}W\cap W^{I_2}$. If $\Scal$ is a $k$-stack endowed with a map $\beta\colon \Scal \to [(P_1\times P_2)\backslash G_k]$, then we call the induced stratification on $\Scal$ the \textit{$(P_1\times P_2)$-Bruhat stratification}. When $P_1,P_2$ are implicitly understood, we simply refer to it as the Bruhat stratification on $\Scal$. We denote a Bruhat stratum associated to $w$ by $\Scal_w^{\Br}$.

The inclusion $E_\Zcal\subseteq P\times Q$ yields a natural projection map $\Xcal^\Zcal=[E_\Zcal \backslash G_k]\to [P\times Q \backslash G_k]$. The map $g\mapsto gz$ induces an isomorphism $[P\times Q \backslash G_k]\simeq [P\times {}^{z^{-1}}Q \backslash G_k]$, where $P$ and ${}^{z^{-1}}Q$ are standard parabolics. Let $\beta^\Zcal \colon \Xcal^\Zcal\to [P\times {}^{z^{-1}}Q \backslash G_k]$ be the composition. In this paper we refer to \textit{the Bruhat stratification on $\Xcal^\Zcal$} as the one induced by $\beta^\Zcal$. We denote by ${\Xcal}^{\Zcal, \Br}_w$ the Bruhat stratum of $\Xcal^{\Zcal}$ for $w\in {}^I W^J$. This stratification is coarser than the zip stratification of $\Xcal^\Zcal$. Explicitly, we have ${\Xcal}^{\Zcal, \Br}_w=\left[ E_\Zcal \backslash Pwz^{-1}Q \right]$, and the Bruhat stratification on $\Xcal^\Zcal$ corresponds to the decomposition of $G_k$ into $P\times Q$-orbits.
\begin{definition} \label{def bruhat weakly bruhat}
Let $\Scal$ be a $k$-stack endowed with a surjective map $\gamma\colon \Scal\to \Xcal^{\Zcal}$. Consider the zip and Bruhat stratifications on $\Scal$ induced by $\gamma$, respectively $\beta^\Zcal \circ \gamma$. Let $w\in {}^I W$ be an element.
\begin{enumerate}
    \item We say that the zip stratum $\Scal_w$ is weakly Bruhat if it is open in some Bruhat stratum ${\Scal}^{\Br}_{w'}$.
    \item We say that the zip stratum $\Scal_w$ is Bruhat if it coincides with a Bruhat stratum $\Scal^{\Br}_{w'}$ (in this case $w'=w$).
\end{enumerate}
\end{definition}
In particular, this definition applies in the case when $\Scal=\Xcal^\Zcal$ and $\gamma=\id_{\Xcal^{\Zcal}}$. In this case, we also say that the element $w\in {}^I W$ itself is weakly Bruhat (resp.\@ Bruhat) if the zip stratum $\Xcal^{\Zcal}_w$ is. Equivalently, $w$ is weakly Bruhat if and only if the corresponding $E_{\Zcal}$-orbit $E_{\Zcal}\cdot (wz^{-1})$ is open in the $P\times Q$-orbit of $wz^{-1}$.

\begin{remark}
\label{prop-normCMBruhat}
The Zariski closure of any Bruhat stratum in $G_k$ is normal and Cohen--Macaulay, \cite{ramarama, rama, andersi}. Thus:
\begin{enumerate}
    \item If $\beta\colon \Scal\to [(P_1\times P_2)\backslash G_k]$ is smooth and surjective, the Zariski closure of any Bruhat stratum $\Scal^{\Br}_w$ is normal and Cohen--Macaulay.
\item If the map $\gamma\colon \Scal\to \Xcal^\Zcal$ is smooth and surjective, then the closure of a weakly Bruhat stratum is normal.
\end{enumerate}
\end{remark}

\subsubsection{Zip strata inside a Bruhat stratum}
We recall some results from \cite[\S 4]{Pink-Wedhorn-Ziegler-zip-data}. Let $w\in {}^I W^J$ and define
\begin{align*}
    &P^{(w)}\colonequals M\cap {}^{zw^{-1}}P, \quad  Q^{(w)}\colonequals \varphi(L\cap {}^{wz^{-1}}Q)\\
    &L^{(w)}\colonequals M\cap {}^{zw^{-1}}L, \quad  M^{(w)}\colonequals \varphi(L\cap {}^{wz^{-1}}M). 
\end{align*}
One attaches to $w$ the tuple $\Zcal^{(w)}\colonequals (M,P^{(w)},Q^{(w)},L^{(w)},M^{(w)})$. Note that $M$ may not be defined over $\FF_p$, hence this tuple is not a zip datum in the sense of \S\ref{sec-zip-data} above. However, one can define a more general notion of zip data (and stacks of $G$-zips), as explained in \cite[Definition 1.1]{Pink-Wedhorn-Ziegler-zip-data}. Using this notion, one shows:

\begin{proposition}[{\cite[Proposition 4.7]{Pink-Wedhorn-Ziegler-zip-data}}]
Let $w\in  {}^I W^J$. The underlying topological space of the Bruhat stratum $\Xcal^{\Zcal,\Br}_w$ is homeomorphic to that of the stack $\MZip^{\Zcal^{(w)}}$.
\end{proposition}

Furthermore, the type of the parabolic $P^{(w)}\subseteq M$ is given by $J\cap {}^{w^{-1}} I$, by \loccit Proposition 4.13. From this, it is possible to give an explicit characterization of Bruhat elements:

\begin{corollary}\label{cor-str-Br}
An element $w\in {}^I W$ is Bruhat if and only if ${}^w J= I$.
\end{corollary}

\begin{proof}
The points of $\MZip^{\Zcal^{(w)}}$ are in 1-to-1 correspondence with elements of ${}^{J\cap {}^{w^{-1}}I} W_{J}$, as in \S\ref{sec-param} above. This set is reduced to a point if and only if $J\cap {}^{w^{-1}} I=J$, which is equivalent to $J= {}^{w^{-1}} I$ because $J$ and ${}^{w^{-1}} I$ have the same cardinality. The result follows.
\end{proof}

\subsubsection{Closure relations}\label{subsubsec: closure relations}
Since ${}^z B\subseteq Q$, the group $W_J$ coincides with the Weyl group of the standard Levi subgroup ${}^{z^{-1}}M={}^{z^{-1}}\varphi(L)$. Hence, we have a natural map $\psi\colon W_I\to W_J$ given by $\psi(x)=z^{-1} \varphi(x) z$. For two elements $w,w'\in {}^I W$, define a relation $\preccurlyeq$ by
\begin{equation}
    w' \preccurlyeq w \Longleftrightarrow \exists x\in W_I, \ xw'\psi(x)^{-1}\leq w.
\end{equation}
By \cite[Corollary 6.3]{Pink-Wedhorn-Ziegler-zip-data}, the relation $\preccurlyeq$ is a partial order on ${}^I W$, which is, in general, finer than the Bruhat order $\leq$. For $w\in {}^I W$, denote by $\overline{\Xcal}^{\Zcal}_w$ the Zariski closure of the zip stratum $\Xcal_w^{\Zcal}$. By \loccit Theorem 7.5, we have
\begin{equation}
    \overline{\Xcal}^\Zcal_w = \bigcup_{w'\preccurlyeq w} \Xcal^\Zcal_{w'}.
\end{equation}
The twisted order satisfies \textit{the chain property}: if $w'\preccurlyeq w$, we can find a sequence $w'=w_1 \preccurlyeq w_2 \preccurlyeq \dots \preccurlyeq w_m=w$ such that $\ell(w_{i+1})=\ell(w_i)+1$ for all $1\leq i \leq m-1$.

\subsection{Flag space}
\label{ssecflagspace}
Write simply $E\subseteq P\times Q$ for the zip group of $\Zcal$, and $\Xcal$ for the stack of $G$-zips of type $\Zcal$.

\subsubsection{Definition}\label{subsubsec-flagspace-def}
Let $P_0$ be a parabolic subgroup of $G_k$ satisfying $B\subseteq P_0 \subseteq P$. Denote the type of $P_0$ by $I_0\subseteq I$. Define the flag space $\Fcal^{(P_0)}$, sometimes denoted in previous papers by $\GF^{\Zcal}_{P_0}$, by
\begin{equation*}
    \Fcal^{(P_0)}\colonequals \left[E \backslash \left(G_k\times P/P_0\right)\right],
\end{equation*}
where $(x,y)\in E$ acts on $(g,hP_0)\in G_k\times P/P_0$ by the rule $(x,y)\cdot (g,hP_0) \colonequals (xgy^{-1},xhP_0)$. If we have a morphism \(S \to \Xcal\), then we denote by \(\Fcal^{(P_0)}_S\) the base change of \(\Fcal^{(P_0)}\) from \(\Xcal\) to \(S\). Denote by $\pi_{P_0}\colon\Fcal^{(P_0)}\to \Xcal$ the natural map induced by the first projection $\pr_1\colon G_k\times P/P_0 \to G_k$. For two parabolics $B\subseteq P_0\subseteq P_1\subseteq P$, we have similarly a projection $\pi_{P_0,P_1}\colon \Fcal^{(P_0)}\to \Fcal^{(P_1)}$ (thus $\pi_{P_0}$ identifies with $\pi_{P_0,P}$).

\subsubsection{\texorpdfstring{\(\Zcal_{P_0}\)}{Z0}-zip stratification} \label{sec-fine-stratif}
The flag space $\Fcal^{(P_0)}$ admits a so-called \emph{\(\Zcal_{P_0}\)-zip stratification} as in \S\ref{section: definition zip stratification}. This was originally defined in \cite[\S 3.1]{Goldring-Koskivirta-zip-flags} and called the fine stratification. We opt to change terminology to make it more uniform. The \(\Zcal_{P_0}\)-zip stratification generalizes the zip stratification on the stack of $G$-zips from \ref{sec-param}. 

First, note that the inclusion map $G_k\to G_k\times P/P_0$, $x\mapsto (x,1)$, induces an isomorphism
\begin{equation}\label{psiP0}
    \Fcal^{(P_0)}\simeq [E'_{P_0}\backslash G_k]
\end{equation}
where $E'_{P_0}\colonequals E\cap (P_0\times G)$. Let $L_0\subseteq L$ denote the Levi subgroup of $P_0$ containing $T$, and set $M_0\colonequals L_0^{(p)}$ and $Q_0\colonequals M_0 {}^z B$. Then $Q_0$ is a parabolic subgroup of $G_k$ with Levi subgroup $M_0$. The tuple $\Zcal_{P_0}\colonequals (G,P_0,Q_0,L_0,M_0)$ is a zip datum (in general, not attached to a cocharacter of $G$), and the tuple $(B,T,z)$ is again a frame for $\Zcal_{P_0}$. The zip datum $\Zcal_{P_0}$ gives rise to a stack $\Xcal^{\Zcal_{P_0}}\colonequals [E_{\Zcal_{P_0}}\backslash G_k]$. One sees immediately that $E'_{P_0}\subseteq E_{\Zcal_{P_0}}$, thus we have a natural, smooth projection map
\begin{equation}\label{psimap}
   \Psi_{P_0}\colon \Fcal^{(P_0)} \to \Xcal^{\Zcal_{P_0}}.
\end{equation}
The \textit{$\Zcal_{P_0}$-zip stratification on $\Fcal^{(P_0)}$} is the zip stratification afforded by $\Psi_{P_0}$ (\S\ref{section: definition zip stratification}). 

By \ref{sec-param}, the points of $\Xcal^{\Zcal_{P_0}}$ are parametrized by ${}^{I_0} W$ (the element $w\in {}^{I_0} W$ corresponds to the $E_{\Zcal_{P_0}}$-orbit of $wz^{-1}$ in $G_k$). Since $Q_0=\varphi(L_0) \ {}^{z}B$, the type of $Q_0$ is
\begin{equation}\label{J0-eq}
J_0=w_0 \sigma(w_{0,I}\cdot I_0).    
\end{equation}
We could also parametrize the points of $\Xcal^{\Zcal_{P_0}}$ by $W^{J_0}$ instead of ${}^{I_0} W$, as in \S\ref{sec-param}. For $w\in {}^{I_0}W$ or $w\in W^{J_0}$, define $\Fcal_{w}^{(P_0)}$ to be the preimage by $\Psi_{P_0}$ of $\Xcal^{\Zcal_{P_0}}_w$. We call $\Fcal_{w}^{(P_0)}$ the \(\Zcal_{P_0}\)-zip stratum of $w\in {}^{I_0} W$. For $w\in {}^{I_0}W$ or $w\in W^{J_0}$, let $C_{P_0,w}\colonequals E_{\Zcal_{P_0}}\cdot (wz^{-1})$ be the $E_{\Zcal_{P_0}}$-orbit of $wz^{-1}$. By definition of $\Fcal_{w}^{(P_0)}$, we have $\Fcal_{w}^{(P_0)}=\left[E'_{P_0} \backslash C_{P_0,w} \right]$.

\subsubsection{Bruhat stratification on \texorpdfstring{$\Fcal^{(P_0)}$}{the flag space}}
In this paper, we refer to \textit{the Bruhat stratification\footnote{This stratification was termed coarse in \cite{Goldring-Koskivirta-zip-flags}.} on $\Fcal^{(P_0)}$} as the one afforded by $\Psi_{P_0}$ (\S\ref{sec-Bruhat}). As per \Cref{def bruhat weakly bruhat}, we have a notion of Bruhat and weakly Bruhat $\Zcal_{P_0}$-zip strata. In the case $P_0=B$, the \(\Zcal_B\)-zip and Bruhat stratifications of $\Fcal^{(B)}$ coincide. Thus any \(\Zcal_B\)-zip stratum of $\Fcal^{(B)}$ is Bruhat. The stratum parametrized by $w\in W$ is simply $\Fcal^{(B)}_{w}=[E'_B \backslash BwBz^{-1}]$. 

\subsubsection{Image of strata}
For a general $w\in {}^{I_0}W$ or $w\in W^{J_0}$, the image of $\Fcal_{w}^{(P_0)}$ by the projection $\pi_{P_0}\colon \Fcal^{(P_0)}\to \Xcal$ is a union of certain zip strata of $\Xcal$, but it is difficult to determine which zip strata appear in this union. However, for $w\in {}^I W$ or $w\in W^J$, we have the following result.

\begin{proposition}\label{prop-flag}
Let $w\in {}^I W$ or $w\in W^J$. 
\begin{enumerate}
    \item \label{prop-flag-1} We have $\pi_{P_0}(\Fcal_{w}^{(P_0)})=\Xcal_w$ and $\pi_{P_0}$ restricts to a finite \'{e}tale morphism $\pi_{P_0}\colon \Fcal_{w}^{(P_0)}\to \Xcal_{w}$.
    \item \label{prop-flag-2} The preimage of $\Xcal_w$ by the map $\pi_{P_0} \colon \overline{\Fcal}^{(P_0)}_w \to \overline{\Xcal}_w$ coincides with $\Fcal^{(P_0)}_w$.
    \item \label{prop-flag-3} The set $C_{P_0,w}$ coincides with the $E'_{P_0}$-orbit of $wz^{-1}$. Hence, we have a commutative diagram
    $$
    \xymatrix@1@M=5pt{
    \Fcal_{w}^{(P_0)} \ar[r]^-{\simeq} \ar[d] & [\Stab_{E'_{P_0}}(wz^{-1})\backslash 1] \ar[d] \\
    \Xcal_{w} \ar[r]^-{\simeq} & [\Stab_{E}(wz^{-1})\backslash 1]
    }
    $$
\end{enumerate}
\end{proposition}

\begin{proof}
Statements \eqref{prop-flag-1} and \eqref{prop-flag-3} are proved in \cite[Proposition 2.2.1]{Koskivirta-Normalization}. Statement \eqref{prop-flag-2} can easily be reduced to the case $P_0=B$, which was shown in \cite[Proposition 2.4.3 (c)]{koskgold}. Alternatively, it follows easily from the inequality \eqref{ineq-ell} below.
\end{proof}

\begin{remark}\label{intermediate-rmk}
For the sake of simplicity, we only stated the above proposition for the map $\pi_{P_0}$. There is a similar result for the case of $\pi_{P_0,P_1}\colon \Fcal^{(P_0)}\to \Fcal^{(P_1)}$ where $B\subseteq P_0\subseteq P_1\subseteq P$ (see \cite[Proposition 4.2.2]{Goldring-Koskivirta-zip-flags}).
\end{remark}

In the case $P_0=B$ and $w\in W$, \cite[Lemma 2.4.1]{koskgold} shows that $\pi_B(\Fcal^{(B)}_w)$ coincides with the union of zip strata $\Xcal_{w'}$ satisfying the condition 
\begin{equation}\label{eq-nonempty}
E\cdot (w'z^{-1})\cap B w z^{-1} \neq \emptyset.    
\end{equation}
This result generalizes to any parabolic $B\subseteq P_0\subseteq P$. Indeed, let $w\in {}^{I_0} W$. By \cite[Proposition 3.2.2]{Goldring-Koskivirta-zip-flags}, the projection $\pi_{B,P_0}\colon \Fcal^{(B)} \to \Fcal^{(P_0)}$ maps $\Fcal^{(B)}_w$ surjectively to $\Fcal^{(P_0)}_w$. Therefore $\pi_{P_0}(\Fcal^{(P_0)}_w) = \pi_{B}(\Fcal^{(B)}_w)$. Hence this image can be described as the union of strata satisfying \eqref{eq-nonempty}.

\subsubsection{Lower neighbors} \label{subsec-lownb}
Let $w,w'\in {}^I W$ be two elements. Recall that $\Xcal_{w'}$ is contained in the closure of $\Xcal_w$ if and only if $w' \preccurlyeq w$. This condition is in general weaker than $w'\leq w$. Note that if $I_0\subseteq I$ then ${}^I W\subseteq {}^{I_0} W$, hence the notation $\preccurlyeq$ is slightly ambiguous (the set ${}^{I_0} W$ also comes with its own twisted order). To avoid this problem, we write $\preccurlyeq_I$ for the twisted order of ${}^I W$. Note that the order $\preccurlyeq_{\emptyset}$ on $W={}^{\emptyset} W$ is simply the Bruhat order $\leq$.

\begin{definition} Let $w,w'\in {}^I W$.
\begin{enumerate}
\item When $w' \preccurlyeq_I w$ and $\ell(w')=\ell(w)-1$, we say that $w'$ is a lower neighbor of $w$. 
\item When $w' \leq w$ and $\ell(w')=\ell(w)-1$, we say that $w'$ is a Bruhat-lower neighbor of $w$.  
\item We say that $w'$ is an exceptional lower neighbor of $w$ if it is a lower neighbor that is not a Bruhat-lower neighbor.
\end{enumerate}
\end{definition}

For a subset $I\subseteq \Delta$ and an element $w\in {}^I W$, write $\Gamma_I(w)$ for the set of lower neighbors of $w$ in ${}^I W$, i.e.\@
\begin{equation}
    \Gamma_I(w) \colonequals \{w'\in {}^I W \ | \ w'\preccurlyeq_I w \textrm{ and }\ell(w')=\ell(w)-1\}.
\end{equation}
In particular, the set $\Gamma_{\emptyset}(w)$ is the set of Bruhat-lower neighbors in $W$ of $w$. Note that when $w\in {}^I W$, we have $\Gamma_{\emptyset}(w)\cap {}^I W\subseteq \Gamma_I(w)$, but the equality does not always hold. 

For an intermediate parabolic $B\subseteq P_0\subseteq P$, consider the map $\pi_{P_0}\colon \Fcal^{(P_0)}\to \Xcal$. For $w\in {}^{I_0} W$, recall that the image $\pi_{P_0}(\Fcal^{(P_0)}_w)$ by $\pi_{P_0}$ of the \(\Zcal_{P_0}\)-zip stratum $\Fcal^{(P_0)}_w$ is a union of zip strata of $\Xcal$. Since $\pi_{P_0}(\Fcal^{(P_0)}_w)$ is irreducible, there is a unique zip stratum $\Xcal_{w_1}$ which is open in $\pi_{P_0}(\Fcal^{(P_0)}_w)$. By abuse of notation, we denote this element $w_1\in {}^I W$ by $\pi_{P_0}(w)$. We obtain a map $\pi_{P_0}\colon {}^{I_0}W\to {}^I W$ which extends the identity map of ${}^I W$. Note that for dimension reasons, for any $w\in {}^{I_0} W$ we have
\begin{equation} \label{ineq-ell}
    \ell(\pi_{P_0}(w)) \leq \ell(w).
\end{equation}

\begin{lemma}\label{lemma-low-nei}
Let $I_0\subseteq I$. For any $w\in {}^I W$, we have
\begin{enumerate}
    \item \label{lemma-low-nei-item1} $\Gamma_{I_0}(w)\cap {}^I W \subseteq \Gamma_{I}(w)$.
    \item \label{lemma-low-nei-item2} $\Gamma_I(w)\subseteq \pi_{P_0}(\Gamma_{I_0}(w))$.
\end{enumerate}
\end{lemma}

\begin{proof}
The first statement is clear. By properness, the image of the map $\pi\colon \overline{\Fcal}_w^{(P_0)}\to \Xcal$ contains all of \(\overline{\Xcal}_w\). Furthermore, the preimage of $\Xcal_w$ in $\overline{\Fcal}_w^{(P_0)}$ is exactly $\Fcal^{(P_0)}_w$. Also, for any $w'\in {}^{I_0} W$ such that $\ell(w')<\ell(w)-1$, inequality \eqref{ineq-ell} shows that $\pi_{P_0}(w')$ is not a lower neighbor of $w$. Therefore the preimage of any stratum $\Xcal_{w'}$ for $w'\in \Gamma_I(w)$ in $\overline{\Fcal}_w^{(P_0)}$ is entirely contained in the union of \(\Zcal_{P_0}\)-zip strata corresponding to elements of $\Gamma_{I_0}(w)$. The result follows.
\end{proof}

Similarly to Remark \ref{intermediate-rmk}, Lemma \ref{lemma-low-nei} above generalizes to the case of two intermediate parabolics $B\subseteq P_0\subseteq P_1 \subseteq P$. If $I_0\subseteq I$ and $w\in {}^I W$, it is not true in general that $\Gamma_{I}(w)\subseteq \Gamma_{I_0}(w)$ (for example, it may fail for $I_0=\emptyset$). However, we have the following.

\begin{corollary}\label{cor-same-neighbors}
Let $I_0\subseteq I$, $w\in {}^I W$ and assume that $\Gamma_{I_0}(w)\subseteq {}^I W$. Then $\Gamma_{I}(w)=\Gamma_{I_0}(w)$.
\end{corollary}

\begin{proof}
By assumption $\Gamma_{I_0}(w)\subseteq {}^I W$, hence $\Gamma_{I_0}(w)\subseteq \Gamma_{I}(w)$. Conversely, we have $\Gamma_I(w)\subseteq \pi_{P_0}(\Gamma_{I_0}(w))$ by \Cref{lemma-low-nei}. By assumption $\Gamma_{I_0}(w)\subseteq {}^IW$, so $\pi_{P_0}(\Gamma_{I_0}(w))=\Gamma_{I_0}(w)$. The result follows.
\end{proof}

\begin{definition}\label{def-Ew}
For $w\in W$, denote by $E_w$ the set of $\alpha\in \Phi^+$ such that $ws_\alpha \leq w$ and $\ell(ws_\alpha)=\ell(w)-1$.
\end{definition}

Since any Bruhat-lower neighbor of $w$ can be written in the form $ws_\alpha$, $\alpha\in \Phi_+$, the map $\varepsilon_w \colon E_w \to \Gamma_{\emptyset}(w)$, $\alpha\mapsto ws_\alpha$ is a bijection.

\subsubsection{Canonical parabolic}\label{sec-canonical}

For $w\in {}^I W$, we attach to $w$ a standard parabolic subgroup $P_w$ of $G$ following \cite[\S 5.1]{Pink-Wedhorn-Ziegler-zip-data}. This reference uses the convention $B\subseteq Q$, whereas we have assumed $B\subseteq P$. We have transcribed below all results and definitions using our convention. First, there exists a largest subgroup $L_w\subseteq L$ such that $L_w= {}^{wz^{-1}}\varphi(L_w)$ (as subgroups of $G_k$). Furthermore, $L_w$ is a standard Levi subgroup of $G$ contained in $L$ (\loccit Proposition 5.6). Define a parabolic subgroup $P_w\subseteq G$ by $P_w\colonequals L_w B$. 
\begin{definition}
We call $P_w$ the \emph{canonical parabolic} of $w$.
\end{definition}
The parabolic $P_w$ is a standard parabolic of type $I_w$, where $I_w \subseteq I$ is the largest subset $I_0\subseteq I$ satisfying $I_0={wz^{-1}}\cdot \sigma(I_0)$ (viewed as subsets of $\Phi$). Denote by $\varphi_w$ the map $\varphi_w\colon \Phi\to \Phi$, $\alpha\mapsto (wz^{-1})\cdot\sigma(\alpha)$. We deduce:
\begin{equation}\label{eq-Iw}
    I_w = \bigcap_{m\geq 0} \varphi^m_w(I).
\end{equation}
Since $\varphi_w$ is a bijection of $\Phi$, we may also take intersections for $m\geq m_0$ for any integer $m_0$. In the case when the Galois action is trivial, equation \eqref{eq-Iw} boils down to $I_w = \bigcap_{m\geq 0} (wz^{-1})^m(I)$.

\begin{remark}
In \cite[Proposition 4.3.1]{Koskivirta-Normalization}, it was checked (in the case of a symplectic group $G=\Sp_{2n}$) that $P_w$ coincides with the parabolic subgroup stabilizing the canonical filtration of the Dieudonn\'{e} space corresponding to $w$. A similar description can be given in other cases, when the group $G$ is given as the automorphism group of a PEL datum (for example if $G=\GL_{n,\FF_p}$, or a unitary group). In these cases, as we will see in \ref{ssec-fzipcanfilt}, the stack $\Xcal$ admits a natural moduli interpretation in terms of certain filtrations endowed with Frobenius and Verschiebung endomorphisms $F,V$, compatible with the PEL structure. The parabolic $P_w$ is then the parabolic stabilizing the canonical filtration, i.\@e.\ the maximal filtration stable under $F^{-1}$ and $V$ (and the PEL structure).
\end{remark}
An equivalent formulation is the following. Define $I^*=I\cup \{0\}$, $\Delta^*=\Delta\cup \{0\}$ and two modified maps 
\begin{equation} \label{phistar}
\varphi_{w,I}\colon I^*\to I^*, \quad \varphi_{w,\Delta}\colon \Delta^*\to I^*\subseteq \Delta^*
\end{equation}
by setting $\varphi_{w,I}(\alpha)=\varphi_{w}(\alpha)$ (resp.\ $\varphi_{w,\Delta}(\alpha)=\varphi_{w}(\alpha)$) if $\varphi_{w}(\alpha)\in I$ (resp.\ $\varphi_w(\alpha)\in \Delta$) and $\varphi_{w,I}(\alpha)=\varphi_{w,I}(0)=0$ (resp.\ $\varphi_{w,\Delta}(\alpha)=\varphi_{w,\Delta}(0)=0$) otherwise. Then, we have
\begin{equation}
    I_w=\bigcap_{m\geq 0} \varphi^m_{w}(I) = \bigcap_{m\geq 0} \varphi^m_{w,I}(I) = \bigcap_{m\geq 0} \varphi^m_{w,\Delta}(I).
\end{equation}
A first example is the case when $w=w_{0,I}w_0$ is the longest element of ${}^I W$. From the above discussion, we immediately conclude the following.

\begin{lemma}\label{lem-Pz}
    The canonical parabolic of the longest element $w_{0,I}w_0\in {}^I W$ is the standard parabolic whose type is $-w_{0,I}\left( \bigcap_{n\in \ZZ} \sigma^n(I)\right)$.
    
\end{lemma}

We explain the link between the canonical parabolic and point stabilizers of the $E$-action on $G_k$. For $x\in G$, note that the stabilizer $\Stab_E(x)$ embeds via the first projection $\pr_1\colon E\to P$ as a subgroup of $P$. We will henceforth identify $\Stab_E(x)$ with this subgroup of $P$.

\begin{lemma}
\label{lem-stabinpw}
For $w\in {}^I W$, the stabilizer $\Stab_E(wz^{-1})$ is contained in $P_w$. Moreover, \(P_w\) is the smallest standard parabolic of \(G_k\) containing \(\Stab_E(wz^{-1})\).
\end{lemma}
\begin{proof}
Let $P_0\subseteq P$ be a standard parabolic. By \cite[Lemma 3.2.1]{Koskivirta-Normalization}, the inclusion $P_w\subseteq P_0$ is equivalent to the fact that $\pi_{P_0}\colon \Fcal^{(P_0)}_w\to \Xcal_w$ is an isomorphism. By \loccit Proposition 2.2.1.(3), this is also equivalent to the inclusion $\Stab_E(wz^{-1})\subseteq P_0$. The result follows.
\end{proof}

We prove \Cref{lemma: canonical parabolic is everywhere}, which refines \Cref{lem-stabinpw} by dropping the assumption that the parabolic \(P_0\subseteq P\) be standard. This result will be used in the proof of \Cref{lem-rigidityofsubtorsors}. Let us set up some notation before stating and proving the lemma. For any element \(u\) of \(R_u(B)(A)\), with \(A\) a \(k\)-algebra, we can write
\[
    u = \prod_{\alpha \in \Phi^-} x_\alpha(a_\alpha), \, a_\alpha \in A,
\]
where \(x_\alpha \colon \GG_a \to U_\alpha \subseteq G_k\) denotes the root morphism of the root \(\alpha\). For such a \(u\)
we set \(\supp(u) = \{\alpha \in \Delta \mid a_{-\alpha} \neq 0\}\). Since \(\prod_{\alpha \in -\Delta} x_\alpha \colon \GG_a^{\lvert\Delta\rvert} \to R_u(B) \to R_u(B)/[R_u(B),R_u(B)]\) determines an isomorphism of \(k\)-group schemes, \(\supp(u)\) does not depend on the fixed isomorphism of \(k\)-varieties \(\prod_{\alpha \in \Phi^-} x_\alpha \colon \prod_{\alpha \in \Phi^-} U_\alpha \to R_u(B)\). Moreover, for the same reason, if \(u_1, u_2 \in R_u(B)(A)\) satisfy \(\supp(u_1) \cap \supp(u_2) = \emptyset\), then \(\supp(u_1 u_2) = \supp(u_1) \cup \supp(u_2)\). Recall that, for all \(\alpha \in \Phi\) and \(w \in W\), there is some \(c_\alpha \in k^\times\) such that \({}^wx_\alpha(a)= x_{w(\alpha)}(c_\alpha a)\).

\begin{lemma}\label{lemma: canonical parabolic is everywhere}
Let \(w \in \IW\). Consider \(P_0\) an arbitrary parabolic subgroup of \(G_k\) such that \(\Stab_E(wz^{-1}) \subseteq P_0 \subseteq P\). Then, $B\subseteq P_0$, and hence \(P_w \subseteq P_0\).
\end{lemma}

\begin{proof}
By \Cref{lem-stabinpw}, if $P_0$ is standard, then $P_w\subseteq P_0$. 
By \cite[Proposition 12.2.2]{dignemichel}, to show that \(B \subseteq P_0\), it suffices to construct \(u \in (R_u(B) \cap P_0)(k)\) such that \(\supp(u) = \Delta\). First consider \(\alpha \in I_w\). The \(\varphi_w\)-orbit \(O(\alpha)\) of \(-\alpha\) is entirely contained in \(-I_w \subseteq -I\). Consider \(U_{O(\alpha)}\) the subgroup of \(R_u(B)\) generated by the \(U_\beta, \beta \in O(\alpha)\). By abuse of notation, let us also write \(\varphi_w \colonequals {}^{wz^{-1}}\varphi \colon G_k \to G_k\). Notice that \(\varphi_w\) induces an isogeny of \(U_{O(\alpha)}\) into itself, which preserves \([U_{O(\alpha)}, U_{O(\alpha)}]\). In particular, we have an isogeny
\[
    \varphi_w \colon U_{O(\alpha)}/[U_{O(\alpha)}, U_{O(\alpha)}] \longrightarrow U_{O(\alpha)}/[U_{O(\alpha)}, U_{O(\alpha)}] \cong \GG_a^{O(\alpha)}.
\]
We show that this map has a non-trivial fixed point. For each $\beta\in O(\alpha)$, let $a_\beta\in k$ and set \[\overline{u}_{O(\alpha)} = \prod_{\beta \in O(\alpha)} x_\beta(a_\beta) [U_{O(\alpha)}, U_{O(\alpha)}].\]
Then
\[
    \varphi_w(\overline{u}_{O(\alpha)}) = \prod_{\beta \in O(\alpha)} x_{\varphi_w(\beta)}(c_\beta a^p_\beta) [U_{O(\alpha)}, U_{O(\alpha)}].
\]
Imposing the relation \(\varphi_w(\overline{u}_{O(\alpha)}) = \overline{u}_{O(\alpha)}\) gives rise to a series of equations relating the \(a_\beta, \beta \in O(\alpha)\), and it is clear that the \(a_\beta\) can be chosen in \(k^\times\). Lang's theorem implies that we can lift \(\overline{u}_{O(\alpha)}\) to an element \(u_{O(\alpha)} \in U_{O(\alpha)}\) such that \(\supp(u_{O(\alpha)}) = O(\alpha)\) and \(\varphi_w(u_{O(\alpha)}) = u_{O(\alpha)}\). Since we can decompose \(I_w\) into a finite disjoint union of \(\varphi_w\)-orbits, taking products of elements like \(u_{O(\alpha)}\), we obtain \(u_{I_w}\) such that \(\varphi_w(u_{I_w}) = u_{I_w}\) and \(\supp(u_{I_w}) = I_w\).

Consider now \(\alpha \in I \setminus I_w\). Let \(i \ge 1\) be the smallest positive integer such that \(\varphi_w^{-i}(-\alpha) \in \Phi \setminus \Phi_I\). Write \(\beta = \varphi_w^{-i}(-\alpha)\). Write \(m \ge 2\) for the smallest positive integer such that \(\varphi_w^m(\beta) \in \Phi \setminus \Phi_I\). By \cite[Lemma~3.6.1]{koskivirta2026decidabilitysingularitiesekedahloortstratification}, \(\beta\) is positive and \(\varphi_w^j(\beta) \in \Phi_I^-\) for all \(1 \le j \le m-1\).
For each \(\beta \in \Phi^+ \setminus \Phi_I\), consider
\begin{align*}
    Y_\beta &= x_{\sigma(\beta)}(a) x_{\sigma(\varphi_w(\beta))}(b_1 a^p) x_{\sigma(\varphi_w^2(\beta))}(b_2 a^{p^2}) \ldots x_{\sigma(\varphi_w^{m-1}(\beta))}(b_{m-1}a^{p^{m-1}}),\\
    X_\beta &= {}^{w z^{-1}}Y_\beta = \\
    &= x_{\varphi_w(\beta)}(c_{\sigma(\beta)}a) x_{\varphi_w^2(\beta)}(c_{\sigma(\varphi_w(\beta))} b_1 a^p) 
    \ldots x_{\varphi_w^{m}(\beta)}( c_{\sigma(\varphi^{m-1}_w(\beta))}b_{m-1} a^{p^{m-1}}),
\end{align*}
with \(b_i \in k^\times\), such that \(b_0 = 1, b_{i+1} = c_{\sigma(\varphi_w^i(\beta))}^p b_i, 0 \le i \le m-2\).
If \(\varphi_w^m(\beta)\) is positive, we pick \(a = t \in k[t]/(t^{p^m})\). If \(\varphi_w^m(\beta)\) is negative, we can take any \(a \in k^\times\). Then, \((X_\beta, Y_\beta) \in \Stab_E(w z^{-1})\), \(X_\beta \in R_u(B)\) and \(\alpha \in \supp(X_\beta)\). If \(\varphi_w^m(\beta)\) is positive, since \(P_0\cap R_u(B)\) is smooth, hence reduced, and isomorphic to a subvariety of \(\GG_a^{\lvert \Phi^-\rvert}\), there is some \(u_{\beta} \in (P_0 \cap R_u(B))(k)\) such that \(\supp(u_{\beta}) = \supp(X_\beta)\). If \(\varphi_w^m(\beta)\) is negative, take \(u_\beta = X_\beta\). Write \(u_{\rm{nc}} \coloneqq \prod_{\beta \in \Phi^+ \setminus \Phi_I} u_{\beta}\). Then, \(I \setminus I_w \subseteq \supp(u_{\rm{nc}}) \subseteq \Delta \setminus I_w\) and \(I \subseteq \supp(u_{I_w} u_{\rm{nc}}) \subseteq \Delta\). Since \(P_0 \subseteq P\), we have \(R_u(P) \subseteq R_u(P_0)\); hence we can find \(u \in R_u(P) \subseteq P_0\) such that \(\supp(u u_{I_w} u_{\rm{nc}}) = \Delta\). This completes the proof.
\end{proof}

As mentioned in the proof of \Cref{lem-stabinpw}, the canonical parabolic is the smallest standard parabolic subgroup $P_0$ of \(P\) such that the map $\pi_{P_0}\colon \Fcal_{w}^{(P_0)}\to \Xcal_w$ is an isomorphism. It was falsely claimed in \cite{Koskivirta-Normalization} that the canonical parabolic $P_w$ is the largest standard parabolic subgroup $P_0\subseteq P$ such that $\Fcal^{(P_0)}_w$ is a weakly Bruhat stratum of $\Fcal^{(P_0)}$. The correct statement is the following: $P_w$ is the largest standard parabolic $P_0\subseteq P$ such that $\Fcal^{(P_0)}_w$ is Bruhat. Summing up, we have the following proposition.

\begin{proposition}[{\cite[Theorem 3.1.3]{Koskivirta-Normalization}}] \label{prop-minimal-Pw}
Let $w\in {}^I W$. For any parabolic subgroup $B\subseteq P_0\subseteq P$, we have equivalences $(1a) \Leftrightarrow (1b) \Leftrightarrow (1c)$, as well as $(2a) \Leftrightarrow (2b)$:
\begin{enumerate}[\ (1)]
\item[\textnormal{(1a)}] The map $\pi_{P_0}\colon \Fcal_{w}^{(P_0)}\to \Xcal_w$ is an isomorphism.
\item[\textnormal{(1b)}] One has $P_w\subseteq P_0$.
\item[\textnormal{(1c)}] One has $\Stab_E(wz^{-1}) \subseteq P_0$.
\end{enumerate}
\begin{enumerate}[\ (1)]
\item[\textnormal{(2a)}] The stratum $\Fcal_{w}^{(P_0)}$ is Bruhat.
\item[\textnormal{(2b)}] One has $wz^{-1}\sigma(I_0)=I_0$. In particular, one has $P_0\subseteq P_w$ in this case.
\end{enumerate}
\end{proposition}

\begin{remark}\label{rmk-1d}
The proof of \cite[Lemma 3.2.1]{Koskivirta-Normalization} shows a slightly stronger result, namely that conditions (1a), (1b) and (1c) above are also equivalent to the condition:
\begin{enumerate}
\item[\textnormal{(1d)}]  The map $\pi_{P_0}\colon \Fcal_{w}^{(P_0)}\to \Xcal_w$ is bijective.
\end{enumerate}
\end{remark}

Therefore, if we take $P_0=P_w$, the \(\Zcal_{P_w}\)-zip stratum $\Fcal_w^{(P_w)}$ combines all properties listed above. Namely, we obtain the following consequence: 
\begin{proposition}\label{prop-Pw}
For $w\in {}^I W$, the map $\pi_{P_w}\colon \Fcal_w^{(P_w)}\to \Xcal_w$ is an isomorphism and $\Fcal_w^{(P_w)}$ is a Bruhat stratum. In particular, $\overline{\Fcal}_w^{(P_w)}$ is normal and Cohen--Macaulay.
\end{proposition}
We obtain a proper, surjective, birational morphism $\pi_{P_w}\colon \overline{\Fcal}^{(P_w)}_{w}\to \overline{\Xcal}_w$, where $\overline{\Fcal}^{(P_w)}_{w}$ is a normal $k$-stack. Therefore, this map controls to a certain degree the singularities of $\overline{\Xcal}_w$. Apply \Cref{prop-minimal-Pw} to the case $P_0=P$, in which case $\Fcal^{(P)}$ is simply $\Xcal$. We deduce from the equivalence of (2a) and (2b) the following characterization of Bruhat strata:

\begin{proposition}\label{prop-str-Br-Pw}
Let $w\in {}^I W$. The following are equivalent.
\begin{enumerate}
    \item The stratum $\Xcal_w$ is Bruhat.
    \item $P_w=P$.
\end{enumerate}    
\end{proposition}
This result is also consistent with \Cref{cor-str-Br}, because the equality ${}^w J= I$ can be transformed into ${}^{wz^{-1}}\sigma(I)=I$ using \eqref{J-eq}, which says exactly that $I$ is stable under the operator $\varphi_w$, i.e.\@ that $I_w=I$.

\begin{corollary}\label{strict-Bruh-Pw}
The zip stratum $\Xcal^{\Zcal_{P_w}}_w$ of $\Xcal^{\Zcal_{P_w}}$ is Bruhat.
\end{corollary}

\begin{proof}
It is clear that the canonical parabolic subgroup of $w$ with respect to $\Xcal$ coincides with the canonical parabolic subgroup of $w$ with respect to $\Xcal^{\Zcal_{P_w}}$. Therefore the result follows from \Cref{prop-str-Br-Pw} applied to $\Xcal^{\Zcal_{P_w}}$.
\end{proof}

\begin{lemma}\label{low-neigh-Bruhat}
If $\Xcal_w$ is Bruhat, then for any lower neighbor $w'$ of $w$, the stratum $\Xcal_{w'}$ is weakly Bruhat.
\end{lemma}

\begin{proof}
By assumption, the $E$-orbit $E\cdot (wz^{-1})$ coincides with $Pwz^{-1}Q$. Hence, the complement of $Pwz^{-1}Q$ in its Zariski closure $\overline{Pwz^{-1}Q}$ is again stable under $P\times Q$. The result follows.
\end{proof}

\section{Changing the center}\label{section: changing the center}

\par Let \(G, G'\) be two connected reductive \(\Fp\)-groups related by an epimorphism \(\pi \colon G' \to G\) with central kernel \(K\). Consider \(\Zcal = (G, P, Q, L, M)\) a zip datum. We can use \(\pi\) to pull \(\Zcal\) back to \(\Zcal'\) a zip datum for \(G'\): for \(\square \in \{P, Q, L, M\}\) we simply take \(\square' \colonequals \pi^{-1}(\square)\); in particular, we have \(K \leq \square'\). Let \(E\) be the zip group of \(\Zcal\) and \(E'\) that of \(\Zcal'\). Write \(\Xcal = [E\backslash G_k], \Xcal' = [E'\backslash G_k']\) for the corresponding stacks of $G$-zips and $G'$-zips. The map $\pi$ induces a morphism of stacks $\Xcal' \to \Xcal$. By \cite[Lemma 3.4.1]{shen.zhang}, this map is smooth. We will show that this map is also a homeomorphism, using the following argument we learned from a draft revision of \cite{goldring2024ogusprinciplezipperiod}. Consider \(K^0\) the connected component of the identity of \(K\). Then we have
\[
    0\to K^0\longrightarrow K \longrightarrow K^{\rm{et}} \to 0,
\]
where \(K^{\rm{et}}\) is finite and \'{e}tale. This allows us to factor \(\pi\) as \(G' \to G'/K^0 \to G\), where the first map is an epimorphism with central kernel and the second map is an isogeny. By \cite[Lemma 3.4.1.(2)]{shen.zhang}, the isogeny \(G'/K^0 \to G\) induces a smooth homeomorphism. Thus, to conclude, it is enough to consider the case when \(K\) is connected. Under this assumption, we show something stronger: if \(K\) is connected, then \(\Xcal' \to \Xcal\) is a gerbe banded by the finite \'etale group scheme $K(\mathbb{F}_p)$.

\begin{lemma}
\label{lem-finet} Assume that \(K\) is connected. Let \(x \in G(R)\) and \(x' \in G'(R')\) a lift of \(x\) via \(\pi\), for \(R \to R'\) an fppf morphism of \(k\)-algebras. Then, \(\Stab_{E'}(x')\) is a central extension of \(\Stab_E(x)\) by the finite constant group \(K(\Fp)\). In particular, the map \(\Stab_{E'}(x') \to \Stab_E(x)\) is a \(K(\Fp)\)-torsor, hence finite \'{e}tale.
\end{lemma}
\begin{proof}

The morphism \(\pi\) induces a natural map \(\Stab_{E'}(x') \to \Stab_E(x)\). Take \((a, x^{-1}ax) \in \Stab_E(x)(S)\), for \(S\) an \(R'\)-algebra. Let \(a' \in P'(S')\) be any lift of \(a\), for \(S \to S'\) an fppf morphism of \(R'\)-algebras, and consider \((x')^{-1}a'x'\). Since \(\pi((x')^{-1}a'x') = x^{-1}ax \in Q(S)\), we know that \((x')^{-1}a'x' \in Q'(S')\). Moreover, we can write \(\varphi(\overline{a}') = (x')^{-1}\overline{a}'x' h\), for some \(h \in K(S')\). Consider an arbitrary element \(h_1 \in K(S')\). By construction, \(K\) is contained (centrally) in both \(L'\) and \(M'\). Furthermore, since \(K\) is of multiplicative type (because the center of a reductive group is), it will have trivial intersection with \(R_u(P')\) and \(R_u(Q')\), so that the maps \(R_u(Q') \to R_u(Q)\) and \(R_u(P') \to R_u(P)\) induced by \(\pi\) are isomorphisms. In particular, we know that \(\overline{a' h_1} = \overline{a}' h_1\).
 Then, \(\varphi(\overline{a}'h_1) = (x')^{-1}\overline{a}'h_1x' h_1^{-1}\varphi(h_1)h\), since \(K\) is central in \(G'\). In particular, to find a lift of \((a, x^{-1}ax)\) to an element of \(\Stab_{E'}(x')\), we must solve \(h = \varphi(h_1)^{-1} h_1\). Recall that the \emph{Lang torsor map} $K  \to K$, $t  \mapsto \varphi(t)^{-1}t,$ is finite \'{e}tale for any connected group of multiplicative type over \(\Fp\), such as \(K\). Its kernel is given by the finite group \(K(\Fp)\). This allows us to conclude.
\end{proof}

\begin{proposition}\label{proposition: map of gzips is a gerbe}
Assume that \(K\) is connected. The map \(\Xcal' \to \Xcal\) is a \(K(\Fp)\)-banded gerbe. In particular, \(\Xcal' \to \Xcal\) is a smooth homeomorphism.
\end{proposition}
\begin{proof}
Consider \(x, x'\), as in \Cref{lem-finet}. Let also \(x''\) be a lift via \(\pi\) of any element in the \(E\) orbit of \(x\). Taking an fppf morphism $S\to S'$ we find $e'\in E'(S')$ such that $\pi(e'x')=\pi(x'')$. Replacing $e'$ by a multiple $e'h$, for $h\in K(S')$, if necessary, we may assume that $e'x'=x''$. This, together with the arguments of Lemma \ref{lem-finet}, shows that the hypotheses of point (2) of \cite[\href{https://stacks.math.columbia.edu/tag/06P1}{Lemma 06P1}]{stacks-project} are satisfied. Thus, \(\Xcal' \to \Xcal\) is a gerbe. Moreover, by the orbit-stabilizer theorem, we have that
\[
\begin{tikzcd}
    & {\left[\Stab_{E'}(x') \backslash 1\right]} \arrow[r, "\sim"] \arrow[d] & {[E'\backslash (E'\cdot x')]} \arrow[r, hook] \arrow[d] & \Xcal' \arrow[d]\\
    & {\left[\Stab_{E}(x)\backslash 1\right]} \arrow[r, "\sim"] & {[E\backslash (E\cdot x)]} \arrow[r, hook] & \Xcal.
\end{tikzcd}
\]
This shows that the fibers over an \(R\)-point of \(\Xcal\) are locally isomorphic to the classifying stack \([K(\Fp)\backslash 1]\), from which we conclude that \(\Xcal' \to \Xcal\) is banded by \(K(\Fp)\). By \cite[\href{https://stacks.math.columbia.edu/tag/06R9}{Lemma 06R9}]{stacks-project}, \(\Xcal' \to \Xcal\) is a homeomorphism.
\end{proof}

\begin{theorem}\label{theorem: central extension gives gzip map}
Let $\mathscr{P}$ be a property of schemes that is local in the smooth topology. Let $\Gamma\subseteq \IW$ be a locally closed subset, let \(\Ucal_\Gamma = \cup_{w \in \Gamma} \Xcal_w\) and let \(\Ucal_\Gamma' = \cup_{w \in \Gamma} \Xcal_w'\). Then, $\Ucal_\Gamma'$ satisfies $\mathscr{P}$ if $\Ucal_\Gamma$ does, and $\Ucal_\Gamma$ satisfies $\mathscr{P}$ if $\Ucal_\Gamma'$ does.
\end{theorem}
\begin{proof}
This follows immediately from Proposition \ref{proposition: map of gzips is a gerbe} and \cite[Lemma 3.4.1.(2)]{shen.zhang}.
\end{proof}

\begin{remark}\label{rem-changing-center}
We remark that being smooth over \(k\), being normal and being Cohen--Macaulay are all local properties in the smooth topology.
\end{remark}

For a zip datum \(\Zcal\), let us denote by \(\Zcal^\mr{ad}\) the zip datum obtained from \(\Zcal\) by replacing \(G\) by \(G^\mr{ad} = G/Z(G)\) and any \(\square \in \{P, Q, L, M\}\) by \(\square/Z(G)_k\).
\begin{corollary}\label{smoothbeforeadj}
Let \(S\) be a \(k\)-scheme with a \(k\)-morphism \(\zeta \colon S \to \Xcal^\Zcal\). Suppose that the composition \(S \overset{\zeta}{\longrightarrow} \Xcal^\Zcal \to \Xcal^{\Zcal^\mr{ad}}\) is smooth. Then, \(\zeta\) is smooth.
\end{corollary}
\begin{proof}
This follows from Proposition \ref{proposition: map of gzips is a gerbe}, \cite[Lemma 3.4.1.(2)]{shen.zhang} and Lemma \ref{unramifiedtosmooth} below.
\end{proof}

\begin{lemma}
\label{unramifiedtosmooth}
Let  $X,Y,Z$ be algebraic stacks with maps $f \colon X\to Y$ and $g\colon Y \to Z$. Assume that $g$ is unramified and that the composition $g \circ f$ is smooth. Then, $f$ is smooth.
\end{lemma}

\begin{proof}
We can write $f$ as the composition of the second projection $X \times_{Z} Y  \to Y$  and the section $s\colon X \to X \times_{Z} Y$ of the first projection given by $f$. The second projection $X \times_{Z} Y  \to Y$ is smooth because $X \to Z$ is smooth. The diagonal $Y \to Y \times_{Z} Y$ is \'etale, since $g \colon Y \to Z$ is unramified, \cite[\href{https://stacks.math.columbia.edu/tag/0CJ0}{Lemma 0CJ0}]{stacks-project}. Pulling back along \(X \times_Z Y \to Y \times_Z Y\), we find that \(s\colon X \to X \times_{Z} Y\) is \'etale, hence smooth. Therefore $f$ is smooth.
\end{proof}

Let $(G,\mu)$ be a minuscule cocharacter datum of simple $\Csf_n$-type and let $\Zcal=\Zcal_\mu=(G,P,Q,L,M)$ be the corresponding zip datum. Wedhorn shows in \cite[Remark 2.5]{wedhorn.bruhat.PEL} that the Bruhat stratification (\S\ref{sec-Bruhat}) is linear and determined by the conjugate line position\footnote{The conjugate line position is sometimes referred to as the $a$-number.} \cite[\S(1.2.5)]{goldring2024ogusprinciplezipperiod} of the universal $G$-zip. Let  $\Xcal=\overline{\Xcal}_0\supseteq \overline{\Xcal}_1\supseteq  \ldots\supseteq \overline{\Xcal}_n$ denote the Bruhat stratification, where $\Xcal_i$ is the locus where the conjugate line position has value $i$.
\begin{corollary}
    Let $\Xcal_w$ denote the unique zip stratum of codimension one. Then $\overline{\Xcal}_w$ is Cohen--Macaulay and $\overline{\Xcal}_w^{\text{sm}}=\overline{\Xcal}_w\cap \Xcal_1$.
\end{corollary}
\begin{proof}
For $G=
\GSp_{2n}$ this was shown in \cite[Theorem 6.2.7]{goldring2024ogusprinciplezipperiod}. The general case thus follows from Theorem \ref{theorem: central extension gives gzip map}.
\end{proof}

\section{Singularities of strata}
\label{sec-singostrat}

\subsection{Canonical covers}\label{can-cov-sec}

Let $I_0\subseteq I$ be the type of an intermediate parabolic $B\subseteq P_0\subseteq P$. Recall that the twisted order with respect to the stack $\Xcal^{\Zcal_{P_0}}$ is denoted by $\preccurlyeq_{I_0}$.

\begin{definition}\label{def-elemop} Fix an element $w\in {}^I W$.
\begin{enumerate}
\item  We say that a nonempty subset $\Ucal\subseteq \Xcal$ is \emph{$w$-open} if $\Ucal\subseteq \overline{\Xcal}_w$ and is open in $\overline{\Xcal}_w$. We always endow a \(w\)-open subset with the reduced substack structure.
    \item Let $w'\in \IW$ be a lower neighbor of $w$. For such a pair $(w,w')$, define
\begin{equation*}
    \Ucal(w,w')\colonequals \Xcal_w \cup \Xcal_{w'},
\end{equation*}
endowed with the reduced substack structure. We call $\Ucal(w,w')$ an \emph{elementary substack}.
    
\end{enumerate}
\end{definition}

For a \(w\)-open substack $\Ucal\subseteq \overline{\Xcal}_w$, we let $\Gamma_\Ucal\subseteq {}^I W$ denote the subset corresponding to $\Ucal$ via the parametrization ${}^I W\to |\Xcal|$ given in \ref{sec-param}. Being $w$-open translates to the following three conditions: (1) $w\in \Gamma_\Ucal$, (2) For all $w'\in \Gamma_\Ucal$, one has $w'\preccurlyeq_I w$, and (3) if $w''\preccurlyeq_I w' \preccurlyeq_I w$ with $w''\in \Gamma_\Ucal$, then $w'\in \Gamma_\Ucal$.

In particular, note that $w\in \Gamma_\Ucal$ is the longest element of $\Gamma_\Ucal$. Set
\begin{equation}\label{UP0}
    \Ucal^{(P_0)}\colonequals \bigcup_{w\in \Gamma_\Ucal} \Fcal^{(P_0)}_w.
\end{equation}
Note that since the partial order $\preccurlyeq_I$ and (the restriction to ${}^I W$ of) the partial order $\preccurlyeq_{I_0}$ may differ, it can happen that $\Ucal^{(P_0)}$ is not contained in $\overline{\Fcal}^{(P_0)}_w$.

\begin{definition} \ 
\label{def-cover}
\begin{enumerate}
    \item We say that $\Ucal$ admits a \emph{$P_0$-cover} if $\Ucal^{(P_0)}$ is contained in $\overline{\Fcal}^{(P_0)}_w$ and is open in it.
    \item We say that $\Ucal$ admits a \emph{canonical cover} if it admits a $P_w$-cover.
\end{enumerate}
\end{definition}

The \(w\)-open substack $\Ucal$ admits a $P_0$-cover if and only if the set $\Gamma_{\Ucal}$, viewed as a subset of ${}^{I_0} W$, satisfies the following two properties:  (1) \(v\preccurlyeq_{I_0} w\) for all $v\in \Gamma_{\Ucal}$ and, (2) for all $w',w''\in {}^{I_0} W$, if $w''\preccurlyeq_{I_0}w'\preccurlyeq_{I_0} w$ and $w''\in \Gamma_\Ucal$, then $w'\in \Gamma_{\Ucal}$. This has the following consequence: if $P_0\subseteq P_1\subseteq P$ and if $\Ucal$ admits a $P_0$-cover, then $\Ucal$ also admits a $P_1$-cover.

\begin{remark}\label{rmk-U-bruh}
Assume that $w'$ is a Bruhat-lower neighbor of $w$. Then, for any parabolic $B\subseteq P_0\subseteq P$, the elementary substack $\Ucal(w,w')$ admits a $P_0$-cover. Indeed, since $w'\leq w$ we have $w'\preccurlyeq_{I_0} w$ for any $I_0\subseteq I$. However, for general elementary substacks $\Ucal(w,w')$, it can happen that $w'\preccurlyeq_{I_w} w$ does not hold, hence $\Ucal(w,w')$ does not always admit a canonical cover.
\end{remark}

We always endow the stratum $\overline{\Fcal}^{(P_0)}_w$ (for $w\in {}^{I_0}W$) with the reduced substack structure. Similarly, if $\Ucal$ admits a $P_0$-cover, we endow $\Ucal^{(P_0)}$ with the reduced substack structure. If a \(w\)-open substack $\Ucal$ admits a $P_0$-cover, the map $\pi_{P_0}\colon \overline{\Fcal}^{(P_0)}_w\to \overline{\Xcal}_w$ restricts to a map of stacks
\begin{equation}
   \pi_{P_0} \colon \Ucal^{(P_0)}\to \Ucal.
\end{equation}

\begin{definition}\label{def-sepcov}
Let $\Ucal\subseteq \overline{\Xcal}_w$ be an open subset. We say that $\Ucal$ admits a \emph{separating $P_0$-cover} if it admits a $P_0$-cover and the following diagram is Cartesian

\begin{equation}\label{diag-sep-cover}
\xymatrix@M=5pt{
\Ucal^{(P_0)} \ar@{^{(}->}[r] \ar[d]_{\pi_{P_0}} & \overline{\Fcal}_w^{(P_0)} \ar[d]^{\pi_{P_0}} \\
\Ucal \ar@{^{(}->}[r] & \overline{\Xcal}_w.
}
\end{equation}
We say that $\Ucal $ admits a \emph{separating canonical cover} if it admits a separating $P_w$-cover.
\end{definition}

Since it is difficult to compute images of \(\Zcal_{P_0}\)-zip strata, it is in general hard to check whether $\Ucal$ admits a separating cover. Note that $\Ucal$ admits a separating $P_0$-cover if and only if for all $w'\in {}^{I_0}W$ such that $w'\preccurlyeq_{I_0} w$, we have:
\begin{equation}
  \pi_{P_0}(w')\in \Gamma_\Ucal \ \Longrightarrow \ w'\in \Gamma_\Ucal
\end{equation}
where $\pi_{P_0} \colon {}^{I_0}W \to {}^IW$ is the map defined in \Cref{subsec-lownb}. If $P_0\subseteq P_1\subseteq P$ and $\Ucal$ admits a separating $P_0$-cover, then it also admits a separating $P_1$-cover. We deduce the following.
\begin{proposition}\label{prop-proper}
If $\Ucal$ admits a separating $P_0$-cover, the map $\pi_{P_0}\colon \Ucal^{(P_0)}\to \Ucal$ is finite. If $\Ucal$ admits a separating canonical cover, the map $\pi_{P_w}\colon \Ucal^{(P_w)}\to \Ucal$ is finite and birational.
\end{proposition}

\begin{proof}
For any $w'\in {}^I W$, the restriction of the map $\pi_{P_0}$ to the stratum $\Fcal^{(P_0)}_{w'}$ is finite onto $\Xcal_{w'}$, see \Cref{prop-flag}. Hence the map $\pi_{P_0}\colon \Ucal^{(P_0)}\to \Ucal$ is quasi-finite. Since \eqref{diag-sep-cover} is Cartesian, this map is also proper, hence we deduce that it is finite. Finally, if $P_0=P_w$, the map $\Fcal^{(P_w)}_w\to \Xcal_w$ is an isomorphism, so the result follows.
\end{proof}

\subsection{\texorpdfstring{$w$}{w}-Bounded subset}

Let $w\in {}^I W$ and let $\Ucal\subseteq \overline{\Xcal}_w$ be a \(w\)-open substack.
\begin{definition}\label{def-wbounded}
We say that $\Ucal$ is \emph{$w$-bounded} if, for any $w'\in \Gamma_\Ucal$, we have $P_{w'}\subseteq P_w$.
\end{definition}

We recall the following result on the fibers of proper birational morphisms.

\begin{proposition}[Zariski's Main Theorem, {\cite[Corollary 4.3]{liu-AG}}] \label{prop-birat-proper}
Let $X$ be an integral scheme, $Y$ a normal locally Noetherian scheme, and let $f\colon X\to Y$ be a proper birational morphism. Then, there exists an open subset $V\subseteq Y$ such that
\begin{enumerate}
    \item $Y\setminus V$ has codimension $\geq 2$,
    \item $f\colon f^{-1}(V)\to V$ is an isomorphism,
    \item for all $y\in Y\setminus V$, the fiber $X_y$ is connected.
\end{enumerate}
\end{proposition}

We deduce the following result.
\begin{theorem}\label{normal-bounded}
Let $\Ucal$ be a \(w\)-open substack. Assume that $\Ucal$ is normal and admits a canonical cover. Then $\Ucal$ is $w$-bounded.
\end{theorem}

\begin{proof}
Let $\Vcal$ denote the preimage of $\Ucal$ via the birational map $\pi_{P_w}\colon \overline{\Fcal}_{w}^{(P_w)}\to \overline{\Xcal}_w$. By assumption $\Ucal^{(P_w)}\subseteq \Vcal$ is an open subset. Apply \Cref{prop-birat-proper} to the proper, birational map $\pi_{P_w}\colon \Vcal\to \Ucal$. We deduce that for any $x\in \Ucal$, the fiber of $x$ in $\Vcal$ is connected. In particular, it is either trivial or contains no isolated point. Hence, the same holds for the fiber of $x$ in $\Ucal^{(P_w)}$. Since the map $\Ucal^{(P_w)}\to \Ucal$ is quasi-finite, by \Cref{prop-flag}.(\ref{prop-flag-1}), we deduce that it is bijective. Therefore, for all $w'\in \Gamma_{\Ucal}$, we have $P_{w'}\subseteq P_w$ by Condition (1d) in \Cref{rmk-1d}.
\end{proof}

\begin{remark}
When $\Ucal$ is an elementary \(w\)-open, we will see in \Cref{thm-elemsmoothchar} that the result is true without assuming that $\Ucal$ admits a canonical cover. For general \(w\)-open substacks however, the conclusion of \Cref{normal-bounded} is false if we remove the assumption that $\Ucal$ admits a canonical cover. For example, take $w$ to be the longest element $w_{0,I}w_0\in {}^I W$ and $\Ucal=\Xcal$, which is smooth. It can happen that the canonical parabolic $P_w$ is strictly smaller than the canonical parabolic of some other point of ${}^I W$. For example, for a unitary group $G=\U_n$ (for $n\geq 3$) endowed with a Hodge parabolic subgroup $P$ of type $(n-1,1)$, the canonical parabolic of the identity $\id\in {}^I W$ is $P$, whereas the canonical parabolic of the longest element is strictly smaller, by \Cref{lem-Pz}.
\end{remark}

Recall that a map of finite-type $k$-schemes $f\colon X\to Y$ is unramified if for any geometric point $s\colon \Spec(k)\to Y$, the fiber product $\Spec(k)\times_{Y,s} X$ is isomorphic to a disjoint union of copies of $\Spec(k)$. It is bijective unramified if and only if all its fibers are isomorphic to $\Spec(k)$. We prove the following lemma:

\begin{lemma}\label{fiber-triv}
Let $\Ucal$ be a $w$-bounded \(w\)-open substack of $\overline{\Xcal}_w$ which admits a canonical cover. Then the map
\begin{equation}
    \pi_{P_w}\colon \Ucal^{(P_w)}\to \Ucal
\end{equation}
is birational, bijective and unramified.
\end{lemma}

\begin{proof}
We already know that the map $\pi_{P_w}\colon \Ucal^{(P_w)}\to \Ucal$ is birational. It is bijective, by \Cref{prop-minimal-Pw} applied to each $w'\in \Gamma_\Ucal$ (using $P_{w'}\subseteq P_w$). We show that it is unramified. Let $U\subseteq G$ be the preimage of $\Ucal$ under the canonical projection $G\to \Xcal$. Similarly, the set $\Ucal^{(P_w)}$ corresponds to the union $U'=\bigcup_{v\in \Gamma_\Ucal} C_{P_w,v}$ (\S\ref{sec-fine-stratif}). By assumption, $U'$ is a locally closed subset of $G$. We have $\Ucal\simeq [E\backslash U]$ and $\Ucal^{(P_w)} \simeq [E'_{P_w}\backslash U']$. It suffices to show that the map $\Ucal^{(P_w)}\times_{\Ucal} U \to U$ has all of its geometric fibers isomorphic to $\Spec(k)$. This map corresponds to the natural map
\begin{equation}
    [E'_{P_w} \backslash (U'\times E)] \to U, \quad [g,(x,y)]\mapsto x^{-1}gy
\end{equation}
where $(a,b)\in E'_{P_w}$ acts on $(g,(x,y))\in U'\times E$ by $(agb^{-1},(ax,by))$.
Let $s\colon \Spec(k)\to U$ be a geometric point. If we conjugate $s$ by the action of $E$, the fiber above $s$ is correspondingly translated by $E$. Hence, we are reduced to the case when the image of $s$ is  $w'z^{-1}$, for some $w'\in \Gamma_\Ucal$. The fiber above $s$ identifies with the scheme $E'_{P_w} \backslash Y$, where
\begin{align*}
    Y& = \left\{ (g,(x,y))\in C_{P_{w},w'}\times E \ \relmiddle| \ x^{-1}gy=w'z^{-1} \right\} \\
    & \simeq  \left\{ (x,y)\in E \ \relmiddle| \  x w'z^{-1} y^{-1} \in C_{P_{w},w'}\right\}.
\end{align*}
Since $C_{P_{w},w'}$ is a single $E'_{P_{w}}$-orbit, any element of $Y$ can be written as a product of an element of $E'_{P_{w}}$ and an element of $\Stab_E(w'z^{-1})$. Since $\Stab_E(w'z^{-1})\subseteq E'_{P_{w}}$ (scheme-theoretically, not merely on $k$-points), we deduce that $Y$ coincides with $E'_{P_{w}}$ and thus the fiber above $s$ is trivial. This concludes the proof.
\end{proof}

\begin{proposition}
\label{bounded-implies-normal}
Let $\Ucal$ be an open subset of $\overline{\Xcal}_w$. Assume that
\begin{enumerate}
    \item $\Ucal$ admits a separating canonical cover, and
    \item $\Ucal$ is $w$-bounded.
\end{enumerate}
Then, the map $\pi_{P_w}\colon \Ucal^{(P_w)}\to \Ucal$ is an isomorphism. In particular, $\Ucal$ is normal and Cohen--Macaulay.
\end{proposition}

\begin{proof}
In this situation, the map $\pi_{P_w}\colon \Ucal^{(P_w)}\to \Ucal$ is finite, by \Cref{prop-proper}, birational, bijective and unramified, by \Cref{fiber-triv}. It follows that it is a closed immersion and, since $\Ucal$ is reduced, it is an isomorphism.
\end{proof}

\subsection{Canonical torsors} \

In \Cref{prop-extcanfilt-exa}, we saw how, in the case of \(\GL_n\) with signature \((r, s)\), for \(r\) and \(s\) coprime, we could extend the canonical filtration from the unique one-dimensional stratum \(S_\alpha\) to its closure. This, in turn, gives rise to a section \(\tilde{s}_{\alpha}\colon \overline{S}_{\alpha} \to \Fcal_S^{(P_\alpha)}\) of the natural projection \(\Fcal_S^{(P_\alpha)} \to S\), where \(\Fcal_S^{(P_\alpha)}\) is the flag space over \(S\) relative to \(P_\alpha = B\). Here \(P_\alpha\) denotes, as before, the canonical parabolic for \(s_{\alpha_r} = (r, r+1) \in \IW\), the unique non-compact simple root \(\alpha_r\) (corresponding to the unique one-dimensional stratum). In fact, since \(\alpha\) gives the relative position of the (conjugate) canonical filtration with respect to the Hodge filtration, it is clear that \(\tilde{s}_\alpha\) sends isomorphically \(S_\alpha\) to the Bruhat stratum \(\Fcal_{S, \alpha}^{(P_\alpha)}\). This is an example of a more general phenomenon. 
\subsubsection{The case of \texorpdfstring{\(\FZip\)}{F-zips}}
\label{ssec-fzipcanfilt}

Let $G=\GL_n$, endowed with the cocharacter $\mu\colon t\mapsto \diag(t\mathbb{1}_r, \mathbb{1}_s)$. With the notation of \cite{moonen.wedhorn.discrete.invariants}, \(\Xcal^\mu = \FZip^\tau\), for \(\tau \colon \Z \to \Z_{\geq 0}, 0, 1 \mapsto r, s\), supported on \(\{0, 1\}\). The de Rham cohomology determines a morphism $S\to \FZip^\tau$ \eqref{eq-zeta-shimura}. To $w\in \IW$ we associate a \emph{standard \(F\)-zip} \(\ul{M}_\tau^w \in \FZip^\tau(k)\) \cite[\S2.6]{moonen.wedhorn.discrete.invariants}. For \(\ul A \in S(k)\), the corresponding \(F\)-zip is in \(\FZip^\tau_w(k)\) if and only if \(w\) gives the relative position of the Hodge filtration \(H^1_{\mr{dR}, {\overline{\nu}}}(A/k) \supseteq \ul \omega_{\overline \nu}\) and the (conjugate) canonical filtration. The latter filtration is defined as the unique coarsest refinement \(D_{\mr{can},\bullet}(A/k)\) of the conjugate filtration stable under taking images via \(F\) and pre-images via \(V\).
\par The notion of canonical filtration makes sense on the substack \(\FZip^\tau_w, w \in \IW.\) In fact, using the same notation, we can consider \(\ul{M}_\tau^w\) as an \(F\)-zip on \(\FZip^\tau_w\). Correspondingly, we have a conjugate filtration \(D_0 = \Ocal_{\FZip^\tau_w}^r \subseteq D_1 = \Ocal_{\FZip^\tau_w}^n\), now a filtration by finite locally free sheaves, on the stack \(\FZip^\tau_w\). The structural isomorphism \(\varphi \colon \mr{gr}^\bullet(C)^{(p)} \to \mr{gr}_\bullet(D)\) gives a unique way of extending \(D_\bullet\) to \(D_{\mr{can},\bullet}\). Explicitly, we can consider
\[
    M^{(p)} = C^{0, (p)} \to \mr{gr}^0(C)^{(p)} \overset{\varphi}{\longrightarrow} \mr{gr}_0(D)=D_0 \subseteq D_1 = M,
\]
which corresponds to \(F \colon H^1_{\mr{dR}, {\overline{\nu}}}(A/T)^{(p)} \to H^1_{\mr{dR}, {\overline{\nu}}}(A/T)\), on the level of Shimura varieties. Similarly,
\[
    M = D_1 \to D_1/D_0 = \mr{gr}_1(D) \overset{\varphi^{-1}}{\longrightarrow} \mr{gr}^1(C)^{(p)} = C^{1, (p)} \subseteq C^{0, (p)} = M^{(p)}
\]
corresponds to \(V \colon H^1_{\mr{dR}, {\overline{\nu}}}(A/T) \to H^1_{\mr{dR}, {\overline{\nu}}}(A/T)^{(p)}\). Using these two maps, we obtain \(D_{\mr{can}, \bullet}\), which is a filtration by finite locally free sheaves over \(\FZip^\tau_w\). In fact, in a similar fashion, \(\varphi\) gives us a way to extend \(C^\bullet\) to a filtration \(C^\bullet_\mr{can}\), whose relative position with respect to \(D_{\bullet, w}\) is given by \(w\), and such that \(\varphi\) naturally extends to an isomorphism \(\mr{gr}^\bullet(C^\mr{can})^{(p)} \to \mr{gr}_\bullet(D_\mr{can})\).
\begin{definition}
The filtration \(C^\mr{can}\), resp.\@ \(D^\mr{can}\), naturally defined on \(\Xcal_w \cong \FZip^\tau_w\), is called the \emph{canonical filtration}, resp.\@ \emph{(conjugate) canonical filtration}.
\end{definition}
Let us write \(\tau_w (i) = \mr{rk}(\mr{gr}^i(C^\mr{can}))\) for the \emph{type} of \(C^\mr{can}\). One can define the stack \(\FFZip^{\tau, \tau_w}\) by asking that its \(T\)-points, for \(T \in \sch_k\), are given by pairs \((\ul M, \widetilde{C}^\bullet)\), where \(\ul M \in \FZip^\tau(T)\) and \(\widetilde{C}^\bullet\) is a descending filtration of \(M\) by finite locally free sheaves (with finite locally free quotients) of type \(\tau_w\) that refines the filtration \(C^\bullet(\ul M)\).
\par Notice that, as before, we can use \(\varphi\), which is part of the datum \(\ul M\), to obtain from \(\widetilde{C}^\bullet\) a refinement \(\widetilde{D}_\bullet\) of \(D_\bullet(\ul M)\). Thus, we have:
\begin{enumerate}[(i)]
    \item a natural forgetful morphism \(\FFZip^{\tau, \tau_w} \to \FZip^\tau, (\ul M, \widetilde{C}^\bullet) \mapsto \ul M\),
    \item a morphism \(\FFZip^{\tau, \tau_w} \to \FZip^{\tau_w}, (\ul M, \widetilde{C}^\bullet) \mapsto (M, \widetilde{C}^\bullet, \widetilde{D}^\bullet, \widetilde{\varphi})\).
\end{enumerate}
Via the identifications $\GZip^\mu = \FZip^\tau$, $\GF^\mu_{P_w} = \FFZip^{\tau, \tau_w}$, $\GZip^{\mc Z_{P_w}} = \FZip^{\tau_w}$ given by the standard representations, they correspond to the morphisms of stacks described in \ref{ssecflagspace}. Moreover, via the same identifications, the canonical filtration \(C^\mr{can}\) on \(\Xcal_w\) corresponds to a torsor \(\Ical_w\) for \(P_w\) the canonical parabolic.

\subsubsection{Canonical torsors in general} 
Let \(s_w\) be the inverse of the isomorphism \(\pi_{P_w} \colon \Fcal_w^{(P_w)} \to \Xcal_w\) of \Cref{prop-Pw} and define \(\Ical_w\) to be the pullback via \(s_w\) of the natural \(P_w\)-torsor on \(\Fcal_w^{(P_w)}\).
\begin{definition}\label{def-cantorcansec}
We call \(\Ical_w\) the \emph{canonical torsor} on \(\Xcal_w\) and \(s_w\) the \emph{canonical section}.
\end{definition}
The canonical torsor $\Ical_w$ can also be defined intrinsically in terms of $\Xcal_w$ as the torsor corresponding to the map
\[
\Xcal_w\cong [\Stab_E(wz^{-1})\backslash 1] \to [P_w\backslash 1],
\]
induced by the inclusion $\Stab_E(wz^{-1})\subseteq P_w$ (\Cref{lem-stabinpw}), where the first isomorphism is \eqref{stab-isom}.

\subsection{A characterization of normality}
We begin by proving \Cref{lem-rigidityofsubtorsors}, which gives, on a fixed stratum \(\Xcal_w\), a rigidity result for subtorsors of the Hodge torsor \(\Ical_P\).
\begin{lemma}
\label{lem-rigidityofsubtorsors}
Let \(P_0\) be a standard parabolic contained in \(P\). We have the following.
\begin{enumerate}
    \item For each \(w \in \IW\), there is at most one \(P_0\)-torsor \(\Ical_{P_0}\) over \(\Xcal_w\) contained in the Hodge torsor \(\Ical_P\).
    \item Such a torsor exists if and only if \(P_w \subseteq P_0\), in which case \(\Ical_{P_0} = \Ical_w P_0\), that is, the torsor \(\Ical_{P_0}\) is generated by the canonical torsor.
\end{enumerate}
\end{lemma}
\begin{proof}
Giving a \(P_0\)-torsor \(\Ical_{P_0}\) over \(\Xcal_w\) is the same as giving a morphism of stacks
\[
    \Xcal_w = [\Stab_E(wz^{-1})\backslash 1] \longrightarrow [P_0\backslash 1].
\]
The condition that \(\Ical_{P_0} \subseteq \Ical_P\) translates to the fact that the composition of this morphism with the one given by the inclusion \(P_0 \subseteq P\) corresponds to \(\Ical_P\). Following the conventions of \Cref{prop-minimal-Pw}, we can identify \(\Stab_E(wz^{-1})\) with a subgroup of \(P\), via the projection on the first factor. Since
\begin{align*}
    \hom([\Stab_E(wz^{-1})\backslash 1], [P_0\backslash 1]) &= \hom(\Stab_E(wz^{-1}), P_0)/P_0(k),\\
    \hom([P_0\backslash 1], [P\backslash 1]) &= \hom(P_0, P)/P(k), \\
    \hom([\Stab_E(wz^{-1})\backslash 1], [P\backslash 1]) &= \hom(\Stab_E(wz^{-1}), P)/P(k),
\end{align*}
we can factor the morphism of stacks \([\Stab_E(wz^{-1})\backslash 1] \to [P\backslash 1]\) as
\[[\Stab_E(wz^{-1})\backslash 1] \to [P_0\backslash 1] \to [P\backslash 1],\]
if and only if there is some \(x\in P(k)\) such that the first map is induced by \(\ad(x^{-1}) \colon \Stab_E(wz^{-1}) \to P_0\) and the second by \(\ad(x) \colon P_0 \to P\), which happens if and only if \(\Stab_E(w z^{-1}) \subseteq {}^{x}P_0\). In that case, Lemma \ref{lemma: canonical parabolic is everywhere} implies that \(x \in P_0(k)\), so that, by \Cref{prop-minimal-Pw}, this factorization exists if and only if \(P_w \subseteq P_0\) and the morphism \([\Stab_E(wz^{-1})\backslash 1] \to [P_0\backslash 1]\) is induced by
\(
    \Stab_E(wz^{-1})\subseteq P_w \subseteq P_0.
\)
This concludes the proof.
\end{proof}

\begin{remark}
The above Lemma has the following consequence: If $s\colon \Xcal_w\to \Fcal^{(P_0)}$ is a section of the projection $\pi_{P_0}\colon \Fcal^{(P_0)}\to \Xcal$, then $P_w\subseteq P_0$ and $s$ is the inverse of the isomorphism $\pi_{P_0}\colon \Fcal^{(P_0)}_w\to \Xcal_w$ (see Proposition \ref{prop-minimal-Pw}).
\end{remark}

\begin{theorem}[cf.\@ Theorem \ref{theorem: thmA intro}] 
\label{thm-extnorm} 
Let \(\Ucal \subseteq \GZip^\Zcal\) be a \(w\)-open, for a fixed \(w \in {}^IW\). The following are equivalent.
\begin{enumerate}
    \item \label{thm-extnorm1} The canonical section \(s_w \colon \Xcal_w \to \Fcal_w^{(P_w)}\) extends to a section \( \tilde{s}_w \colon \Ucal \to \Fcal^{(P_w)}\) of \(\pi_{P_w}\).
    \item \label{thm-extnorm2} The canonical torsor \(\Ical_w\) on $\Xcal_w$ extends to a \(P_w\)-torsor on \(\Ucal\) contained in \(\Ical_P\).
    \item \label{thm-extnorm3} \(\Ucal\) admits a separating canonical cover and is \(w\)-bounded.
    \item \label{thm-extnorm4} \(\Ucal\) admits a separating canonical cover and is normal.
    \item \label{thm-extnorm5} \(\Ucal\) admits a separating canonical cover and \(\pi_{P_w} \colon \Ucal^{(P_w)} \to \Ucal\) is an isomorphism.
\end{enumerate}
In that case, \(\Ucal\) is Cohen--Macaulay and the extensions of \(s_w\) and \(\Ical_w\) are unique.
\end{theorem}
\begin{proof}
From the moduli description of $\Fcal^{(P_w)}$, it is clear that (\ref{thm-extnorm1}) and (\ref{thm-extnorm2}) are equivalent.
Set $U\colonequals \bigcup_{w'\in \Gamma_\Ucal}G_{w'}$. If $\tilde{s}_w$ is an extension of $s_w$ to $\Ucal$, we obtain by pullback, a section $U \to G_k\times (P/P_w)$ of the first projection (see \S\ref{subsubsec-flagspace-def}). Since $G_k\times (P/P_w)$ is separated, there is always at most one extension of \(s_w\), hence there is at most one extension of \(\Ical_w\).
\par Assume that (\ref{thm-extnorm1}) or (\ref{thm-extnorm2}) hold. Denote the unique extension of \(s_w\) to $\Ucal$ by \(\tilde{s}_w \colon \Ucal \to \Fcal^{(P_w)}\). Denote by \(\widetilde{\Ical}_w\) the unique extension of \(\Ical_w\) to \(\Ucal\). By continuity, we have
\(\tilde{s}_w(\Ucal)  \subseteq \overline{\tilde{s}_w(\Xcal_w)}\subseteq \overline{\Fcal}_w^{(P_w)}.\) Let \(\Vcal \coloneqq\pi_{P_w}^{-1}(\Ucal)\cap \overline{\Fcal}_w^{(P_w)}\). Since $\widetilde{s}_w\colon \Ucal \to \Vcal$ is a section of $\pi_{P_w}\colon \Vcal\to \Ucal$, it is a closed embedding. Since \(s_w\) is an isomorphism and the open \(s_w(\Xcal_w) = \Fcal_w^{(P_w)}\) is dense, \(\widetilde{s}_w\colon \Ucal \to \Vcal\) is birational and surjective. Note that \(\Vcal\) is integral because it is open in \(\overline{\Fcal}_w^{(P_w)}\), which is itself integral. This implies that \(\widetilde{s}_w\colon \Ucal \to \Vcal\) is an isomorphism. Moreover, from \Cref{prop-Pw}, we know that \(\overline{\Fcal}_w^{(P_w)}\) is normal and Cohen--Macaulay. Therefore, \(\Ucal\cong \Vcal\) is also normal and Cohen--Macaulay, since it is open inside \(\overline{\Fcal}_w^{(P_w)}\). Assume that $w'\in \Gamma_\Ucal$. Since $\Xcal_{w'}$ is reduced to a point we have $\widetilde{s}_w(\Xcal_{w'})\subseteq {\Fcal}_{w_1'}^{(P_w)}$ for a unique $w_1'\in {}^{I_w}W$ with $w_1'\preccurlyeq_{I_w} w$. We show that \(w_1' = w'\). Let \(\ul \Ical\) denote the restriction to \(\Xcal_{w'}\) of the universal \(G\)-zip of type \(\Zcal\) over \(\Xcal\). We have that \(\tilde{s}_w(\ul{\Ical}) =(\ul{\Ical}, \widetilde{\Ical}_w|_{\Xcal_{w'}})\). \Cref{lem-rigidityofsubtorsors} implies that \(\Ical_{w'} \subseteq \widetilde{\Ical}_w |_{\Xcal_{w'}}\). The inclusion \(\Ical_{w'} \subseteq \widetilde{\Ical}_w|_{\Xcal_{w'}}\) is equivalent to the fact that the morphism \(\Psi_{P_w}\colon \Fcal^{(P_w)}\to \Xcal^{\Zcal_{P_w}}\) of \eqref{psimap} sends \((\ul{\Ical}, \widetilde{\Ical}_w|_{\Xcal_{w'}})\) to a point of the stratum \(\Xcal^{\Zcal_{P_w}}_{w'}\). Hence, $w_1'=w'$ and $\widetilde{s}_w(\Xcal_{w'})=\Fcal_{w'}^{(P_w)}$. This shows that $\Ucal$ admits a canonical cover. In fact, \(\Ucal\) admits a separating canonical cover because \(\widetilde{s}_w\colon \Ucal \to \Vcal\) is an isomorphism. By \Cref{normal-bounded}, \(\Ucal\) is \(w\)-bounded. This shows that \eqref{thm-extnorm1} or \eqref{thm-extnorm2} imply \eqref{thm-extnorm3} and \eqref{thm-extnorm5}.
\par Assume (\ref{thm-extnorm3}). By \Cref{bounded-implies-normal}, \(\pi_{P_w} \colon \Ucal^{(P_w)} \to \Ucal\) is an isomorphism. We can take \(\widetilde{s}_w\) to be its inverse. This shows (\ref{thm-extnorm1}). Moreover, by \Cref{bounded-implies-normal} again, \eqref{thm-extnorm3} implies \eqref{thm-extnorm4}. By \Cref{normal-bounded}, \eqref{thm-extnorm4} implies \eqref{thm-extnorm3}. Finally, it is clear that \eqref{thm-extnorm5} implies \eqref{thm-extnorm1}.
\end{proof}

If any of the equivalent conditions of Theorem \ref{thm-extnorm} is satisfied, irreducible substacks contained in $\Ucal$ with codimension one are also normal:

\begin{proposition}
Assume that the equivalent conditions of Theorem \ref{thm-extnorm} are satisfied. Let $\Ucal'\subseteq \Ucal$ denote an irreducible component of $\Ucal\setminus\{w\}$, endowed with the reduced structure. Then $\Ucal'$ is normal.
\end{proposition}

\begin{proof}
We have an isomorphism \(\pi_{P_w} \colon \Ucal^{(P_w)} \to \Ucal\). The preimage of $\Xcal_w$ by this map coincides with $\Fcal^{(P_w)}_w$, which is a Bruhat stratum of $\Fcal^{(P_w)}$. Hence, any irreducible component of $\overline{\Fcal}^{(P_w)}_w \setminus \Fcal^{(P_w)}_w$ is weakly Bruhat, by Lemma \ref{low-neigh-Bruhat}. In particular, it is normal. Since $\Ucal'$ is open inside an irreducible component of $\overline{\Fcal}^{(P_w)}_w \setminus \Fcal^{(P_w)}_w$, it is also normal.
\end{proof}

\begin{corollary}
\label{thm-elemsmoothchar}
Let \(\Ucal = \Ucal(w, w'), w, w' \in \IW,\) be an elementary \(w\)-open. Then, $\Ucal$ is smooth (equivalently, normal) if and only if any of the equivalent conditions of \Cref{thm-extnorm} holds.
\end{corollary}
\begin{proof}
Smoothness is equivalent to normality for $\Ucal$ since $\Xcal_w, \Xcal_{w'}$, have codimension zero, respectively one, in $\Ucal$. Assume now that $\Ucal$ is normal. The morphism \(\pi_{P_w}\colon\pi_{P_w}^{-1}(\Ucal)\cap \overline{\Fcal}_w^{(P_w)} \to \Ucal\) is birational, since it is an isomorphism over \(\Xcal_w\), and proper, since \(\pi_{P_w}\) is, and surjective, since $\overline{\Fcal}_w^{(P_w)}\to \overline{\Xcal}_w$ is, by properness. Therefore, since \(\Xcal_{w'}\) has codimension one in \(\Ucal\) and is reduced to a point, we know that \(\pi_{P_w}\colon\pi_{P_w}^{-1}(\Ucal)\cap \overline{\Fcal}_w^{(P_w)} \to \Ucal\) is an isomorphism, by \Cref{prop-birat-proper}. Its inverse extends \(s_w\), i.e.\@ \eqref{thm-extnorm1} of \Cref{thm-extnorm} holds. 
\end{proof}

\begin{remark}
\label{rmk-countergl7}
Consider \(G = \GL_7\), with \(\mu\) minuscule cocharacter of signature \((r, s)=(3, 4),\) in the notation of \ref{ssec-fzipcanfilt}. Take
\(
w = [4, 1, 2, 5, 3, 6, 7], w' = [1, 2, 4, 5, 6, 3, 7] \in \IW \subseteq \Sfr_7,
\)
of length \(4\) and \(3\), respectively. One can show that \(w' \preccurlyeq w\) and \(P_w = B = P_{w'}\), so that \(\Ucal(w, w')\) is a \(w\)-bounded elementary \(w\)-open. On the other hand, one can also check that \(w'\) is not a lower neighbor of \(w\) in the order \(\preccurlyeq_{I_w}\), so that in particular \(\Ucal(w, w')\) does not admit a canonical cover. By \Cref{thm-elemsmoothchar}, this implies that \(\Ucal(w, w')\) is not smooth, even though \(P_{w'} \subseteq P_w\) holds. In particular, this shows that a \(w\)-bounded \(w\)-open need not be smooth, or normal, even when it is elementary.
\end{remark}
\begin{remark}
\Cref{thm-elemsmoothchar} shows that, for elementary \(w\)-open substacks, our strategy to detect smoothness of $\Ucal$ using the canonical flag space $\Fcal^{(P_w)}$ is optimal.
\end{remark}
\begin{remark}
\label{rmk-computable}
The \(w\)-boundedness of a \(w\)-open \(\Ucal\) can be checked algorithmically, both in the positive and the negative, because the canonical parabolic can be computed effectively, given \(w\) and a description of the Weyl group \(W\) and the based root system \((\Phi, \Delta)\). It is also possible to decide algorithmically the existence of a canonical cover, as in \Cref{rmk-countergl7}, see the discussion after \Cref{def-cover}. The existence of a separating canonical cover is more difficult. When $\Ucal$ is an elementary $w$-open subset, the existence of a separating canonical cover, as well as the smoothness of $\Ucal$, can be decided by an algorithm: this is one of the main results of \cite{koskivirta2026decidabilitysingularitiesekedahloortstratification}.

In some cases, like that of \Cref{prop-extcanfilt-exa}, one can ensure the existence of a separating canonical cover by inspecting the canonical filtrations \(C_{\mr{can},w'}^\bullet\) on each \(\Xcal_{w'} \subseteq \Ucal\), using the standard Dieudonn\'{e} modules of \cite{gsas}. It is not clear if a similar procedure can be used to {completely} determine, for \(\Ucal\) a general \(w\)-open, whether a separating canonical cover exists or not. 
\end{remark}

Say that a \(w\)-open substack $\Ucal$ has \emph{depth one} if it consists of $w$ and certain lower neighbors of $w$. The statement of \Cref{thm-elemsmoothchar} generalizes immediately for \(w\)-open substacks of depth one. Consider now an arbitrary \(w\)-open substack $\Ucal$. Let $\Pcal$ be a property applying to substacks of $\Xcal$ (for example, ``Cohen--Macaulay'', or ``admits a separating cover''). We say that $\Ucal$ \emph{satisfies $\Pcal$ in codimension one} if $\Pcal$ is satisfied by all $w$-open elementary substacks $\Ucal(w,w')$ contained in $\Ucal$. For example, $\Ucal$ is \emph{$w$-bounded in codimension one} if for any lower neighbor $w'$ of $w$ in ${}^I W$ such that \(\Xcal_{w'} \subseteq \Ucal\), we have $P_{w'}\subseteq P_w$. Recall Serre's criterion for normality: a scheme $X$ is normal if and only if it satisfies properties $(R_1)$ and $(S_2)$, see \cite[\href{https://stacks.math.columbia.edu/tag/031O}{Tag 031O}]{stacks-project}. We deduce from \Cref{thm-elemsmoothchar} the following characterization of normality of arbitrary \(w\)-open substacks.
\begin{corollary}\label{thm-normchargen}
Let $\Ucal$ be a \(w\)-open substack. Then, $\Ucal$ is normal if and only if it satisfies the following conditions:
\begin{enumerate}
    \item \label{normal1}  $\Ucal$ satisfies property $(S_2)$,
    \item \label{normal2} $\Ucal$ is $w$-bounded in codimension one,
    \item \label{normal3} $\Ucal$ admits separating canonical covers in codimension one.
\end{enumerate}
\end{corollary}

For example, if $\Ucal$ is Cohen--Macaulay, it satisfies property $(S_k)$ for all $k$.

\begin{corollary}\label{proposition criteria for smooth locus to be open stratum}
Let $w\in {}^IW$ and let $\overline{\Xcal}_w^{\textnormal{no}}$ (resp.\@ $\overline{\Xcal}_w^{\textnormal{sm}}$) denote the maximal open substack that is normal (resp.\@ smooth). The following statements are equivalent:
\begin{enumerate}
    \item\label{item normal locus} One has $\overline{\Xcal}_w^\textnormal{no}=\overline{\Xcal}_w^\textnormal{sm}=\Xcal_w$.
    \item\label{item lower neighrbours} For any lower neighbor $w'\in {}^IW$ of $w$, $\Ucal(w,w')$ is either not $w$-bounded or does not admit a separating canonical cover.
\end{enumerate}
\end{corollary}

\section{Length one strata}
\label{ssec-lengthone}

\subsection{General results}
For an integer $m\geq 0$ and $w\in {}^I W$ such that $\ell(w)=m$, we call $\Xcal_w$ a \emph{length $m$ zip stratum}. In this section, we examine the case of length one strata. Let $w\in {}^I W$ with $\ell(w)=1$. Hence $w'=\id$ is the unique lower neighbor of $w$ in ${}^I W$, and it is a Bruhat-lower neighbor. In particular, $\overline{\Xcal}_w=\Ucal(w, \id)$ admits a $P_0$-cover for any parabolic $B\subseteq P_0\subseteq P$, by \Cref{rmk-U-bruh}. Moreover, since $\id$ is the unique lower neighbor of $w$ in $W$, it is clear that $w$ admits a separating $P_0$-cover. In particular, it admits a separating canonical cover. Write $P_{\id}$ for the canonical parabolic of the identity element $\id\in {}^I W$. The smoothness of the Zariski closure of the length one stratum $\Xcal_w$ is equivalent to the inclusion $P_{\id}\subseteq P_w$, by \Cref{thm-elemsmoothchar}. We will prove that the opposite inclusion $P_w\subseteq P_{\id}$ is always satisfied. We have the following easy lemma, of which we omit the proof.

\begin{lemma}\label{lem-obvious}
Let $\alpha$ be a simple root and let $\beta$ be an arbitrary root.
\begin{enumerate}
    \item \label{obv1} If $s_\alpha(\beta)$ is simple, then $\beta=-\alpha$ or $\beta$ is positive.
    \item \label{obv2} If $\beta$ and $s_\alpha(\beta)$ are both simple, then $s_\alpha(\beta)=\beta$.
\end{enumerate}
\end{lemma}

\begin{lemma} \label{lem-roots}
Let $\alpha \in \Delta\setminus I$ be a non-compact simple root. For any simple root $\beta\in \Delta$, we have the following.
\begin{enumerate}
\item \label{asser1} If $\beta\in \sigma(I)$, then $z^{-1}\beta$ is simple.
\item \label{asser2} If $\beta\notin \sigma(I)$, then $z^{-1}\beta$ is negative.
\item \label{asser3} If $s_\alpha (z^{-1}\beta)$ is simple, then $s_\alpha (z^{-1}\beta)=z^{-1}\beta$ or $z^{-1}\beta=-\alpha$.
\end{enumerate}
\end{lemma}

\begin{proof}
Assume first that $\beta\in \sigma(I)$. Note that $z^{-1}=w_0 \sigma(w_{0,I})$, and the root $\sigma(w_{0,I})\beta$ is the opposite of a simple root. Hence $z^{-1}\beta = w_0 \sigma(w_{0,I})\beta$ is always a simple root. Next, assume that $\beta\notin \sigma(I)$. Then, the root $\sigma(w_{0,I})\beta$ remains positive, hence $z^{-1}\beta$ is negative. For the third statement, we deduce from \Cref{lem-obvious} that $z^{-1}\beta =-\alpha$ or that $z^{-1}\beta$ is positive. In the latter case, from \eqref{asser1} and \eqref{asser2}, we deduce that $z^{-1}\beta$ is simple. Hence, the result follows from \eqref{obv2} of \Cref{lem-obvious}.
\end{proof}

For $\alpha\in \Delta\setminus I$, write $\varphi_\alpha$ for the operator $\varphi_{w}\colon \Phi \to \Phi$ for $w=s_\alpha$ (see \ref{sec-canonical}). Write $\varphi_{\id}$ for the operator corresponding to the identity element $\id\in {}^I W$.

\begin{proposition} \label{prop-contained}
Let $\alpha\in \Delta\setminus I$ be a non-compact simple root. For any subset $J\subseteq I$, we have
\begin{equation}
   I \cap \varphi_\alpha(J) \subseteq I \cap \varphi_{\id}(J).
\end{equation}
\end{proposition}

\begin{proof}
Let $\beta\in J$ be a simple root such that $\varphi_\alpha(\beta)\in I$. In particular $\varphi_\alpha(\beta)=s_\alpha z^{-1}\sigma(\beta)$ is simple. Applying \Cref{lem-roots}.\eqref{asser3} to $\sigma(\beta)$, we deduce that $s_\alpha z^{-1}\sigma(\beta)= z^{-1}\sigma(\beta)$ or that $z^{-1}\sigma(\beta)=-\alpha$. In the first case, we can write $\varphi_\alpha(\beta)=\varphi_{\id}(\beta)$, hence $\varphi_\alpha(\beta)$ lies in $I\cap \varphi_{\id}(J)$. In the second case, $s_\alpha(z^{-1}\sigma(\beta))= \alpha$ which does not lie in $I$.
\end{proof}

Using the modified operator $\varphi_{w,I}\colon I^* \to I^*$ defined in \eqref{phistar}, \Cref{prop-contained} can be written as $\varphi_{\alpha,I}(J)\subseteq \varphi_{\id,I}(J)$ for any subset $J\subseteq I^*$ containing 0. For $\alpha \in \Delta\setminus I$, write $P_\alpha$ for the canonical parabolic $P_{w}$ where $w=s_\alpha$.

\begin{corollary}\label{cor-oppinccanparl1}
For any non-compact root $\alpha\in \Delta\setminus I$, we have the inclusion $P_\alpha\subseteq P_{\id}$.
\end{corollary}

\begin{proof}
The types of the parabolics $P_\alpha$ and $P_{\id}$ are given by $\bigcap_{m\geq 0}\varphi_{\alpha,I}^m(I)$ and $\bigcap_{m\geq 0}\varphi_{\id,I}^m(I)$ respectively. Hence, the result follows from \Cref{prop-contained}.
\end{proof}

The inclusion $P_\alpha\subseteq P_{\id}$ shows that in general length one strata ``tend'' to be non-smooth. We deduce the following.

\begin{theorem}\label{dim1-prop}
For any $\alpha\in \Delta\setminus I$, the following are equivalent:
\begin{enumerate}
    \item The locally closed substack $\overline{\Xcal}_\alpha$ is smooth.
    \item The locally closed substack $\overline{\Xcal}_\alpha$ is normal.
    \item One has $P_{\id}\subseteq P_\alpha$.
    \item One has $P_{\id} = P_\alpha$.
    \item One has $s_\alpha(I_{\id})=I_{\id}$.
\end{enumerate}       
\end{theorem}

\begin{proof}
We already know that the first four assertions are equivalent. Moreover, they are equivalent to the equality $\varphi_\alpha(I_{\id})=I_{\id}$. We have $\varphi_\alpha(I_{\id})=s_\alpha z^{-1}\sigma(I_{\id})=s_\alpha \varphi_{\id}(I_{\id})=s_\alpha(I_{\id})$, so the result follows.
\end{proof}

\subsubsection{Dense subsets of \texorpdfstring{$\Delta$}{D}}\label{section: dense subsets of delta}
For a subset $K\subseteq\Delta$, say that $\alpha\in \Delta$ is \textit{$K$-isolated} if $\langle \alpha,\beta^\vee\rangle=0$ for all $\beta\in K$ (in other words, if $\alpha$ is orthogonal to all roots in $K$). Furthermore, say that $K$ is \textit{dense} in $\Delta$ if no root in $\Delta$ is $K$-isolated.

\begin{corollary}\label{cor-Borel} \ 
\begin{enumerate}
    \item \label{dim1-cor1} Assume that $I_{\id}=\emptyset$. Then, the Zariski closure of any length $1$ stratum is smooth.
    \item \label{dim1-cor2} Assume that $I_{\id}$ is dense in $\Delta$. Then, the Zariski closure of any length $1$ stratum is singular.
\end{enumerate}
\end{corollary}

\begin{proof}
The first assertion is clear. For \eqref{dim1-cor2}, it suffices to show that for any $\alpha\in \Delta\setminus I$, we have $s_\alpha(I_{\id})\neq I_{\id}$. Since $I_{\id}$ is dense in $\Delta$, the simple root $\alpha$ is connected to some root $\beta\in I_{\id}$. Hence $s_\alpha(\beta)$ is not simple, which shows the result.
\end{proof}

\subsection{Type \texorpdfstring{$\Asf$}{A}}\label{sec-antype}
In this section, we investigate the smoothness of one-dimensional strata for all groups of type $\Asf$. Moreover, we assume that $G$ is endowed with a minuscule cocharacter $\mu\colon \GG_{\textrm{m},k}\to G_k$. By \S\ref{section: changing the center} (\Cref{rem-changing-center}) it suffices to cover the case where $G_k$ is isomorphic to a product of copies of $\GL_{n,k}$, for some fixed \(n \geq 1\).

\subsubsection{Simple split case}
\label{subsub-simple-split-A}

We start with the simpler case $G=\GL_{n,\FF_p}$, for $n\geq 1$, which corresponds to the toy example in \S\ref{sec-simpexa}. Let $B\subseteq G$ (resp.\@ $B^+$) denote the lower triangular (resp.\@ upper triangular) Borel subgroup, and $T\subseteq B$ the diagonal torus. Identify $X^*(T)=\ZZ^n$ so that $(\lambda_1, \dots, \lambda_n)\in \ZZ^n$ corresponds to the character
\begin{equation}
    \diag(t_1, \dots,t_n)\mapsto \prod_{i=1}^n t_i^{\lambda_i}.
\end{equation}
Write $(e_1, \dots,e_n)$ for the standard basis of $\ZZ^n$. For $1\leq i\leq n-1$, set $\alpha_i\colonequals e_i-e_{i+1}$. The set of simple roots is given by $\Delta=\{\alpha_i \ | \ 1\leq i \leq n-1\}$. Let $r,s \geq 0$ be integers such that $n=r+s$. Endow $G$ with the cocharacter
\begin{equation*}
    \mu\colon \GG_{\textrm{m},k}\to G_k, \ t\mapsto \diag(t\mathbb{1}_r, \mathbb{1}_s).
\end{equation*}
Let $\Zcal\colonequals \Zcal_\mu=(G,P,Q,L,M)$ be the associated zip datum. In the case $r=0$ or $s=0$, there is no length one stratum, so we assume $r,s>0$. Since $\mu$ is defined over $\FF_p$, we have $L=M$, which is a Levi subgroup isomorphic to $\GL_{r,\FF_p}\times \GL_{s,\FF_p}$. The type of $P$ is then given by $I=\Delta\setminus \{\alpha_r\}$. In this section, it will be convenient to define $\Delta^*\colonequals \Delta\cup\{0\}$ and identify $\Delta^* = \ZZ/n\ZZ$ via the correspondence $\alpha_i\mapsto i$ for $1\leq i\leq n-1$ and $0\mapsto 0$ (unless otherwise specified, all integers will always be taken modulo $n$ and viewed in $\ZZ/n\ZZ$). Using this identification, we may write $I=\ZZ/n\ZZ \setminus\{ 0,r\}$. 

We first look at the map $\varphi_{\id,\Delta}\colon \Delta^*\to \Delta^*$. We have $\varphi_{\id,\Delta}(0)=0$ and $\varphi_{\id,\Delta}(l)=l+s$ for $l\in \Delta$ (note that this formula also holds if $l=n-s=r$). In other words, the map $\varphi_{\id,\Delta}$ is given on $\Delta\subseteq \Delta^*$ by the restriction of the translation map
\begin{equation}
  T_s \colon\ZZ/n\ZZ \to \ZZ/n\ZZ, \quad l\mapsto l+s.
\end{equation}
By definition, the type $I_{\id}$ of the canonical parabolic $P_{\id}$ is the largest subset $K\subseteq I$ which satisfies $T_s(K)=K$. It is clear that this subset is precisely the complement of the subgroup generated by $s$ inside $\ZZ/n\ZZ$. We have shown:

\begin{proposition}
Let $\delta\colonequals \gcd(r,s)$. The type of the canonical parabolic $P_{\id}$ is $I_{\id} = \{\alpha_i \ | \ i\notin \delta\ZZ \}$.
\end{proposition}

\begin{corollary}
The following are equivalent:
\begin{enumerate}
    \item One has $P_{\id}=B$.
    \item One has $\gcd(r,s)=1$.
\end{enumerate}
\end{corollary}

Next, we compute the canonical parabolic of the unique length one stratum. Write simply $\alpha=\alpha_r$ for the unique non-compact simple root. Write $\varphi_{\alpha,I}$, $P_\alpha$, $I_\alpha$ respectively for the objects $\varphi_{w,I}$, $P_w$ and $I_w$ when $w=s_\alpha$. As an element of $W=\Sfr_n$, the reflection $s_\alpha$ corresponds to the transposition exchanging $r$ and $r+1$. For $l\in I$, we find
\begin{equation*}
  \varphi_{\alpha,I}(l) =  \begin{cases}
  l+s & \quad \textrm{if } \ l\notin \{r-s-1,r-s,r-s+1\} \\
  0 & \quad \textrm{otherwise.}
    \end{cases}
\end{equation*}
We deduce that the complement of $I_{\alpha}$ in $\ZZ/n\ZZ$ consists of the elements of the form $l+\delta \ZZ$ for $l\in \{-1,0,1\}$. We have shown:

\begin{proposition}
The type of $P_\alpha$ is $I_{\alpha} = \{ \alpha_i \ | \ i\not\equiv -1,0,1 \pmod{\delta} \}$.
\end{proposition}

\begin{corollary}
The following are equivalent:
\begin{enumerate}
    \item One has $P_\alpha=B$.
    \item One has $\gcd(r,s)\leq 3$.
\end{enumerate}
\end{corollary}

Write $\Xcal_\alpha$ for the unique length one stratum in $\Xcal$. We deduce from \Cref{dim1-prop}:

\begin{theorem}\label{thm-len1gln}
The following are equivalent:
\begin{enumerate}
\item The Zariski closure $\overline{\Xcal}_\alpha$ is smooth.
\item The Zariski closure $\overline{\Xcal}_\alpha$ is normal.
\item One has $P_{\id}=P_{\alpha} = B$.
\item One has $I_{\id} = I_{\alpha}=\emptyset$.
\item One has $\gcd(r,s)=1$.
\end{enumerate}
\end{theorem}

\subsubsection{General split case}
We consider the case when $G=\Res_{\FF_{p^d}\mid\FF_p}(\GL_{n,\FF_{p^d}})$. We have an identification
\begin{equation}
    \FF_{p^d}\otimes_{\FF_p} k = \prod_{\tau\colon \FF_{p^d}\to k} k.
\end{equation}
Fix an embedding \(\tau_1 \colon \FF_{p^d}\to k\) and set \(\tau_i = \sigma^{i-1}\tau_1\). This ordering yields an identification of $G_k$ with the direct product $\prod_{i=1}^d \GL_{n,k}$. The Frobenius homomorphism $\sigma\in \Gal(k/\FF_p)$ acts on $G(k)$ by translation to the right:
\begin{equation}
    \sigma(x_1, \dots, x_d) = (\sigma(x_d),\sigma(x_1), \dots, \sigma(x_{d-1}))
\end{equation}
where $\sigma(x)$, for $x\in \GL_n(k)$, is the matrix obtained by raising the entries of $x$ to the $p$-th power. We fix the Borel subgroup $B\subseteq G$ defined as $B=B_1\times \dots \times B_d$, where $B_i$ is the lower triangular Borel subgroup of $\GL_n$. Similarly, we fix the maximal torus $T=T_1\times \dots \times T_d$ where $T_i\subseteq \GL_n$ is the diagonal torus. We identify $X^*(T)=\prod_{i=1}^d X^*(T_i)$ and $X^*(T_i)=\ZZ^n$ as above.
Write $(e_{i,j})_{1\leq j\leq n}$ for the standard basis of $X^*(T_i)=\ZZ^n$. Note that $(B,T)$ is defined over $\FF_p$. The Frobenius homomorphism $\sigma\in \Gal(k/\FF_p)$ acts on $X^*(T)$ by translation to the right: $\sigma\cdot (\lambda_1, \dots,\lambda_d)=(\lambda_d,\lambda_1, \dots,\lambda_{d-1})$. The set of simple roots of $G$ is $\Delta=\Delta_1\sqcup \Delta_2 \sqcup \dots \sqcup \Delta_d$, where $\Delta_i=\{\alpha_{i,j} \ | \ 1\leq j \leq n-1\}$ with $\alpha_{i,j}=e_{i,j}-e_{i,j+1}$. 

Let $r=(r_1, \dots, r_d)$ be a $d$-tuple of integers such that $0\leq r_i \leq n$. For each $i=1, \dots,d$, we let $s_i\colonequals n-r_i$. We consider the cocharacter $\mu\colon \GG_{\textrm{m},k}\to G_k$ given on the $i$-th factor by
\begin{equation}\label{mui-def}
    \mu_i\colon t \mapsto \diag(t \mathbb{1}_{r_i}, \mathbb{1}_{s_i}).
\end{equation}
Let $\Zcal$ be the zip datum attached to $(G,\mu)$. For $\square\in \{G,P,Q,L,M\}$, we have a decomposition $\square = \prod_{i=1}^d \square_i$. The set of compact simple roots (i.e.\@ simple roots contained in $L$) is given by
\begin{equation}
    I = I_1\sqcup \cdots \sqcup I_d
\end{equation}
where $I_i=\Delta_i\setminus \{\alpha_{i,r_i}\}$ if $1<r_i<n$ and $I_i=\Delta_i$ if $r_i\in \{0,n\}$. We set $s\colonequals s_1+\dots+s_d$ and $\delta\coloneqq\gcd(s,n)$. The subgroup $\delta \ZZ/n\ZZ\subseteq \ZZ/n\ZZ$ will have significance in the following results. For any $i\in \ZZ$, define $s_i$ as the element $s_j$ where $j$ is the unique integer such that $1\leq j\leq d$ and $i\equiv j \pmod{d}$. For $l \in \ZZ$, denote by $\overline{l}$ the set of all integers modulo $n$ that can be written as $l+ \delta m$ for $m\in \ZZ$. In other words, $\overline{l}$ is the coset in $\ZZ/n\ZZ$ of the class of $l$ with respect to the subgroup $\delta \ZZ/n\ZZ$. Note that $\overline{l}$ also consists of elements that can be written as $l + s m'$ for some $m'\in \ZZ$. In what follows, we will implicitly identify $\Delta_i$ with the set of classes $\{1, \dots, n-1\}\subseteq \ZZ/n\ZZ$. Using this identification, we write $\Delta_i\setminus \overline{l}$, for the subset obtained by removing from $\Delta_i$ all elements of the form $l + \delta m$, $m \in \ZZ$. To simplify notation, we write \(\cup S\) to denote the union \(\cup_{s \in S} s\) of the elements of a set \(S\subseteq \Pcal(T)\) of subsets of a set \(T\).

\begin{proposition}\label{prop-canon0}
The canonical parabolic subgroup $P_{\id}$ of $\id \in {}^I W$ is given by $P_{\id}=P_{\id,1}\times \dots \times P_{\id,d}$, where $P_{\id,i}\subseteq \GL_n$ is the standard parabolic subgroup of $\GL_{n}$ of type
\begin{equation}\label{Ii-formula}
    I_{\id,i} = \Delta_i\setminus \left(\cup\left\{\sum_{l = 1}^N \overline{s}_{i-l} \ \relmiddle| \ 1\leq N\leq d \right\} \right).
\end{equation}
\end{proposition}

\begin{proof}
Denote by $\widetilde{I}_{\id,i}$ the set defined by the right-hand side of \eqref{Ii-formula}, and write $\widetilde{I}_{\id} \colonequals \widetilde{I}_{\id,1}\sqcup \dots \sqcup \widetilde{I}_{\id,d}$. Note that the term corresponding to $N=d-1$ is 
\begin{equation}
    \sum_{j=1}^{d-1} \overline{s}_{i-j} = -\overline{s}_{i} = \overline{r}_i.
\end{equation}
Hence, it is clear that $\widetilde{I}_{\id,i}\subseteq I_{i}=\Delta_i\setminus \{r_i\}$ and thus $\widetilde{I}_{\id} \subseteq I$. We also note that the term corresponding to $N=d$ is the class modulo $\delta$ of $s=s_1+\dots +s_d$, which is $\overline{0}$. We obtain $\varphi_{\id}(\widetilde{I}_{\id})=J_1\sqcup \dots \sqcup J_d$ where
\begin{align*}
    J_i=& \Delta_i\setminus \left(\cup\left\{\overline{s}_{i-1}+\sum_{j=1}^N \overline{s}_{i-1-j} \ \relmiddle| \ 1\leq N\leq d \right\}\right) \\
    =& \Delta_i\setminus \left(\cup\left\{\sum_{j=1}^N \overline{s}_{i-j} \ \relmiddle| \ 1\leq N\leq d \right\}\right) = \widetilde{I}_{\id,i}.
\end{align*}
We deduce $\varphi_{\id}(\widetilde{I}_{\id})=\widetilde{I}_{\id}$, thus $\widetilde{I}_{\id}$ is a subset of $I$ stable under $\varphi_{\id}$, which implies $\widetilde{I}_{\id}\subseteq I_{\id}$. To prove the reverse inclusion, we use \eqref{eq-Iw} and show by induction on $m$ that $\bigcap_{l = 0}^m \varphi_{\id}^l(I) = K_1\sqcup \dots \sqcup K_d$ where
\begin{equation}
K_i = \Delta_i \setminus \left(\{r_i\} \cup \left\{\sum_{j=1}^{N} s_{i-j} \ \relmiddle| \ 1\leq N\leq m \right\}\right),
\end{equation}
for $m\geq 0$. Note that here we do not remove the whole class modulo $\delta \ZZ / n\ZZ$, only the elements $\sum_{j=1}^N s_{i-j}$ (taken modulo $n$). The claim holds clearly for $m=0$. Assume that it holds for some $m\geq 0$. Since $\varphi_{\id}\colon \Phi\to \Phi$ is injective,
\begin{equation}
 \bigcap_{l = 0}^{m+1} \varphi_{\id}^l(I) = I\cap \varphi_{\id} \left(\bigcap_{l=0}^m \varphi_{\id}^l(I) \right) = \bigsqcup_{i=1}^d I_i \cap \varphi_{\id}(K_{i-1}),
\end{equation}
where indices are taken modulo $d$ as usual. On the $i$-th factor, $s_{i-1}$ never lies in the image of $\varphi_{\id}$, hence:
\begin{align*}
I_i \cap \varphi_{\id}(K_{i-1}) &= I_i \cap \varphi_{\id}\left( \Delta_{i-1} \setminus \left(\{r_{i-1}\} \cup \left\{\sum_{j=1}^{N} s_{i-1-j} \ \relmiddle| \ 1\leq N\leq m \right\}\right) \right) \\
&= I_i \setminus \left(\{ s_{i-1}\}\cup \left\{s_{i-1}+\sum_{j=1}^{N} s_{i-1-j} \ \relmiddle| \ 1\leq N\leq m \right\}\right)  \\
&= \Delta_i \setminus \left(\{r_i\} \cup \left\{\sum_{j=1}^{N} s_{i-j} \ \relmiddle| \ 1\leq N\leq m+1 \right\}\right).
\end{align*}
We deduce that the claim holds for all $m\geq 0$. Changing $m$ to $m + d l$ for $l\geq 1$ recovers all elements in the same $\delta \ZZ/n\ZZ$-coset, hence the result follows from $I_{\id}=\bigcap_{l\geq 0}\varphi_{\id}^l(I)$.
\end{proof}

Next, we examine the elements of length one. Note that non-compact simple roots $\alpha\in \Delta\setminus I$ correspond bijectively to the factors $1\leq j \leq d$ such that $r_j\notin \{0,n\}$. For each such integer $j$, denote simply by $\alpha_j$ the unique non-compact root $\alpha_{j,r_j}$ in the $j$-th factor. Moreover, write $P_{j}$ for the canonical parabolic of the simple reflection $s_{\alpha_j}$. Decompose $P_j$ as a product
\begin{equation}
    P_j=P_{j,1}\times \dots \times P_{j,d}.
\end{equation}

\begin{proposition}\label{prop-canon1}
Let $1\leq j \leq d$ such that $1\leq r_j\leq n-1$. For each $1\leq i \leq d$, the parabolic $P_{j,i}$ is the standard parabolic subgroup of $\GL_n$ of type
\begin{equation}\label{Ii-formula2}
    I_{j,i} = \Delta_i\setminus \left(\cup \left(\left\{\sum_{l = 1}^N \overline{s}_{i - l} \ \relmiddle| \ 1\leq N\leq d \right\}\cup \left\{ \pm \overline{1} + \sum_{l = 1}^{d+i-j-1} \overline{s}_{i - l}  \right\}\right) \right)
\end{equation}
\end{proposition}

\begin{proof}
Denote by $\Gamma_{j,i}$ the set defined by the right-hand side of \eqref{Ii-formula2}, and write $\Gamma_j$ for  $\Gamma_{j,1}\sqcup \dots \sqcup \Gamma_{j,d}$. It is again clear that $\Gamma_{j,i}\subseteq I_{i}=\Delta_i\setminus \{r_i\}$, hence $\Gamma_j \subseteq I$. Denote by $\varphi_j$ the operator $\varphi_w$ for the element $w=s_{\alpha_j}$. We obtain $\varphi_j(\Gamma_j)=J_1\sqcup \dots \sqcup J_d$, where
\begin{align*}
    J_i&= \varphi_j\left( \Delta_{i-1}\setminus \left( \cup \left(\left\{\sum_{l = 1}^N \overline{s}_{i - 1 - l} \ \relmiddle| \ 1\leq N\leq d \right\}\cup \left\{ \pm \overline{1} + \sum_{l = 1}^{d+i-j-2} \overline{s}_{i - 1 - l}  \right\} \right) \right)\right) \\
    &= \Delta_i \setminus \left(\cup \left(\left\{\overline{s}_{i-1} + \sum_{l = 1}^N \overline{s}_{i - 1 - l} \ \relmiddle| \ 1\leq N\leq d \right\}\cup \left\{ \pm \overline{1} + \overline{s}_{i-1}+\sum_{l = 1}^{d+i-j-2} \overline{s}_{i - 1 - l}  \right\} \right) \right) \\
    &=\Delta_i\setminus \left( \cup \left(\left\{\sum_{l = 1}^N \overline{s}_{i - l} \ \relmiddle| \ 1\leq N\leq d \right\}\cup \left\{ \pm \overline{1} + \sum_{l = 1}^{d+i-j-1} \overline{s}_{i - l}  \right\} \right)\right).
\end{align*}
Thus $\varphi_j(\Gamma_j)=\Gamma_j$ and hence $\Gamma_j\subseteq I_j$. For the reverse inclusion, we show by induction on $m\geq 0$ that $\bigcap_{l = 0}^m \varphi_j^l(I) = K_1\sqcup \dots \sqcup K_d$ where
\[
    K_i = \Delta_i \setminus \left(\{r_i\} \cup \left\{\sum_{l=1}^{N} s_{i-l} \ \relmiddle| \ 1\leq N\leq m \right\}\cup \left\{ r_{i}\pm1 + \sum_{l = 1}^{d+i-j-1} s_{i - l}  \right\} \right)
\]
if $i \equiv j, j + 1, \ldots, j + m-1 \pmod{d}$, and
\[
 K_i = \Delta_i \setminus \left(\{r_i\} \cup \left\{\sum_{l=1}^{N} s_{i-l} \ \relmiddle| \ 1\leq N\leq m \right\} \right)
\]
otherwise. 
It is a routine check analogous to the proof of \Cref{prop-canon0}.
\end{proof}


\begin{theorem}\label{thm-smooth}
    The following are equivalent.
\begin{enumerate}
    \item The length one stratum $\overline{\Xcal}_{\alpha_j}$ is smooth.
 \item The length one stratum $\overline{\Xcal}_{\alpha_j}$ is normal.
 \item One has
 \begin{equation}
     \pm 1 \in \cup\left\{\sum_{l=1}^N \overline{s}_{j+1-l} \ \relmiddle| \ 1\leq N\leq d \right\}.
 \end{equation}
\end{enumerate}
\end{theorem}

\begin{proof}
By \Cref{dim1-prop}, smoothness and normality of $\overline{\Xcal}_j$ are both equivalent to the equality $P_{\id}=P_j$, which amounts to $P_{\id,i}=P_{j,i}$ for all $1\leq i \leq d$, equivalently to $I_{\id,i}=I_{j,i}$ for all $1\leq i \leq d$. Comparing the types of these parabolic subgroups given in Propositions \ref{prop-canon0} and \ref{prop-canon1}, we see that this equality is satisfied for some $i\in \{1, \dots, d\}$ if and only if it holds for all indices $i$ at the same time, since the condition corresponding to a different value of $i$ is given by adding a constant on each side. Choosing the index $i=j+1$, we deduce the theorem.
\end{proof}

\begin{corollary}
If $\delta=1$, then all length one strata are smooth.
\end{corollary}

\begin{proof}
In this case, there is only one $\delta \ZZ/n\ZZ$-coset in $\ZZ/n\ZZ$, so the result follows trivially from \Cref{thm-smooth}.
\end{proof}

\begin{example}
    Assume that $r_j=n-1$ and $r_{j-1}=2$. Then, $\overline{\Xcal}_{\alpha_j}$ is smooth. Indeed, we have in this case $s_j=1$ and $s_{j-1}=n-2$, hence the first terms in the sequence $\sum_{l = 1}^N \overline{s}_{j + 1 - l}$ for $N=1,2, \dots$ are $1,-1$. Therefore, the result follows from \Cref{thm-smooth}.
\end{example}

We now restrict to the case $d=2$. In this case, the set of non-compact roots $\Delta\setminus I$ contains at most two elements.

\begin{proposition}
The following are equivalent:
\begin{enumerate}
\item The Zariski closure of each (any) length $1$ stratum is smooth.
\item We have either $\delta=1$, or $\delta=2$ and $r_1, r_2$ are odd.
\item We have $P_{\id}=B$.
\end{enumerate}
\end{proposition}

\begin{proof}
The smoothness of $\overline{\Xcal}_1$ (resp.\@ $\overline{\Xcal}_2$) is equivalent to $\pm 1 \in \cup\{\overline{s}_1,\overline{0}\}$ (resp.\@ $\pm 1 \in \cup\{\overline{s}_2,\overline{0}\}$). Hence we must have $\delta\leq 2$. In the case $\delta=2$, all integers $s_1$, $s_2$, $r_1$, $r_2$ have the same parity, thus we deduce that both strata are smooth simultaneously, and this happens if and only if $r_1$, $r_2$ are odd. By \Cref{prop-canon0}, the type of the parabolic $P_{\id,1}$ (resp.\@ $P_{\id,2}$) is $\Delta_1 \setminus (\cup \{\overline{0}, \overline{s}_2\})$ (resp.\@ $\Delta_2 \setminus (\cup\{\overline{0}, \overline{s}_1\})$), hence the condition is also equivalent to  the equality $P_{\id}=B$.
\end{proof}

\subsubsection{Simple non-split case}\label{sec-unitary-gps}

Let $V$ be an $n$-dimensional $\FF_{p^{2}}$-vector space endowed with a non-degenerate Hermitian form $\psi\colon V\times V\to \FF_{p^{2}}$. We will assume that there is an $\FF_{p^{2}}$-basis \(\Bcal\) of $V$ where the Hermitian form $\psi$ is given by the matrix
\begin{equation}\label{unitaryJmat}
    J\colonequals \left( 
    \begin{matrix}
        && 1 \\
        & \iddots& \\
        1&&
    \end{matrix}
    \right).
\end{equation}
Let $G=\U(V,\psi)$ be the unitary group of $(V,\psi)$, defined by
\begin{equation}
G(R) = \{f\in \GL_{\FF_{p^{2}}}(V\otimes_{\FF_p} R) \mid  \psi_R(f(x),f(y))=\psi_R(x,y), \ \forall x,y\in V\otimes_{\FF_p} R \}
\end{equation}
for any $\FF_p$-algebra $R$. One has an identification $G_{\FF_{p^{2}}}\simeq \GL(V_{\FF_{p^{2}}})$, given as follows: For any $\FF_{p^{2}}$-algebra $R$, we have an $\FF_{p^{2}}$-algebra isomorphism $\FF_{p^{2}}\otimes_{\FF_p} R\to R\times R$, $a\otimes x\mapsto (ax,\sigma(a)x)$. We obtain an isomorphism $V\otimes_{\FF_p}R\to (V\otimes_{\FF_{p^{2}}}R)\oplus (V\otimes_{\FF_{p^{2}}}R)$; any $g\in G(R)$ stabilizes this decomposition, and the restriction to the first summand yields an isomorphism $G_{\FF_{p^{2}}}\simeq \GL(V_{\FF_{p^{2}}})$. Via the basis $\Bcal$ we identify $G_{\FF_{p^{2}}}$ with $\GL_{n,\FF_{p^{2}}}$. The action of $\sigma$ on the set $\GL_n(k)$ is given by $\sigma\cdot A = J \sigma({}^t \!A)^{-1}J$. 

Let $B$ denote the lower triangular Borel subgroup of $G_k$ and $T\subseteq B$ the diagonal torus. One sees immediately that $B$ and $T$ are defined over $\FF_p$. Identify again $X^*(T)=\ZZ^n$. Let $(r,s)$ be positive integers such that $r+s=n$, and let $\mu\colon \GG_{\mathrm{m},k} \to G_{k}$ be the cocharacter $t\mapsto \diag(t\mathbb{1}_{r},\mathbb{1}_s)$. Let $\Zcal_{\mu}=(G,P,Q,L,M)$ be the associated zip datum. Note that $P$ is not defined over $\FF_p$ in this case, unless $r=s$. One has $\Delta\setminus I =\{\alpha_r\}$. 

\begin{lemma} \ 
\begin{enumerate}
    \item The canonical parabolic of $w=\id$ coincides with $P$.
    \item For $w=s_{\alpha_r}$, the type of the canonical parabolic $P_w$ is given by
    \begin{equation*}
        I_w=I\setminus\{1,r-1,r+1,n-1\}.
    \end{equation*}
\end{enumerate}
\end{lemma}
\begin{proof}
We first examine the case $w=\id$. 
We see that
\begin{equation*}
\varphi_{\id}(\alpha_l)=\begin{cases}
    \alpha_{r - l} & \quad \textrm{if} \ 1\leq l <r \\
        \alpha_{n + r - l} & \quad \textrm{if} \ r< l \leq n-1.
\end{cases}
\end{equation*}
We deduce $\varphi_{\id}(I)=I$, hence the canonical parabolic of $w=\id$ coincides with $P$. Next, assume $w=s_{\alpha_r}$, the unique element of ${}^I W$ of length $1$. We find:
\begin{equation*}
\varphi_{w,I}(\alpha_l)=\begin{cases}
    \alpha_{r - l} & \quad \textrm{if} \ 1< l <r \\
        \alpha_{n + r - l} & \quad \textrm{if} \ r< l < n-1\\
      0 & \quad \textrm{if} \ l \in \{1,n-1\}. 
\end{cases}
\end{equation*}
From this, we see that $\varphi_{w}$ stabilizes the set $I\setminus \{1,r-1,r+1,n-1\}$.
Therefore
\begin{equation*}
  I_w = I\cap \varphi_w(I)\cap \varphi_w^2(I) =I\setminus\{1,r-1,r+1,n-1\}. \qedhere
\end{equation*}
\end{proof}

\begin{corollary}\label{cor-len1uninert}
    The closure of the one-dimensional stratum is not smooth, except when $r=s=1$.
\end{corollary}

\subsubsection{General non-split case}
\label{ssec-weilrestunit}
For an integer $d\geq 1$, denote by $(V,\psi)$ the Hermitian space over $\FF_{p^{2d}}$ given by the matrix $J$ of \eqref{unitaryJmat}, with respect to the quadratic extension $\FF_{p^{2d}}/\FF_{p^d}$. Write $G_0\colonequals \U_{\FF_{p^d}}(V,\psi)$ for the unitary group over $\FF_{p^d}$ attached to $(V,\psi)$. We consider the case when 
\begin{equation}
    G=\Res_{\FF_{p^{d}}\mid\FF_p}(G_0).
\end{equation}
By the same identification as in \S \ref{sec-unitary-gps}, we have an isomorphism $(G_0)_{\FF_{p^{2d}}}\simeq \GL_{n,\FF_{p^{2d}}}$. Fix an embedding $\tau_1\colon \FF_{p^d}\to k$ and define inductively $\tau_{i+1}\colonequals \sigma \circ \tau_i$. We obtain again an identification
\begin{equation}\label{Weyl-ident}
    G_k \simeq \prod_{i=1}^d \GL_{n,k}
\end{equation}
where the action of $\sigma$ on $G(k)$ corresponds via this isomorphism to the action sending a $d$-tuple $(x_1, \dots, x_d)$ to
\begin{equation}
    (J {}^t \sigma(x_d)^{-1} J, \sigma(x_1), \dots, \sigma(x_{d-1}))
\end{equation}
Note that the identification \eqref{Weyl-ident} is not symmetric, in the sense that the first factor of the product is singled out, and this choice is reflected in the action of $\sigma$. Let $B\subseteq G$ (resp.\@ $T$) denote the product of $d$ copies of the lower triangular Borel (resp.\@ diagonal torus) of $\GL_n$. It is clear that the Borel pair $(B,T)$ is defined over $\FF_p$. Consider the cocharacter $\mu$ of $G$ given by a $d$-tuple of numbers $(r_1, \dots, r_d)$, similarly to \eqref{mui-def}.

The type $I$ of the Hodge parabolic $P$ can be written as the disjoint union $I=\bigsqcup_{j=1}^d\Delta_j\setminus \{r_j\}$, where we identify $\Delta_j$ with the subset $\ZZ/n\ZZ\setminus \{0\}$. Moreover, note that this formula also holds when $r_j\in \{0,n\}$, in which case the component in the $j$-th factor is simply $\Delta_j$. For each $j\in\{1, \dots,d\}$ such that $0<r_j<n$, denote by $\alpha_j$ the non-compact root $\alpha_{j,r_j}=e_{j,r_j}-e_{j,r_j+1}$. Moreover, for such $j$, denote by $\Xcal_j$ the length one stratum $\Xcal_w$ corresponding to $w=s_{\alpha_j}$.

Let $P_{\id}$ denote the canonical parabolic of the identity element $\id\in {}^I W$. We may decompose $P_{\id}$ as a product $P_{\id}=P_{\id,1}\times \dots \times P_{\id,d}$, where $P_{\id,i}$ is a standard parabolic subgroup of $\GL_n$. We write $I_{\id,i}\subseteq \Delta_i$ for the type of $P_{\id,i}$. Let $j\in \{1, \dots, d\}$ be an integer such that $0 < r_j < n$. The canonical parabolic of $s_{\alpha_j}$ will be denoted by $P_j$, and we write $P_j = P_{j, 1}\times \dots \times P_{j,d}$ where $P_{j, i}$ is a standard parabolic subgroup of $\GL_n$. Write $I_{j, i} \subseteq \Delta_i$ for the type of $P_{j, i}$. We always identify $\Delta_j$ with a subset of $\ZZ/n\ZZ$. The following proposition gives the types of all canonical parabolic subgroups of length zero and one. For $i, l\in \{1, \dots, d\}$, we put:
\begin{equation}
    \epsilon_{i, l} = \begin{cases}
        -1 & \quad \textrm{if} \quad  l < i \\
        1 & \quad \textrm{if} \quad  l \geq i
    \end{cases}
\end{equation}
Moreover, if $i, l$ are only defined modulo $d$, we define $\epsilon_{i, l}$ as the element $\epsilon_{i_0, l_0}$ where $i\equiv i_0 \pmod{d}$, $l\equiv l_0 \pmod{d}$ and $i_0, l_0\in \{1, \dots, d\}$.

\begin{proposition}  \ 
\begin{enumerate}
    \item For $1\leq i \leq d$, we have
    \begin{equation}
        I_{\id,i}=I_i\setminus \left(\left\{ \sum_{l = N + i - 1}^{d+i-1} \epsilon_{i, l} r_{l} \ \relmiddle| \  1\leq N \leq d  \right\} \cup \left\{ \sum_{l = i}^{N+i-1}  \epsilon_{i, l}r_{l} \ \relmiddle| \  1\leq N \leq d  \right\}\right)
    \end{equation} 
    \item Let $j\in \{1, \dots, d\}$ such that $0<r_j<n$. Let $i=j+1+h$ (modulo $d$) with $0\leq h \leq d-1$. Then, we have:
    \begin{equation}
        I_{j,i}=I_{\id,i}\setminus \left\{ \sum_{l = j+1}^{h+j}\epsilon_{i, l} r_{l} \pm 1, \quad \sum_{l = h + j + 1}^{j + d}\epsilon_{i, l} r_{l} \pm 1\right\}.
    \end{equation}
\end{enumerate}
\end{proposition}

In the case $h=0$ (corresponding to $i=j+1$), the sum $\sum_{l = j + 1}^{h + j}\epsilon_{i, l} r_{l}$ appearing in the left-hand side element in the bracket is empty, so we assign the value $0$ to it (and the above element is then simply $\pm 1$).

\begin{theorem}\label{thm-weilresunitl1}
Let $j\in\{1, \dots, d\}$ such that $0<r_j<n$. The following are equivalent:
\begin{enumerate}
    \item The Zariski closure of $\Xcal_j$ is smooth.
    \item The Zariski closure of $\Xcal_j$ is normal.
    \item We have $P_{j}=P_{\id}$.
    \item We have $P_{j, i} = P_{\id,i}$ for some (equivalently, for all) $i\in \{1, \dots, d\}$.
    \item The elements $\sum_{l = j+1}^{d} r_{l} \pm 1$ and $\sum_{l = 1}^{j} r_{l} \pm 1$ (taken modulo $n$) lie in the subset
    \begin{equation}
    \left\{ \sum_{l = N}^d r_{l} \ \relmiddle| \  1 \leq N \leq d  \right\} \cup \left\{ \sum_{l = 1}^N r_{l} \ \relmiddle| \  1\leq N \leq d  \right\} \cup \{0\}.
    \end{equation}
\end{enumerate}
\end{theorem}

We now take $d=2$. There are at most two non-compact roots, corresponding to the factors $i$ such that $r_i\notin \{0,n\}$. In this case, we have the following result:

\begin{proposition} \ Let $d=2$. Then
\begin{enumerate}
\item The Zariski closure of $\Xcal_1$ is smooth if and only if \(r = (1, 1)\), \((1, 2)\), \((2, 1)\), \((n - 1, n - 2)\), \((n - 2, n - 1)\),  or \((n - 1, n - 1)\).
\item The Zariski closure of $\Xcal_2$ is smooth if and only if \(r = (1, n - 2)\), \((1, n - 1)\), \((2, n - 1)\), \((n - 2, 1)\), \((n - 1, 1)\), or \((n - 1, 2)\).
\end{enumerate}
\end{proposition}

As a consequence, we see that for $d=2$ and $n\geq 5$, at least one of the length one strata is non-smooth.

\subsection{Type \texorpdfstring{$\Bsf$}{B}, \texorpdfstring{$\Csf$}{C}, \texorpdfstring{$\Dsf$}{D}, \texorpdfstring{$\Esf_6$}{E} and \texorpdfstring{$\Esf_7$}{E}}
\begin{theorem}[cf.\@ Theorem \ref{theorem: thmB intro}]\label{theorem: one-dim smooth locus bn cn dn}
Let $(G,\mu)$ be a minuscule cocharacter datum of type $\Bsf$, $\Csf$, $\Dsf$, $\Esf_6$ or $\Esf_7$.  Let $\Xcal$ denote the corresponding stack of $G$-zips of type $\mu$ (\S\ref{section: the stack of gzips}). For any one-dimensional zip stratum $\Xcal_w$, we have that $\overline{\Xcal}^{\text{sm}}_w=\Xcal_w$. 
\end{theorem}
\begin{proof}
The proofs in all cases are completely analogous to each other, and are explained case-by-case in \S\ref{section: 1dim type Bn} $\sim$ \S\ref{section: exceptional dim one} below. 
\end{proof}

\subsubsection{Type \texorpdfstring{$\Bsf$}{B}}\label{section: 1dim type Bn}
Let $\Phi=\sqcup_{i=1}^d \Phi_0$ be a union of root systems arising from the Weil restriction of a root system $\Phi_0$ of $\Bsf_n$-type. Let $\Delta_0=\{\alpha_1, \ldots,\alpha_n\}$ be a set of simple roots for $\Phi_0$ as in \cite[Planche II]{bourbaki.chapter4.5.6.} (note that our $n$ is Bourbaki's $l$). This makes $\Delta=\sqcup_{i=1}^d \Delta_0$ a set of simple roots for $\Phi$. For $S\subsetneq \{1, \ldots, d\}$, let 
\[
I_i\coloneqq \begin{cases}
    \Delta_0, & i\in S, \\
    \Delta_0 \setminus \{\alpha_1\}, & i \notin S,
\end{cases}
\]
and set $I\coloneqq \sqcup_{i=1}^d I_i\subseteq \Delta$. The Frobenius acts on $\Phi$ by translation to the right. 

Define the set
\begin{equation}
    \Gamma\colonequals \bigsqcup_{i=1}^d \Delta_0 \setminus \{\alpha_1\} \subseteq\Delta.
\end{equation}
It is clear that $\Gamma$ is stable under the action of $\sigma$, $-w_0$, $-w_{0,I}$ and is dense in $\Delta$ (recall the definition in \S\ref{section: dense subsets of delta}). Since $\varphi_{\id}(\alpha)=z^{-1}\sigma(\alpha)=w_0 \sigma(w_{0,I}\alpha)$ for each $\alpha\in \Delta$, we immediately deduce that $I_{\id}$ must contain $\Gamma$. Since $\Gamma$ is dense in $\Delta$, so is $I_{\id}$. Since any Hodge-type cocharacter datum $(G,\mu)$ of $\Bsf_n$-type gives rise to some $(\Phi,\Delta,I)$ as above with $I$ the type of $\mu$, \Cref{cor-Borel} gives \Cref{theorem: one-dim smooth locus bn cn dn} in the $\Bsf_n$-case. 

\subsubsection{Type \texorpdfstring{$\Csf$}{C}}\label{sec-sympgroups}
Let $\Phi=\sqcup_{i=1}^d \Phi_0$ be a union of root systems arising from the Weil restriction of a root system $\Phi_0$ of $\Csf_n$-type. Let $\Delta_0=\{\alpha_1, \ldots,\alpha_n\}$ be a set of simple roots for $\Phi_0$ as in \cite[Planche III]{bourbaki.chapter4.5.6.} (i.e.\@ $\alpha_n$ is the long root). This gives $\Delta=\sqcup_{i=1}^d \Delta_0$ as a set of simple roots for $\Phi$. For $S\subsetneq \{1, \ldots,d\}$, let 
\[
I_i\coloneqq \begin{cases}
    \Delta_0, & i\in S, \\
    \Delta_0\setminus \{\alpha_n\}, & i\notin S,
\end{cases}
\]
and set $I\coloneqq \sqcup_{i=1}^d I_i\subseteq \Delta$. 

Again, Frobenius acts on $\Phi$ by translation to the right. Hence, the set
\begin{equation}
    \Gamma\colonequals \bigsqcup_{i=1}^d \Delta_0 \setminus \{\alpha_n\} \subseteq\Delta
\end{equation}
is stable under the action of $\sigma$, $-w_0$, $-w_{0,I}$ and is dense in $\Delta$. Therefore, as in the $\Bsf_n$-case, we find from \Cref{cor-Borel} that \Cref{theorem: one-dim smooth locus bn cn dn} holds in the $\Csf_n$-case.

\subsubsection{Type \texorpdfstring{$\Dsf$}{D}}\label{section: 1dim type Dn}
Let $\Phi=\sqcup_{i=1}^d \Phi_0$ be a union of root systems arising from the Weil restriction of a root system $\Phi_0$ of $\Dsf_n$-type. Let $\Delta_0=\{\alpha_1, \ldots,\alpha_n\}$ be a set of simple roots for $\Phi_0$ as in \cite[Planche IV]{bourbaki.chapter4.5.6.}. This gives $\Delta=\sqcup_{i=1}^d \Delta_0$ as a set of simple roots for $\Phi$. 
On each simple factor, $-w_0$ acts trivially if $n$ is even, and it permutes $\alpha_{n-1}$ and $\alpha_n$ if $n$ is odd. 

The action of Frobenius on each simple factor is either trivial, or it permutes the two roots $\alpha_{n-1}$ and $\alpha_n$, or, in the case of \(\Dsf_4\) is given by a triality automorphism which permutes \(\alpha_1, \alpha_3, \alpha_4\), while fixing \(\alpha_2\). Which one occurs depends on whether the associated orthogonal group is split over $\FF_p$ or not. Hence, the action of Frobenius on $\Phi$ either (1) only translates to right, or (2) translates to the right and permutes $\alpha_{n-1}$ and $\alpha_n$ in each simple factor, for \(n > 4\), and permutes \(\alpha_1, \alpha_3, \alpha_4\), for \(n = 4\).

The minuscule roots are $\alpha_1, \alpha_{n-1}$ and $\alpha_n$. Thus, for $S\subseteq \{1, \ldots,d\}$, let $I_i=\Delta_0$, for $i\in S$, and for $i\notin S$ let 
\[
    I_i\in \Big\{ \Delta_0\setminus \{\alpha_1\}, \Delta_0\setminus \{\alpha_{n-1}\}, \Delta_0\setminus \{\alpha_n\}\Big\}.
\]
Set $I\coloneqq \sqcup_{i=1}^d I_i$. Consider the set
\begin{equation*}
    \Gamma\colonequals \bigsqcup_{i=1}^d \{\alpha_2, \ldots, \alpha_{n-2}\} \subseteq\Delta.
\end{equation*}
One sees easily that $\Gamma$ is dense in $\Delta$ and is stable under the action of $-w_0$, $-w_{0,I}$ and the Frobenius element $\sigma$. By \Cref{cor-Borel} we conclude that \Cref{theorem: one-dim smooth locus bn cn dn} holds for type $\Dsf_n$.

\subsubsection{Type \texorpdfstring{\(\Esf_6\)}{E6} and \texorpdfstring{\(\Esf_7\)}{E7}}\label{section: exceptional dim one}
Among the exceptional simple groups, the only ones which admit minuscule cocharacters are of type \(\Esf_6\) and \(\Esf_7\). They have the following Dynkin diagrams, using the numbering of \cite{bourbaki.chapter4.5.6.} for the simple roots, where the node labeled by an integer \(i\) corresponds to a simple root \(\alpha_i\).
\begin{center}
    \begin{tabular}{c c}
         \(\Esf_6\): \dynkin[label,ordering=Bourbaki]E6, & with minuscule roots \(\alpha_1, \alpha_6\), \\
         \(\Esf_7\): \dynkin[label,ordering=Bourbaki]E7, & with minuscule root \(\alpha_7\). 
    \end{tabular}
\end{center}
For the classification of minuscule roots, see \cite[Ex.~\S4.15]{bourbaki.chapter4.5.6.}.
The diagram for \(\Esf_6\) has a single involutive non-trivial automorphism, given by the reflection of the diagram around the nodes  of \(\alpha_2, \alpha_4\). Let us denote by \(\tau\) the corresponding element of \(\Sfr_6\), which is given by \([6, 2, 5, 4, 3, 1]\). The diagram for \(\Esf_7\) has no non-trivial automorphisms. This implies the following. 
\begin{lemma} \label{e6e7prep} In the case of \(\Esf_6\), we have the following.
\begin{enumerate}[(i)]
    \item There are two isomorphism classes of simple adjoint reductive \(\Fp\)-groups of type \(\Esf_6\). In the non-split case, the action of the Frobenius \(\sigma\) on the simple roots is given by \(\alpha_i \mapsto \alpha_{\tau(i)}\).
    \item The longest element \(w_0\) of the Weyl group of \(\Esf_6\) acts as \(w_0(\alpha_i) = - \alpha_{\tau(i)}\).
    \item The actions of \(-w_0\) and \(\sigma\) preserve the set \(\Gamma_6 \coloneqq \Delta_0 \setminus \{\alpha_1, \alpha_6\}\), for \(\Delta_0\) the set of simple roots.
\end{enumerate}
In the case of \(\Esf_7\), we have the following.
\begin{enumerate}[(i)]
    \item There is only one isomorphism class of simple reductive \(\Fp\)-groups of type \(\Esf_7\).
    \item The longest element \(w_0\) of the Weyl group of \(\Esf_7\) acts as \(w_0(\alpha_i) = - \alpha_{i}\).
\end{enumerate}
\end{lemma}
\begin{proof}
The automorphisms of the Dynkin diagram correspond to the outer automorphisms of the split isomorphism class of the \(\Fp\)-groups of a given simple adjoint type. This proves the statement about the isomorphism classes and, in the case of \(\Esf_6\), the action of \(\sigma\). For the description of the longest element, one can look at \cite[Planches V et VI]{bourbaki.chapter4.5.6.}.
\end{proof}
Notice the following.
\begin{enumerate}[(i)]
    \item In the case of \(\Esf_6\), for \(I = \Delta_0 \setminus\{\alpha_i\}, i = 1, 6\), \((W_I, I)\) is isomorphic to the Weyl group of \(\Dsf_5 \cong \Esf_5\). In particular, the action of \(-w_{0, I}\) will preserve the set \(\Gamma_6\) defined in \Cref{e6e7prep} (possibly permuting \(\alpha_2, \alpha_3\) and \(\alpha_5\)).
    \item In the case of \(\Esf_7\), for \(I = \Gamma_7 \coloneqq \Delta_0 \setminus \{\alpha_7\}\), \((W_I, I)\) is isomorphic to the Weyl group of \(\Esf_6\). Thus, the longest element \(-w_{0,I}\) will preserve \(\Gamma_7\).
\end{enumerate}
Consider \(\Phi = \sqcup_{i=1}^d \Phi_0\) a union of root systems arising from the Weil restriction of an irreducible root system \(\Phi_0\) of type either \(\Esf_6\) or \(\Esf_7\). In both cases, the type \(I_{\id}\) of the canonical parabolic for \(w = \id\) contains a dense subset of the simple roots \(\Delta\):
\begin{enumerate}[(i)]
    \item if \(\Phi_0\) is of type \(\Esf_6\), we can take \(\sqcup_{i=1}^d \Gamma_6\),
    \item if \(\Phi_0\) is of type \(\Esf_7\), we can take \(\sqcup_{i = 1}^d\Gamma_7\). 
\end{enumerate}
From \Cref{cor-Borel}.\eqref{dim1-cor2}, we deduce \Cref{theorem: one-dim smooth locus bn cn dn} for $\Esf_6$ and $\Esf_7$.

\begin{remark}\label{remark: exceptional periods}
At the moment, there is no known zip period map for exceptional Shimura varieties. The construction of such period maps is an ongoing project by the authors. Upon completion, \Cref{theorem: one-dim smooth locus bn cn dn} will imply the analogous result for the EO-stratification of the corresponding Shimura varieties.
\end{remark}

\section{Odd orthogonal groups} \label{sec-orthogonal}
We assume  $p>2$ throughout this section.
\subsection{Group theory}\label{section orthogonal group theory}
Let $n\geq 1$ be an integer and let $V$ be an $\FF_p$-vector space of dimension $2n+1$, endowed with a non-degenerate symmetric form $\psi\colon V\times V\to \FF_p$. We assume that there is an $\FF_p$-basis $\Bcal$ of $V$ such that $\psi$ is given in $\Bcal$ by the matrix (of size $2n+1$)
\begin{equation}\label{Jmat}
    J_n\colonequals \left(
\begin{matrix}
    &&1 \\ &\iddots&\\ 1&&
\end{matrix}
    \right).
\end{equation}
Let $G$ denote the $\FF_p$-reductive group such that, for any $\FF_p$-algebra $R$, we have
\begin{equation}
    G(R)=\{g\in \SL_R(V\otimes_{\FF_p} R) \ | \ \psi_R(gx,gy)=\psi_R(x,y), \ \forall x,y\in V\otimes_{\FF_p}R \}.
\end{equation}
Via our choice of basis $\Bcal$, we identify $G$ with a subgroup of $\SL_{2n+1}$. The subgroup of lower triangular matrices in $G$ is a Borel subgroup $B$, and the subgroup of diagonal matrices in $G$ is a maximal torus $T$. We have
\begin{equation}
    T(R)=\{\diag(x_1, \dots, x_n,1,x_{n}^{-1},\dots,x_1^{-1}) \ | \ x_i\in R^\times\}
\end{equation}
for all $\FF_p$-algebras $R$. Identify $X^*(T)=\ZZ^n$ so that $\lambda=(\lambda_1, \dots,\lambda_n)\in \ZZ^n$ corresponds to the character
\begin{equation}
    \diag(x_1, \dots, x_n,1,x_{n}^{-1},\dots,x_1^{-1}) \mapsto \prod_{i=1}^n x_i^{\lambda_i}.
\end{equation}
Denote by $(e_1, \dots,e_n)$ the standard basis of $\ZZ^n$. The positive $T$-roots are given by $\Phi^+=\{e_i\pm e_j \ | \ 1\leq i<j\leq n \}\cup \{e_i \ | \ 1\leq i\leq n\}$. Moreover, the set of simple $T$-roots is $\Delta=\{\alpha_1, \dots, \alpha_{n}\}$ where
\begin{equation}
\alpha_i=e_i-e_{i+1} \ \textrm{for} \ 1\leq i\leq n-1, \ \textrm{and} \ \alpha_n=e_n.
\end{equation}
The Weyl group $W=W(G,T)$ can be identified with the set of permutations $w\in \Sfr_{2n+1}$ satisfying $w(i)+w(2n+2-i)=2n+2$. 
Note that $w(n+1)=n+1$ and $w$ is uniquely determined by the values $a_i=w(i)$ for $1\leq i \leq n$. For $1\leq i \leq n$, write simply $s_i\colonequals s_{\alpha_i}$. For $w\in W$, the permutation matrix attached to $w$ lies in $G$ if and only if $\det(w)=1$. If $\det(w)=-1$, we may change the $(n+1,n+1)$-coefficient to $-1$ to obtain an element of $G$. This gives an injective group homomorphism $W\to G$. We will implicitly view $W$ inside $G$ via this map when writing down matrices throughout this section.

\subsection{Bruhat order on \texorpdfstring{$W$}{the Weyl group}} \label{sec-bruhat-weyl}

Say that a pair $(i,j)$ is \emph{admissible} for $w$ if it satisfies the following:
\begin{enumerate}
    \item $1\leq i<j\leq 2n+1$,
    \item $w(i)>w(j)$,
    \item There is no $i<k<j$, $k\neq n+1$, such that $w(j)<w(k)<w(i)$.
\end{enumerate}
Write $M_w\in \GL_{2n+1}$ for the permutation matrix of $w$. The pair $(i,j)$ is admissible for $w$ if and only if the submatrix of $M_w$ with corners $(i,w(i))$ and $(j,w(j))$ has only zero coefficients, except for these two corners and possibly the coefficient $(n+1,n+1)$. For $w\in W$, define three sets $\Ecal^1_w$, $\Ecal_w^2$, $\Ecal_w^3$, as follows:
\begin{align}
    \Ecal_w^1 &\colonequals \{(i,j) \ | \ \textrm{admissible for $w$, and } 1\leq i<j\leq n \}, \\
    \Ecal_w^2 &\colonequals \{(i,j) \ | \ \textrm{admissible for $w$, }i\leq n<j  \textrm{ and } w(i),w(j)\leq n \}, \\
    \Ecal_w^3 &\colonequals \{(i,j) \ | \ \textrm{admissible for $w$, } i\leq n 
 \textrm{ and } j=2n+2-i \}.
\end{align}
It is easy to see that $\Ecal_w^1$, $\Ecal_w^2$, $\Ecal_w^3$, are pairwise disjoint. We set
\begin{equation}\label{Ecalw-def}
    \Ecal_w=\Ecal_w^1\sqcup \Ecal_w^2 \sqcup \Ecal_w^3.
\end{equation}
For $(i,j)\in \Ecal_w$, define a positive root $\gamma'_w(i,j)\in \Phi^+$ as follows:
\begin{equation}
    \gamma'_w(i,j)\colonequals \begin{cases}
        e_i-e_j, & \textrm{if } (i,j)\in \Ecal_w^1,\\
         e_i+e_{2n+2-j}, & \textrm{if }  (i,j)\in \Ecal_w^2, \\
          e_i, & \textrm{if }  (i,j)\in \Ecal_w^3.
    \end{cases}
\end{equation}
Recall the set $E_w\subseteq \Phi^+$ defined in \Cref{def-Ew}. The following result is similar to \cite[Lemma 5.2.1]{Koskivirta-vanishing-Siegel}:

\begin{proposition} \label{prop-Ew}
The map $\gamma'_w\colon \Ecal_w\to E_w$ is a bijection.
\end{proposition}
Composing with the bijection $\varepsilon_w \colon E_w \to \Gamma_{\emptyset}(w)$ of \S\ref{subsec-lownb}, we obtain a bijection $\gamma_w\colonequals\varepsilon_w \circ \gamma'_w \colon \Ecal_w \to \Gamma_{\emptyset}(w)$.

\subsection{Zip strata}
Consider the cocharacter $\mu\colon \GG_{\textrm{m},k}\to G_k$ defined by $t\mapsto \diag(t,1, \dots,1,t^{-1})$. Denote by $\Zcal$ the corresponding zip datum, and by $\Xcal$ the attached stack of $G$-zips. The parabolic subgroup $P$ is the standard parabolic of type $I=\Delta\setminus \{\alpha_1\}$. The set ${}^I W$ contains exactly $2n$ elements, denoted by $\{x_0,x_1, \dots, x_{2n-1}\}$, given as follows:
\begin{align*}
x_{2n-1}&\colonequals 1, \\
    x_{2n-1-i}&\colonequals s_1 \dots s_{i} \quad \textrm{for} \ 1 \leq i \leq n, \\
  x_{i}&\colonequals s_1 \dots s_{n-1} s_n s_{n-1} \dots s_{i+1} \quad \textrm{for} \ 0\leq i \leq n-2. 
\end{align*}
When $w=x_i$, we simply write $\gamma_i$ for the bijection \(\gamma_w\colon \Ecal_w \to \Gamma_{\emptyset}(w)\) of \S\ref{sec-bruhat-weyl}. The length of $x_i$ is $\ell(x_i)=2n-1-i$. In matrix form, we have $x_{2n-1}=\mathbb{1}_{2n+1}$, and
\begin{equation}\label{equation: the xjs}
\begin{aligned}
    x_{0}=\left(
    \begin{matrix}
        &&1 \\ &\mathbb{1}_{2n-1}& \\ 1&&
    \end{matrix}
        \right),  \quad
    x_{j}&=\begin{pNiceArray}{cccccccccc}
    &&&&&&&1&& \\ \hline 
    \Block[borders={bottom, right}]{2-2}{\mathbb{1}_j} &&&&&&& \\
    &&&&&&&&& \\
    &&0&\Block[borders={bottom, right,left,top}]{4-4}{\mathbb{1}_{2(n-j)-1}} &&&&&& \\
    &&&&&&&&& \\
    &&&&&&&&& \\
    &&&&&&&0&& \\
    &&&&&&&&\Block[borders={top, left}]{2-2}{\mathbb{1}_j} & \\
    &&&&&&&&& \\ \hline
    &&1&&&&&&& 
    \end{pNiceArray}, \\
        x_{2n-1-j}&=\begin{pNiceArray}{cccccccccc}
    &&1&&&&&&& \\ \hline 
    \Block[borders={bottom, right}]{2-2}{\mathbb{1}_j} &&&&&&& \\
    &&&&&&&&& \\
    &&0&\Block[borders={bottom, right,left,top}]{4-4}{\mathbb{1}_{2(n-j)-1}} &&&&&& \\
    &&&&&&&&& \\
    &&&&&&&&& \\
    &&&&&&&0&& \\
    &&&&&&&&\Block[borders={top, left}]{2-2}{\mathbb{1}_j} & \\
    &&&&&&&&& \\ \hline
    &&&&&&&1&& 
    \end{pNiceArray}
\end{aligned}
\end{equation}
for all $1\leq j\leq n-1$. Set $z\colonequals w_{0,I}w_0$, which is the element $z=[(2n+1), \ 2, \dots, n]$. Note that $z^{-1}=z$. Write simply $E$ for the zip group $E_{\Zcal}$. The $E$-orbits in $G$ are parametrized by ${}^IW$ via the map $w\mapsto G_w\colonequals E\cdot (wz^{-1})$ (see \S\ref{sec-param}). Let $\Xcal_j\subseteq \Xcal$ denote the zip stratum $\Xcal_w$ for $w=x_j$. The Zariski closure of $\Xcal_j$ is $\bigcup_{i\geq j} \Xcal_i$.

\subsection{Canonical parabolic subgroups}
For $0\leq j\leq 2n-1$, write $P_j\colonequals P_{x_j}$ and $I_j \colonequals I_{x_j}$.

\begin{proposition}\label{prop-can-orth}
For any $0\leq j\leq n-1$, we have $I_j=I_{2n-1-j}=\Delta\setminus \{\alpha_1, \dots,\alpha_{j+1}\}$.
\end{proposition}
In particular, the type of the canonical parabolic starts at $I_0=I$, then $I_j$ decreases by one at each step until $I_{n-1}=I_n=\emptyset$, and then increases by one at each step until $I_{2n-1}=I$.

\begin{proof}
Consider the map $\varphi_{w,\Delta} \colon \Delta^*\to I^*$ of \eqref{phistar}. Since the Galois action is trivial, $I_{0}=I$. Set $w=x_j$, for $1\leq j \leq n-1$. 

From \eqref{equation: the xjs}, we find that 
\begin{align*}
    \varphi_{w,\Delta}(\alpha_i)&=\alpha_{i+1}, \quad \textrm{for} \ i=2, \dots, j-1, \\
    \varphi_{w,\Delta}(\alpha_j)&=\varphi_{w,\Delta}(\alpha_{j+1})=0, \\
    \varphi_{w,\Delta}(\alpha_i)&=\alpha_i, \quad \textrm{for} \ i=j+2, \dots, n.
\end{align*}
The same relations hold for $w=x_{2n-1-j}$, $1\leq j \leq n-1$. The result follows easily.
\end{proof}

\begin{corollary}\label{cor-smoothlocsmallortstrat}
For $0\leq j \leq n-1$, the smooth locus of $\overline{\Xcal}_{2n-1-j}$ is exactly $\Xcal_{2n-1-j}$.
\end{corollary}

\begin{proof}
By \Cref{proposition criteria for smooth locus to be open stratum} this follows from \Cref{prop-can-orth}.
%
\end{proof}

In the remainder of the section, we examine the smooth locus of $\overline{\Xcal}_j$ for $0\leq j \leq n-1$.

\subsection{Canonical flag stratum of \texorpdfstring{$x_j$}{x j}}
Let $0\leq j\leq n-1$ and consider $\Fcal^{(P_j)}$, the intermediate flag space attached to $P_j$. Denote by $\Zcal_j$ the zip datum $\Zcal_j\colonequals (G, P_j, Q_j, L_j, M_j)$ introduced in \ref{sec-fine-stratif}. The stack \(\Fcal^{(P_j)}\) carries a \(\Zcal_j\)-zip stratification $(\Fcal_w^{(P_j)})_{w\in {}^{I_j}W}$ parametrized by ${}^{I_j}W$. Since $I_j=\Delta\setminus \{\alpha_1, \dots,\alpha_{j+1}\}$, the set ${}^{I_j} W$ is the set of elements $w\in W$ satisfying the condition
\begin{equation}\label{IW-orth}
w^{-1}(j+2)<w^{-1}(j+3)<\dots < w^{-1}(n)<n+1.
\end{equation}
By \Cref{prop-Pw}, the natural projection $\pi_{P_j}\colon \Fcal^{(P_j)}\to \Xcal$ restricts to an isomorphism $\pi\colon \Fcal_{x_{j}}^{(P_j)} \to \Xcal_{j}$ and a proper, birational map $\pi_{P_j}\colon \overline{\Fcal}_{x_{j}}^{(P_j)}\to \overline{\Xcal}_{j}$. By \Cref{strict-Bruh-Pw}, the stratum $\Fcal_{x_j}^{(P_j)}$ is Bruhat.

\subsection{Lower neighbors of \texorpdfstring{$x_j$}{x j} in \texorpdfstring{${}^{I_j}W$}{I j W}} 
\label{j-equals-r-sec}

\begin{proposition}\label{prop-Gamma}
For $0\leq j \leq 2n-1$, we have $\Gamma_{I_j}(x_j)={}^{I_j}W\cap \Gamma_\emptyset(x_j)$.
\end{proposition}

For $j=n$ and $j=n-1$ the result is trivial, since $I_j=\emptyset$. For the remaining cases, we distinguish between $0\leq j< n-1$ and $n<j\leq 2n-1$.

\subsubsection{The case \texorpdfstring{$0\leq j < n-1$}{of j less than n-1}}
In this case, consider the set $\Dcal_{x_j}\colonequals \{(i,2n+1-j) \ | \ 1\leq i \leq j \}$, of cardinality $j$. We have 
\begin{align*}
    \Ecal_{x_j}^1 &=\{(j+1,j+2), \dots, (j+1,n)\} \\
    \Ecal_{x_j}^2 &=\Dcal_{x_j} \ \cup \ \{(j+2,2n+1-j), \dots, (n,2n+1-j)\} \\
    \Ecal_{x_j}^3 &= \emptyset.
\end{align*}
and thus $|\Ecal_{x_j}|=| \Gamma_\emptyset(x_j)|=2n-j-2$. For $1\leq i \leq j$, define $w_i^{(j)}\colonequals \gamma_j(i, 2n+1-j)$. From this, one checks easily that ${}^{I_j}W\cap \Gamma_\emptyset(x_j)$ consists of exactly $j+1$ elements given by
\begin{equation}
{}^{I_j}W\cap \Gamma_\emptyset(x_j) = \{x_{j+1}\} \cup \gamma_j(\Dcal_{x_j}).
\end{equation}
This set is contained in $\Gamma_{I_j}(x_j)$ by \Cref{lemma-low-nei}\eqref{lemma-low-nei-item1}. We show the reverse inclusion.

Let $\Hcal_j\subseteq \Ecal_{x_j}$ be the subset of elements corresponding via $\gamma_j\colon \Ecal_{x_j}\to \Gamma_{\emptyset}(x_j)$ to either $x_{j+1}$ or to elements that do not lie in ${}^{I_j} W$, i.e.
\begin{equation} \label{set-eq}
 \Hcal_j\colonequals \Ecal^1_{x_j}  \ \sqcup \ \{(j+2,2n+1-j), \dots, (n,2n+1-j)\}.
\end{equation}
Let $(u,v)$ be an element in this union, and let $w\colonequals \gamma_j(u,v)$ be the corresponding lower neighbor of $x_j$ in $W$. Set $K_1\coloneqq 1$, and for any integer $r\geq 2$, define an $r\times r$-matrix $K_r$ by
\begin{equation}
    K_r\colonequals \left(
\begin{matrix}
   \mathbb{0}_{1,r-1} & 1 \\ \mathbb{1}_{r-1}& \mathbb{0}_{r-1,1}
\end{matrix}
    \right).
\end{equation}
For any element $w\in \gamma_j(\Hcal_j)$, the coefficient $1$ in the first line of $M_w$ lies in one of the columns $j+2, \dots, 2n-j$. Similarly, the coefficient $1$ in column $j+1$ lies in one of the rows $j+2, \dots, 2n-j$. Permuting the columns and lines numbered $j+2, \dots, 2n-j$ in the matrix of $w$ by a permutation (in $W_{I_j}$) corresponds to multiplying $w$ by elements of $W_{I_j}$ on both sides. Therefore, we can multiply on the right with a matrix in $W_{I_j}$ to place the coefficient in the first row in column $j+2$. Similarly, we can multiply on the left with a matrix in $W_{I_j}$ to place the coefficient in the column $j+1$ in row $j+2$. If necessary, we multiply (on either side) the central block of size $2(n-j)-3$ by some matrix in the subgroup $W_{I_{j+1}}\subseteq W_{I_j}$ to change it to the identity matrix. By this discussion, we can find $x,y\in W_{I_j}$ such that 
\begin{equation}
    xwy = \left(
    \begin{matrix}
        K_{j+2} & & \\
 & \mathbb{1}_{2(n-j)-3} & \\
 && {}^t K_{j+2}
    \end{matrix}
    \right) \in {}^{I_j} W^{I_j}.
\end{equation}
We deduce the following corollary.
\begin{lemma}\label{lem-PjQj}
For all $w\in \gamma_j(\Hcal_j)$, the element $wz^{-1}$ lies in the $P_j\times Q_j$-orbit of $x_{j+1}z^{-1}$.
\end{lemma}
\begin{proof}
The above shows that all elements in $\gamma_j(\Hcal_j)$ lie in the same $L_j\times L_j$-orbit (i.e.\@ that of $x_{j+1}$). Note that $z$ normalizes $L_j$, i.e.\@ $zL_jz^{-1}=L_j$. This implies that all elements $wz^{-1}$ for $w\in \gamma_j(\Hcal_j)$ lie in the same $L_j\times L_j$-orbit as $x_{j+1}z^{-1}$. In particular, they lie in the same $P_j\times Q_j$-orbit.
\end{proof}

\begin{lemma}\label{lem-biject}
The map $w\mapsto P_jwz^{-1}Q_j$ induces a bijection between $\Gamma_{I_j}(x_j)$ and the set of $P_j\times Q_j$-orbits contained in the Zariski closure of $P_jx_jz^{-1}Q_j$ with codimension one.
\end{lemma}

\begin{proof}
Since the stratum $\Fcal^{(P_j)}_{x_j}$ is Bruhat, 
the Zariski closure $\overline{\Fcal}^{(P_j)}_{x_j}$ is a union of Bruhat strata, and so is the complement $\overline{\Fcal}^{(P_j)}_{x_j} \setminus \Fcal^{(P_j)}_{x_j}$. The elements of $\Gamma_{I_j}(x_j)$ are in bijection with the irreducible components of this complement. Hence it is clear that any Bruhat stratum contained in $\overline{\Fcal}^{(P_j)}_{x_j}$ with codimension one must contain some $\Fcal^{(P_j)}_w$ for $w\in \Gamma_{I_j}(x_j)$. It contains only one such stratum because each irreducible component of $\overline{\Fcal}^{(P_j)}_{x_j} \setminus \Fcal^{(P_j)}_{x_j}$ is also a union of Bruhat strata (as $P_j\times Q_j$ is itself irreducible).
\end{proof}

Now, let $w\in \Gamma_{I_j}(x_j)$ be an element. We need to show that $w\in \Gamma_{\emptyset}(x_j)$, i.e.\@ that $w\leq x_j$. By \Cref{lemma-low-nei}, we have $\Gamma_{I_j}(x_j)\subseteq \pi_{B,P_j}(\Gamma_{\emptyset}(x_j))$, hence $w=\pi_{B,P_j}(w_1)$ for some $w_1\in \Gamma_{\emptyset}(x_j)$. If $w_1\in {}^{I_j}W$, we have $w=\pi_{B,P_j}(w_1)=w_1$. Otherwise, \Cref{lem-PjQj} shows that $w_1z^{-1}$ lies in the $P_j\times Q_j$-orbit of $x_{j+1}z^{-1}$. We deduce that $x_{j+1}z^{-1}$ and $wz^{-1}$ also lie in the same $P_j\times Q_j$-orbit, hence $w=x_{j+1}$ by \Cref{lem-biject}. This terminates the proof of \Cref{prop-Gamma} for $0\leq j< n-1$.

Before moving on to the remaining cases, we introduce the following notation: For $n\leq j \leq 2n-1$ and $1\leq i \leq 2n-1-j$, set $w_i^{(j)}\colonequals \gamma_j(i, 2n-j)$. For $1\leq j \leq n-1$ and $1\leq i \leq j$, set $w_i^{(j)}\colonequals \gamma_j(i, 2n+1-j)$. One sees easily that $w_i^{(j)}\in \Gamma_{\emptyset}(x_j)$ in each case.

\subsubsection{The case \texorpdfstring{$n< j \leq 2n-1$}{j at least n+1}} \label{subsubsec-case1}

Using \Cref{prop-Ew}, one checks immediately that the elements $w_i^{(j)}$ for $1\leq i \leq 2n-1-j$ are exactly the elements of $\Gamma_{\emptyset}(x_j)$. All of them are contained in ${}^{I_j}W$, i.e.\@ we have $\Gamma_{\emptyset}(x_j)\subseteq {}^{I_j}W$. From \Cref{cor-same-neighbors}, we deduce the statement of \Cref{prop-Gamma} in the case $n< j \leq 2n-1$.

\subsection{Lower neighbors of \texorpdfstring{$x_r$}{x r} in \texorpdfstring{${}^{I_j} W$}{I j W}}

Let $j,r$ be integers such that $0\leq j \leq r < 2n-1-j$ (note that this implies $j\leq n-1$). We extend the result of \Cref{prop-Gamma} and determine the set of lower neighbors of $x_r$ in ${}^{I_j}W$ with respect to the partial order $\preccurlyeq_{I_j}$. Recall that we always have ${}^{I_j} W \cap \Gamma_{\emptyset}(x_r) \subseteq \Gamma_{I_j}(x_r)$. By the assumption on $j,r$, we have $I_r\subseteq I_j$, hence ${}^{I_j}W\subseteq {}^{I_r}W$. One checks easily that for each $1\leq i \leq j$, the elements $w_i^{(r)}\in {}^{I_r}W$ actually lie in the subset ${}^{I_j}W$. This implies readily
\begin{equation} \label{list-low1}
    {}^{I_j} W \cap \Gamma_{\emptyset}(x_r)=\{x_{r+1}\}\cup \{w_i^{(r)} \ | \ 1\leq i \leq j\}
\end{equation}
which is a set with $j+1$ elements. In the next section, we show the following:
\begin{proposition}\label{prop-xr}
Under the above assumptions, one has $\Gamma_{I_j}(x_r)={}^{I_j} W \cap \Gamma_{\emptyset}(x_r)$.
\end{proposition}

\subsection{Images of strata}
Keep the same assumptions on $j,r$. Since $P_r\subseteq P_j$, we have the natural projection $\pi_{P_r,P_j}\colon \Fcal^{(P_r)}\to \Fcal^{(P_j)}$ (see \S\ref{subsubsec-flagspace-def}). We want to understand the images of certain $\Zcal_{P_r}$-zip strata of $\Fcal^{(P_r)}$ via this mapping. Denote simply by $\Zcal_j$ the zip datum $\Zcal_{P_j}$ defined in \S\ref{sec-fine-stratif}, attached to $P_j$. Write $E_j$ for the attached zip group $E_{\Zcal_{P_j}}$. By definition of the \(\Zcal_j\)-zip stratification, the $E_j$-orbits on $G_k$ correspond bijectively to the \(\Zcal_j\)-zip strata of $\Fcal^{(P_j)}$.

For $r\leq n-1$, we have $\Gamma_{I_r}(x_r)=\{x_{r+1}\}\cup \{w_i^{(r)} \ | \ 1\leq i \leq r\}$ by \Cref{prop-Gamma} applied to $j=r$. In the case $r \geq n$, the same proposition gives $\Gamma_{I_r}(x_r)={}^{I_r}W \cap \Gamma_{\emptyset}(x_r) =\{w_i^{(r)} \ | \ 1\leq i \leq 2n-1-r\}$, by \Cref{prop-Gamma} (applied to $j=r$). In both cases, the elements which belong to ${}^{I_j}W$ are $x_{r+1}$ and $w_i^{(r)}$ for $1\leq i\leq j$, by \eqref{list-low1}. For such elements $w$, we have $\pi_{P_r,P_j}(\Fcal_w^{(P_r)})=\Fcal_w^{(P_j)}$. We now consider the remaining elements.

\begin{lemma}\label{lem-Ej-orbits}
For $i> j$ and for any $b\in B$, the element $b w_i^{(r)}z^{-1}$ lies in the same $E_{j}$-orbit as $x_{2n-i} z^{-1}$.
\end{lemma}

\begin{proof}
The matrix of $w_i^{(r)}$ has the form
\begin{equation}\label{wi-shape}
 w_i^{(r)}= \left(
\begin{matrix}
    K_i && \\
    &A& \\
    && {}^t K_i
\end{matrix}
    \right)
\end{equation}
for a certain invertible matrix $A$ of size $2(n-i)+1$. Any element of the form $b'x_{2n-i}z^{-1}$, for $b'\in B$, lies in the same $E_j$-orbit as $x_{2n-i}z^{-1}$ using $\pi_{B,P_j}(\Fcal_w^{(B)})=\Fcal_w^{(P_j)}$ for $w=x_{2n-i}\in {}^{I_j}W$. Hence, it is enough to show that for any $b\in B$, there exists $y\in L_j$ and $b'\in B$ such that 
\begin{equation}\label{find-y}
bw^{(r)}_iz^{-1} = y b'x_{2n-i}z^{-1} \varphi(y)^{-1}.    
\end{equation}
We will choose $y$ in the form
\begin{equation}\label{eq-y}
y= \left(
\begin{matrix}
    \mathbb{1}_i && \\
    &Y& \\
    && \mathbb{1}_i
\end{matrix}
    \right)
\end{equation}
for $Y$ invertible of size $2(n-i)+1$ which is orthogonal with respect to the matrix $J_{n-i}$ defined in \eqref{Jmat}. Since $\varphi(y)$ commutes with $z$, equation \eqref{find-y} amounts to $bw^{(r)}_i = y b'x_{2n-i} \varphi(y)^{-1}$. Consider the central block of size $(2(n-i)+1)\times (2(n-i)+1)$ in the matrix $bw^{(r)}_i$. Since the Lang torsor map is surjective, we may write this central block as $Y\varphi(Y)^{-1}$ for some orthogonal invertible matrix $Y$ of size $2(n-i)+1$. Define $y$ as in \eqref{eq-y} for this choice of $Y$. Set $b'\colonequals y^{-1}bw^{(r)}_i \varphi(y) x_{2n-i}^{-1}$. It suffices to show that $b'$ lies in $B$. The central block of size $(2(n-i)+1)\times (2(n-i)+1)$ in $b'$ is the identity matrix of size $2(n-i)+1$. Moreover, the upper left and lower right $i\times i$-blocks in $x_{2n-i}$ and $w^{(r)}_i$ coincide, hence we see also that the corresponding blocks in $b'$ are lower triangular. Hence $b'\in B$ and the result follows.
\end{proof}

The following is a direct and more explicit reformulation of \Cref{lem-Ej-orbits} in terms of \(\Zcal_j\)-zip strata. Let $i$ be an integer such that $1\leq i\leq r$ if $r< n$, or such that $1\leq i\leq 2n-1-r$ if $r\geq n$.

\begin{proposition} \label{prop-imag-strata-rj}
Set $w=w_i^{(r)}$.
\begin{enumerate}
    \item If $i \leq j$, then $w$ lies in ${}^{I_j}W$ and we have $\pi_{P_r,P_j}(\Fcal_w^{(P_r)})=\Fcal_w^{(P_j)}$.
    \item If $i>j$, we have $\pi_{P_r,P_j}(\Fcal_w^{(P_r)}) = \Fcal^{(P_j)}_{x_{2n-i}}$.
\end{enumerate}
\end{proposition}

We consider the special case $j=0$. Since $I_0=I$, we have $\Fcal^{(P_0)}=\Xcal$. We deduce:
\begin{corollary}\label{cor-imag}
Let $1\leq i \leq \min(r,2n-1-r)$ and set $w=w^{(r)}_i$. The image of $\Fcal_{w}^{(P_r)}$ by the map $\pi_{P_r}\colon \Fcal^{(P_r)}\to \Xcal$ is exactly $\Xcal_{2n-i}$.
\end{corollary}

We can now prove Proposition \ref{prop-xr}. We know that $\Gamma_{I_j}(x_r)\subseteq \pi_{P_r,P_j}(\Gamma_{I_r}(x_r))$. By \Cref{prop-imag-strata-rj}, the elements corresponding to $i>j$ do not give any lower neighbors of $x_r$. The result follows.

\subsection{Canonical cover}
Let $0\leq j \leq n-1$. Define a subset 
\begin{equation}
    \Ycal_j\colonequals \bigcup_{i=j}^{2n-1-j} \Xcal_i
\end{equation}
Hence $\Ycal_0=\Xcal$ and $\Ycal_{n-1}=\Xcal_{n-1}\cup \Xcal_n$. Since the zip stratification of $\Xcal$ is linear, it is clear that $\Ycal_j$ is locally closed in $\Xcal$. It is an $x_j$-open substack, using the terminology of \S\ref{can-cov-sec}. Recall that $\pi_{P_j}$ induces a proper, birational map $\pi_{P_j}\colon \overline{\Fcal}_{x_j}^{(P_j)}\to \overline{\Xcal}_j$. Set $\Ycal^{(P_j)}_j\colonequals \bigcup_{i=j}^{2n-1-j} \Fcal^{(P_j)}_{x_i}$, which is the subset \eqref{UP0} for $\Ucal=\Ycal_j$. 

\begin{proposition} \label{propYi}
The $x_j$-open subset $\Ycal_j$ admits a separating canonical cover.
\end{proposition}

\begin{proof}
We need to show that $\Ycal^{(P_j)}_j = \overline{\Fcal}_{x_j}^{(P_j)} \cap \pi_{P_j}^{-1}(\Ycal_j)$. We clearly have $\Ycal^{(P_j)}_j \subseteq  \overline{\Fcal}_{x_j}^{(P_j)} \cap \pi_{P_j}^{-1}(\Ycal_j)$. For the reverse inclusion, we show by induction on $r\in \{j, \dots, 2n-j-1\}$ the following statement: if $w\in {}^{I_j}W$ with $w\preccurlyeq x_j$ is an element of length $\ell(x_r)$ different from $x_r$, then the image of $\Fcal^{(P_j)}_w$ by $\pi_{P_j}$ is contained in $\overline{\Xcal}_{2n-j}$. For $r=j$ there is nothing to prove. Assume it is true for some $r$, and consider an element $y\in {}^{I_j}W$ of length $\ell(x_{r+1})$ such that $y \preccurlyeq_{I_j} x_j$ and $y\neq x_{r+1}$. Using the chain property of the twisted order (\S\ref{subsubsec: closure relations}), we can find an element $y'\in {}^{I_j}W$ of length $\ell(x_r)$ such that $y\preccurlyeq_{I_j} y' \preccurlyeq_{I_j} x_j$. If $y'\neq x_r$, then the induction hypothesis implies that the image of $\Fcal^{(P_j)}_{y'}$ by $\pi_{P_j}$ is contained in $\overline{\Xcal}_{2n-j}$. Therefore the same holds for the element $y$. Otherwise $y'=x_r$ and hence $y$ is a lower neighbor of $x_r$ in ${}^{I_j}W$. By \Cref{prop-xr}, we have $y=w_i^{(r)}$ for some $i$ such that $1\leq i \leq j$ (hence also $1\leq i\leq \min(r,2n-1-r)$). By \Cref{cor-imag}, the image of $\Fcal^{(P_j)}_y$ in $\Xcal$ is $\Xcal_{2n-i}$. Since $2n-i\geq 2n-j$, we have $\Xcal_{2n-i}\subseteq \overline{\Xcal}_{2n-j}$. This terminates the proof.
\end{proof}

\begin{corollary}\label{coro yj is normal} For any $0\leq j\leq n-1$, the map $\pi_{P_j}\colon \Ycal^{(P_j)}_j \to \Ycal_j$ is an isomorphism. In particular, $\Ycal_j$ is normal and Cohen--Macaulay.
\end{corollary}
\begin{proof}
By \Cref{prop-can-orth}, $\Ycal_j$ is $x_j$-bounded. By \Cref{propYi}, $\Ycal_j$ admits a separating canonical cover. Hence, the result follows from \Cref{bounded-implies-normal}.
\end{proof}

\subsection{Reduced generalized Hasse invariants}
To each character $\lambda\in X^*(L)$, one can attach a line bundle $\Lcal(\lambda)$ on $\Xcal$ (\cite[\S3.1]{koskgold}). Denote by $\eta_\omega$ the character 
\begin{equation}
    \eta_\omega = (-1,0, \ldots,0).
\end{equation}
We call (somewhat abusively) the attached line bundle $\omega\colonequals \Lcal(\eta_{\omega})$ the Hodge line bundle on the stack $\Xcal$. Recall that generalized Hasse invariants exist on all Ekedahl--Oort strata (see \cite[Corollary 6.2.3]{Goldring-Koskivirta-zip-flags}). We deduce that for each $0\leq j<2n-1$, there exists an integer $m_j\geq 1$ and a section
\begin{equation}
    \Ha_j\in H^0(\overline{\Xcal}_j,\omega^{m_j})
\end{equation}
such that the (set-theoretic) vanishing locus of $\Ha_j$ is exactly $\overline{\Xcal}_{j+1}$. Assume now that $0\leq j\leq n-1$. We have shown that $\Ycal_j\subseteq \overline{\Xcal}_j$ is normal. In particular, $\overline{\Xcal}_j$ is normal in codimension one, so we may speak of the multiplicity of a section along $\overline{\Xcal}_{j+1}$. Our next result is that there exists a reduced Hasse invariant on $\overline{\Xcal}_j$, i.e. a section $\Ha_j$ on $\overline{\Xcal}_j$ which vanishes along $\overline{\Xcal}_{j+1}$ with multiplicity one.

We first construct a section on $\Ycal_j^{(P_j)}$ that cuts out the complement of $\Fcal_{x_j}^{(P_j)}$. Since $\Fcal^{(P_j)}_{x_{j}}$ is Bruhat, we have
\begin{equation}
    \Fcal^{(P_j)}_{x_{j}} = \left[ E'_{P_j}\backslash Z_j\right], \quad \textrm{where} \quad Z_j \colonequals P_j (x_j z^{-1}) Q_j.
\end{equation}
Note that $Z_j$ contains a unique open dense $B\times {}^z B$-orbit. Denote by $\Lambda_d$ the $d\times d$ anti-diagonal matrix $\Lambda_d=(\delta_{k,d+1-r})_{k,r}$. The following lemma is an easy verification.
\begin{lemma}
    The unique $B\times {}^zB$-orbit which is open in $Z_j$ is $B\widetilde{x}_j B z^{-1}$ where
    \begin{equation}
        \widetilde{x}_j \colonequals x_j w_{0,I_j} = \begin{pNiceArray}{cccccccccc}
&&&&&&&1&& \\ \hline 
\Block[borders={bottom, right}]{2-2}{\mathbb{1}_j} &&&&&&& \\
&&&&&&&&& \\
&&0&\Block[borders={bottom, right,left,top}]{4-4}{\Lambda_{2(n-j)-1}} &&&&&& \\
&&&&&&&&& \\
&&&&&&&&& \\
&&&&&&&0&& \\
&&&&&&&&\Block[borders={top, left}]{2-2}{\mathbb{1}_j} & \\
&&&&&&&&& \\ \hline
&&1&&&&&&& 
\end{pNiceArray}
    \end{equation}
\end{lemma}
We recall below Chevalley's formula. 
For $w\in W$, let $E_w$ be the set defined in \Cref{def-Ew}. For $(\lambda,\nu)\in X^*(T)\times X^*(T)$, one attaches a line bundle $\Vcal_{\Sbt}(\lambda,\nu)$ on the Schubert stack $\Sbt=\left[ B\backslash G/B\right]$. The Bruhat decomposition of $G$ yields a stratification $(\Sbt_w)_{w\in W}$ where $\Sbt_w=\left[ B\backslash BwB/B\right]$. A section of $\Vcal_{\Sbt}(\lambda,\nu)$ over $\Sbt_w$ can be viewed as a regular map $f\colon BwB\to \AA^1$ satisfying $f(agb^{-1})=\lambda(a)\nu(b)f(g)$ for all $a,b\in B$ and all $g\in BwB$.

\begin{theorem}[{\cite[§1, p. 654]{brion}}]\label{brion}
Let $w \in W$. One has the following.
\begin{enumerate}
\item \label{brion1} $H^0\left(\Sbt_w,\Vcal_{\Sbt}(\lambda,\nu)\right)\neq 0\Longleftrightarrow \nu = -w^{-1} \lambda$.
\item \label{brion2} $\dim_k H^0\left(\Sbt_w,\Vcal_{\Sbt}(\lambda,-w^{-1} \lambda) \right)=1$.
\item \label{brion3} For any nonzero $f\in H^0\left(\Sbt_w,\Vcal_{\Sbt}(\lambda,-w^{-1} \lambda) \right)$, one has $\divi(f)=-\sum_{\alpha \in E_w} \langle \lambda, w\alpha^\vee \rangle \overline{\Sbt}_{w s_\alpha}$.
\end{enumerate}
\end{theorem}
For $\eta\in X^*(T)$, let $h_\eta$ be a nonzero section over $\Sbt_{\widetilde{x}_j}$ of $\Vcal_{\Sbt}(\eta,-\widetilde{x}_j^{-1}\eta)$. View $h_\eta$ as a regular function $B\widetilde{x}_jB\to \AA^1$. Write $h_\eta'$ for the shifted function $h_\eta'\colon B\widetilde{x}_jBz^{-1}\to \AA^1$, $x\mapsto h_\eta(xz)$. As we want to construct a section of a power of $\omega$, we choose $\eta$ such that there exists an integer $m_j\geq 1$ such that
\begin{equation}
h_\eta'(tg\varphi(t)^{-1})=\eta_\omega^{m_j}(t)h_\eta'(g), \quad t\in T, g\in B\widetilde{x}_jBz^{-1}.
\end{equation}
One sees immediately that this is equivalent to the condition:
\begin{equation}\label{eta-eq}
    \eta-p z\widetilde{x}_j^{-1}\eta =m_j \eta_\omega.
\end{equation}
Since the map $\eta \mapsto \eta-p z\widetilde{x}_j^{-1}\eta$ defines an isomorphism of $X^*(T)_{\QQ}$ (because it induces the identity map on $X^*(T)_{\FF_p}$), it is possible to find such $\eta$ and $m_j\geq 1$. From \eqref{eta-eq}, we find
\begin{equation}
    \left(
\begin{array}{ccccc|ccc}
    1 & p & & & &&& \\
     &1 & -p &  & &&& \\
     & & \ddots & \ddots & &&& \\
     & & &1&-p &&& \\
     p&&&&1 &&& \\ \hline
      &&&&&1+p&& \\
            &&&&&&\ddots& \\
                  &&&&&&&1+p
\end{array}
    \right) \eta = m_j \eta_\omega
\end{equation}
where the above matrix has size $n\times n$ and the upper left block has size $(j+1)\times (j+1)$. It follows that
\begin{equation}
    \eta = - \frac{m_j}{p^{j+1}-1} \left(-1, p^j, p^{j-1}, \dots, p, 0, \dots, 0 \right).
\end{equation}
Hence $\widetilde{x}_j^{-1}\eta=- \frac{m_j}{p^{j+1}-1} (p^j, \dots,p,1,0, \dots,0)$. Therefore, $\langle\eta, \widetilde{x}_j\alpha^\vee \rangle=0$ for all $\alpha \in I_j$. This implies that $h_\eta'$ extends to a regular non-vanishing function on the $P_j\times Q_j$-orbit $Z_j$, by \Cref{brion}\eqref{brion3}. By Rosenlicht's theorem (\cite[\S 1]{Knop-Kraft-Vust-G-variety}), $h_\eta'$ is automatically $P_j\times Q_j$-equivariant.

Moreover, the stratum $\Fcal^{(P_j)}_{x_{j+1}}$ corresponds to the $P_j\times Q_j$-orbit of $x_{j+1}z^{-1}$. The unique $B\times {}^zB$-orbit which is open in $P_j x_{j+1} z^{-1} Q_j$ is $B\widetilde{x}_j s_\beta B z^{-1}$ where $\beta=e_{j+1}-e_{j+2}$ for $0\leq j \leq n-2$ and $\beta=e_n$ for $j=n-1$. Therefore, for $0\leq j \leq n-2$, the multiplicity of $h_\eta'$ along $P_jx_{j+1}z^{-1} Q_j$ is given by 

\begin{equation}
 -\langle \eta, \widetilde{x}_j \beta^\vee \rangle = -\langle \widetilde{x}_j^{-1}\eta, \beta^\vee \rangle = \frac{m_j}{p^{j+1}-1}.  
\end{equation}
We take $m_j=p^{j+1}-1$, thus $\eta\in X^*(T)$ and $-\langle \eta, \widetilde{x}_j \alpha^\vee \rangle=1$. Since it satisfies \eqref{eta-eq} and is $P_j\times Q_j$-equivariant, the function $h'_\eta$ given in this way is a section of the line bundle $\omega^{m_j}$. It extends to $\Ycal_j^{(P_j)}$ by normality. Using the isomorphism $\pi_{P_j}\colon \Ycal_j^{(P_j)}\to \Ycal_j$, we set $\Ha_j\colonequals \pi_{P_j *}(h'_\eta)$, which is a section of $\omega^{m_j}$ over $\Ycal_j$. We have shown:
\begin{proposition}\label{prop-mult1}
For $0\leq j \leq n-2$, there exists a section $\Ha_j\in H^0(\Ycal_j,\omega^{p^{j+1}-1})$ whose vanishing locus is exactly the complement of $\Xcal_{x_j}$, with multiplicity one.
\end{proposition}

\begin{theorem}\label{thm-Xj-normal-Hj}
For any $0\leq j \leq n-1$, we have:
\begin{enumerate}
    \item The stratum $\overline{\Xcal}_j$ is normal and a local complete intersection.
    \item\label{item: haj extension} The section $\Ha_j$ extends to a section $\Ha_j\in H^0(\overline{\Xcal}_j,\omega^{p^{j+1}-1})$.
\end{enumerate}
\end{theorem}

\begin{proof}
We prove the result by induction on $j$. For $j=0$, there is nothing to prove. Assume $\overline{\Xcal}_j$ is normal and a local complete intersection for some $0\leq j <n-1$, and that $\Ha_j$ extends to $\overline{\Xcal}_j$. The substack $\overline{\Xcal}_{j+1}$ coincides with the scheme-theoretic vanishing locus of $\Ha_j$. It follows that $\overline{\Xcal}_{j+1}$ is an effective Cartier divisor of $\overline{\Xcal}_j$, hence is a local complete intersection. Moreover $\overline{\Xcal}_{j+1}$ is regular in codimension one since $\Ycal_{j+1}$ is normal. By Serre's criterion, we deduce that $\overline{\Xcal}_{j+1}$ is normal. Then, by normality of $\overline{\Xcal}_{j+1}$, the section $\Ha_{j+1}$ extends to $\overline{\Xcal}_{j+1}$. This terminates the induction.
\end{proof}
\subsection{Application to cycle classes and cohomological vanishing}

For a closed substack $\Ycal\subseteq \Xcal$, let $[\Ycal]$ denote its class in the Chow ring of $\Xcal$. For a line bundle $\Lscr$ on $\Xcal$ let $[\Lscr]$ denote the class of the corresponding divisor.
\begin{corollary}\label{corollary: cycle class formula}
For any $0\leq j \leq n-1$, the cycle class of $\overline{\Xcal}_j$ is given by
\[
(p-1)(p^2-1)\cdots (p^j-1)[\omega]^j
\]
\end{corollary}
\begin{proof}
The proof of \cite[Theorem 15.1]{geer.katsura.stratification.of.k3} applies verbatim. Explicitly, let $\iota_{j-1}\colon \overline{\Xcal}_{j-1}\hookrightarrow \Xcal$ denote the inclusion. Since $\overline{\Xcal}_j$ is cut out in $\overline{\Xcal}_{j-1}$ by $\Ha_{j-1}\in H^0(\overline{\Xcal}_{j-1},\omega^{p^j-1})$, we find that its class in $\overline{\Xcal}_{j-1}$ is given by
\[
\iota_{j-1}^*\Big((p^{j}-1)[\omega]\Big)\cdot [\overline{\Xcal}_{j-1}].
\]
Applying the projection formula (\cite[Proposition 2.5(c)]{fulton.intersection.theory}), its class in $\Xcal$ is given by
\[
\iota_{j-1,*}\Big( \iota_{j-1}^*(p^{j}-1)[\omega]\cdot [\overline{\Xcal}_{j-1}] \Big) = (p^j-1)[\omega]\cdot \iota_{j-1,*}([\overline{\Xcal}_{j-1}]).
\]
The statement now follows by induction on $j$.
\end{proof}
Let $S$ be the special fiber of a Hodge-type Shimura variety of simple $\Bsf_n$-type as in \S\ref{section: intro g-zips to shimura varieties}. Let $S^{\tor}$ and $S^{\text{min}}$ denote a smooth toroidal compactification, respectively the minimal compactification. Let $\pi\colon S^{\tor}\to S^{\text{min}}$ denote the natural map. 
For an automorphic vector bundle $\Vcal(\eta)$ on $S$, let $\Vcal^{\text{sub}}(\eta)$ denote the sub-canonical extension of $\Vcal(\eta)$ to $S^{\tor}$. Let $\zeta^{\tor}\colon S^{\tor}\to \Xcal$ be the zip period map \S\ref{section: intro g-zips to shimura varieties} and let $\overline{S}^{\tor}_j\coloneqq (\zeta^{\tor})^{-1}(\overline{\Xcal}_j)$ denote the Ekedahl--Oort strata of $S^{\tor}$.
\begin{corollary}\label{corollary: coh vanishing}
Assume that $\eta\in X_{+,L}^*(T)$ and that $R^i\pi_{*}\Vcal^{\sub}(\eta)=0$ for all $i>0$. Then, for $m\gg0$, $H^i(\overline{S}^{\tor}_j,\Vcal^{\sub}(\eta)\otimes \omega^m)=0$ for all $j=0, \ldots, n-1$, and all $i>0$.
\end{corollary}
\begin{proof}
By \Cref{thm-Xj-normal-Hj}, this follows from \cite[Lemma 7.3.1]{koskgold}.
\end{proof}
\begin{remark}
By \cite{lan}, the assumption in the statement holds for all $\eta\in X_{+,L}^*(T)$ in the PEL-case, and in the Hodge-type case it holds whenever $\eta$ is a positive $\QQ$-multiple of the Hodge character, by an unpublished result of Stroh (cf.\@ \cite[Remark 7.1.3]{koskgold}).
\end{remark}

\subsection{Smooth locus}

\begin{proposition}
The subset $\Ycal_j$ is smooth.
\end{proposition}

\begin{proof}
The statement is clear for $j=0$, so we may assume $1\leq j\leq n-1$. Consider the isomorphism $\pi_{P_{j-1}}\colon \Ycal_{j-1}^{(P_{j-1})}\to \Ycal_{j-1}$. Denote by $\Ycal^{(P_{j-1})}_{j}$ the preimage of $\Ycal_j\subseteq \Ycal_{j-1}$ by this isomorphism, i.e.
\begin{equation}
    \Ycal^{(P_{j-1})}_{j} \colonequals \bigcup_{i=j}^{2n-1-j} \Fcal^{(P_{j-1})}_{x_i}.
\end{equation}
By construction, $\Ycal^{(P_{j-1})}_{j}$ is locally closed in $\Ycal_{j-1}^{(P_{j-1})}$. Moreover, the isomorphism $\pi_{P_{j-1}}\colon \Ycal^{(P_{j-1})}_{j-1}\to \Ycal_{j-1}$ restricts to an isomorphism $\pi_{P_{j-1}} \colon \Ycal^{(P_{j-1})}_{j}\to \Ycal_j$. Hence, it is equivalent to show that the subset
\begin{equation}
  Y_j^{(P_{j-1})} \colonequals \bigcup_{i=j}^{2n-1-j} E_{P_{j-1}}\cdot (x_iz^{-1}) 
\end{equation}
is smooth. For this, we show that it is open in the $P_{j-1}\times Q_{j-1}$-orbit of $x_{j}z^{-1}$. We know that $E_{P_{j-1}}\cdot (x_{j-1}z^{-1})$ coincides with a $P_{j-1}\times Q_{j-1}$-orbit. Since the Zariski closure of $E_{P_{j-1}}\cdot (x_jz^{-1}) $ is an irreducible component of the complement of $E_{P_{j-1}}\cdot (x_{j-1}z^{-1})$ in its Zariski closure, it is stable under $P_{j-1}\times Q_{j-1}$. Furthermore, we have
\begin{equation}\label{eq-same-PjQj}
    x_{2n-1-j}z^{-1} = x_jz^{-1} s_\alpha, \quad \textrm{where} \ \alpha=e_{j+1}.
\end{equation}
We have $s_\alpha\in L_{j-1}$. The two extremal points (i.e. $x_jz^{-1}$ and $x_{2n-1-j}z^{-1}$) lie in the $P_{j-1}\times Q_{j-1}$-orbit of $x_j z^{-1}$, by \eqref{eq-same-PjQj}. Therefore $Y_j^{(P_{j-1})}\cap P_{j-1}x_j z^{-1} Q_{j-1}$ is a $E_{P_{j-1}}$-stable locally closed subset containing the extremal points $x_jz^{-1}$ and $x_{2n-1-j}z^{-1}$, hence it contains all points of $Y_j^{(P_{j-1})}$. This shows that $Y_j^{(P_{j-1})}$ is contained in the $P_{j-1}\times Q_{j-1}$-orbit of $x_j z^{-1}$. 
Finally, it remains to show that $Y_j^{(P_{j-1})}$ is open inside this $P_{j-1}\times Q_{j-1}$-orbit. It suffices to show that it is open in the Zariski closure of $E_{P_{j-1}}\cdot (x_jz^{-1})$. Since $\Ycal_{j-1}$ is open in $\overline{\Fcal}_{x_{j-1}}^{(P_{j-1})}$, the set $\Ycal_{j-1}\cap \overline{\Fcal}_{x_j}^{(P_{j-1})}$ is open in $\overline{\Fcal}_{x_j}^{(P_{j-1})}$. But $\Ycal_{j-1}\cap \overline{\Fcal}_{x_j}^{(P_{j-1})}=\bigcup_{i=j}^{2n-j} \Fcal^{(P_{j-1})}_{x_i}$ contains $\Ycal^{(P_{j-1})}_j$ as an open subset. The result follows.
\end{proof}

\begin{proposition}\label{prop-orthsmoothlocbigstrat}
The subset $\Ycal_j$ is exactly the smooth locus of $\overline{\Xcal}_j$.
\end{proposition}

\begin{proof}
We may assume $j\geq 1$ since the result holds trivially for $j=0$. Let $\Ha_{j-1}\in H^0(\overline{\Xcal}_{j-1},\omega^{p^j-1})$ be the generalized Hasse invariant afforded by \Cref{thm-Xj-normal-Hj}\eqref{item: haj extension}. The vanishing locus of $\Ha_{j-1}$ is the complement of $\Xcal_{j-1}$. Consider the isomorphism $\pi_{P_{j-1}}\colon \Ycal^{(P_{j-1})}_{j-1}\to \Ycal_{j-1}\subseteq \overline{\Xcal}_{j-1}$. Write $C_{j-1}\colonequals P_{j-1} x_{j-1}z^{-1} Q_{j-1}$. The pullback of $\Ha_{j-1}$ to $\Ycal^{(P_{j-1})}_{j-1}$ arises from the stack
\begin{equation}
    \left[ (P_{j-1}\times Q_{j-1}) \backslash C_{j-1}\right]
\end{equation}
It corresponds to a $P_{j-1}\times Q_{j-1}$-equivariant function $h_{j-1}\colon C_{j-1} \to \AA^1$ satisfying the relation
\begin{equation}
    h_{j-1}(ag)=\eta_{j-1}(a)h_{j-1}(g), \quad a\in P_{j-1}, \ g\in C_{j-1},
\end{equation}
where $\eta_{j-1}=(1,-p^{j-1},-p^{j-2}, \dots, -p, 0, \dots, 0)$. It suffices to show that the order of vanishing of $h_{j-1}$ is strictly greater than $1$ at the point $x_{2n-j}z^{-1}$. We translate the function $h_{j-1}$ to make it $B^+\times B$-equivariant. Define:
\begin{equation}
    \widetilde{h}_{j-1}(x)=h_{j-1}(w_0xz^{-1}).
\end{equation}
This function is regular on $\widetilde{C}_{j-1} \colonequals w_0 C_{j-1} z$ and satisfies $\widetilde{h}_{j-1}(ag)=\widetilde{\eta}_{j-1}(a)\widetilde{h}_{j-1}(g)$ for all $a\in B^+$ and $g\in \widetilde{C}_{j-1}$, where $\widetilde{\eta}_{j-1}=w_0\eta_{j-1}=-\eta_{j-1}$. Set $u_i\colonequals w_0 x_i$ for all $0\leq i \leq 2n-1$. Note that the Zariski closure of $\widetilde{C}_{j-1}$ coincides with the Zariski closure of $B^+u_{j-1}w_{0,I_{j-1}}B$ (since $u_{j-1}w_{0,I_{j-1}}$ is the element of $W$ of minimal length inside $\widetilde{C}_{j-1}$). We need to show that $\ord_{u_{2n-j}}(\widetilde{h}_{j-1})>1$, or equivalently that $\ord_{u_{2n-j} w_{0,I_{j-1}}}(\widetilde{h}_{j-1})>1$, since $\widetilde{h}_{j-1}$ is $P_{j-1}\times Q_{j-1}$-equivariant. Set $U\colonequals R_{\textrm{u}}(B)$ and $U^+\colonequals R_{\textrm{u}}(B^+)$. The function $\widetilde{h}_{j-1}$ coincides (up to nonzero scalar) with the restriction to $\widetilde{C}_{j-1}$ of the function
\begin{equation}
   \gamma_{j-1} \colon u_{j-1}w_{0,I_{j-1}} U^+B \to \AA^1, \quad u_{j-1}w_{0,I_{j-1}} u^+ t u \mapsto (u_{j-1}^{-1}\cdot \widetilde{\eta}_{j-1})(t),
\end{equation}
(where $u\in U$, $t\in T$, $u^+\in U^+$) since both are $B^+\times B$-eigenfunctions for the same character. Since the order of vanishing can only increase through pullback, it suffices to show that $\ord_{u_{2n-j} w_{0,I_{j-1}}}(\gamma_{j-1})>1$. By translation, this order coincides with that of the function
\begin{equation}
    \widetilde{\gamma}_{j-1} \colon U^+B \to \AA^1, \quad u^+ t u \mapsto (u_{j-1}^{-1}\cdot \widetilde{\eta}_{j-1})(t)
\end{equation}
at the point $(u_{j-1}w_{0,I_{j-1}})^{-1}(u_{2n-j}w_{0,I_{j-1}})=u_{j-1}^{-1}u_{2n-j}$. We have $u_{j-1}^{-1}\cdot \widetilde{\eta}_{j-1}=(p^{j-1},p^{j-2}, \dots, p, -1, 0, \dots, 0)$. Moreover
\begin{equation}
   u_{j-1}^{-1}u_{2n-j} = s_{j}\dots s_{n-1}s_n s_{n-1}\dots s_{j}.
\end{equation}
Hence we may directly apply the result of \cite[Theorem 3.8.5]{goldring2024ogusprinciplezipperiod} and find that $\ord_{u_{j-1}^{-1}u_{2n-j}}(\widetilde{\gamma}_{j-1})=2$. 
\end{proof}

\printbibliography 
\end{document}